\documentclass{amsart}
\usepackage[all]{xy}
\usepackage{enumerate}
\usepackage{mathrsfs}
\usepackage{amssymb}
\usepackage{graphicx}

\numberwithin{equation}{subsection}

\newcommand{\nc}{\newcommand}
\nc\rnc{\renewcommand}

\theoremstyle{plain}
\newtheorem{theorem}{Theorem}[section]
\newtheorem{lemma}[theorem]{Lemma}
\newtheorem{corollary}[theorem]{Corollary}
\newtheorem{proposition}[theorem]{Proposition}
\newtheorem{conjecture}[theorem]{Conjecture}
\theoremstyle{definition}
\newtheorem{definition}[theorem]{Definition}

\theoremstyle{remark}
\newtheorem{remark}[theorem]{Remark}
\newtheorem{question}[theorem]{Question}
\newtheorem{problem}[theorem]{Problem}
\newtheorem{claim}{Claim}

\newcommand\Aut{\operatorname{Aut}}
\newcommand\Out{\operatorname{Out}}
\newcommand\Inn{\operatorname{Inn}}
\newcommand\Hom{\operatorname{Hom}}

\newcommand\GL{\operatorname{GL}}

\newcommand\IA{\operatorname{IA}}
\newcommand\IO{\operatorname{IO}}
\newcommand\Ext{\operatorname{Ext}}
\newcommand\id{\operatorname{id}}
\newcommand\pr{\operatorname{pr}}
\newcommand\opegr{\operatorname{gr}}

\newcommand\im{\operatorname{im}}
\newcommand\Sym{\operatorname{Sym}}
\newcommand\Sh{\operatorname{Sh}}

\newcommand\op{\operatorname{op}}
\newcommand\ab{\operatorname{ab}}
\newcommand\sgn{\operatorname{sgn}}
\newcommand\Gr{\operatorname{Gr}}

\newcommand\Z{\mathbb{Z}}
\newcommand\N{\mathbb{N}}
\newcommand\C{\mathbb{C}}

\newcommand\Q{\mathbb{Q}}
\newcommand\repS{\mathbb{S}}

\newcommand\Gp{\mathbf{Gp}}

\newcommand\CCoalg{\mathbf{CCoalg}}
\newcommand\CAlg{\mathbf{CAlg}}
\newcommand\grVect{\mathbf{grVect}}

\newcommand\gpS{\mathfrak{S}}

\newcommand\I{\mathcal{I}}

\newcommand\ti{\tilde}
\newcommand\wti{\widetilde}

\newcommand\wheel{\mathrm{wheel}}
\newcommand\tree{\mathrm{tree}}
\newcommand\tl{\mathrm{tl}}

\newcommand\Prim{\operatorname{Prim}}

\newcommand\red[1]{{\color{red}#1}}
\renewcommand\red{}

\nc\FAb{\mathbf{FAb}}

\nc\xto[1]{{\overset{#1}{\longrightarrow}}}
\nc\yto[1]{{\underset{#1}{\longrightarrow}}}
\nc\xyto[2]{{\overset{#1}{\underset{#2}{\longrightarrow}}}}
\nc\ad{{\operatorname{ad}}}
\nc\even{{\operatorname{even}}}
\nc\ev{{\operatorname{ev}}}
\nc\coev{{\operatorname{coev}}}
\nc\odd{{\operatorname{odd}}}
\nc\half{{\frac12}}
\nc\halfof[1]{{\frac{#1}2}}
\nc\onto\twoheadrightarrow
\nc\projto{\underset{\text{proj}}{\longrightarrow}}
\nc\no[1]{}
\nc\ok{\comm{ok?}}
\nc\ho{{\hat\otimes }}
\nc\plim{\varprojlim}
\nc\np{\newpage}

\nc\bfA{\mathbf{A}} \nc\bbA{\mathbb{A}} \nc\calA{\mathcal{A}}
\nc\bfB{\mathbf{B}} \nc\bbB{\mathbb{B}} \nc\calB{\mathcal{B}}
\nc\bfC{\mathbf{C}} \nc\bbC{\mathbb{C}} \nc\calC{\mathcal{C}}
\nc\bfD{\mathbf{D}} \nc\bbD{\mathbb{D}} \nc\calD{\mathcal{D}}
\nc\bfE{\mathbf{E}} \nc\bbE{\mathbb{E}} \nc\calE{\mathcal{E}}
\nc\bfF{\mathbf{F}} \nc\bbF{\mathbb{F}} \nc\calF{\mathcal{F}}
\nc\bfG{\mathbf{G}} \nc\bbG{\mathbb{G}} \nc\calG{\mathcal{G}}
\nc\bfH{\mathbf{H}} \nc\bbH{\mathbb{H}} \nc\calH{\mathcal{H}}
\nc\bfI{\mathbf{I}} \nc\bbI{\mathbb{I}} \nc\calI{\mathcal{I}}
\nc\bfJ{\mathbf{J}} \nc\bbJ{\mathbb{J}} \nc\calJ{\mathcal{J}}
\nc\bfK{\mathbf{K}} \nc\bbK{\mathbb{K}} \nc\calK{\mathcal{K}}
\nc\bfL{\mathbf{L}} \nc\bbL{\mathbb{L}} \nc\calL{\mathcal{L}}
\nc\bfM{\mathbf{M}} \nc\bbM{\mathbb{M}} \nc\calM{\mathcal{M}}
\nc\bfN{\mathbf{N}} \nc\bbN{\mathbb{N}} \nc\calN{\mathcal{N}}
\nc\bfO{\mathbf{O}} \nc\bbO{\mathbb{O}} \nc\calO{\mathcal{O}}
\nc\bfP{\mathbf{P}} \nc\bbP{\mathbb{P}} \nc\calP{\mathcal{P}}
\nc\bfQ{\mathbf{Q}} \nc\bbQ{\mathbb{Q}} \nc\calQ{\mathcal{Q}}
\nc\bfR{\mathbf{R}} \nc\bbR{\mathbb{R}} \nc\calR{\mathcal{R}}
\nc\bfS{\mathbf{S}} \nc\bbS{\mathbb{S}} \nc\calS{\mathcal{S}}
\nc\bfT{\mathbf{T}} \nc\bbT{\mathbb{T}} \nc\calT{\mathcal{T}}
\nc\bfU{\mathbf{U}} \nc\bbU{\mathbb{U}} \nc\calU{\mathcal{U}}
\nc\bfV{\mathbf{V}} \nc\bbV{\mathbb{V}} \nc\calV{\mathcal{V}}
\nc\bfW{\mathbf{W}} \nc\bbW{\mathbb{W}} \nc\calW{\mathcal{W}}
\nc\bfX{\mathbf{X}} \nc\bbX{\mathbb{X}} \nc\calX{\mathcal{X}}
\nc\bfY{\mathbf{Y}} \nc\bbY{\mathbb{Y}} \nc\calY{\mathcal{Y}}
\nc\bfZ{\mathbf{Z}} \nc\bbZ{\mathbb{Z}} \nc\calZ{\mathcal{Z}}
\nc\bfone {{\mathbf 1}}
\nc\Vect{\mathbf{Vect}}
\nc\Sets{\mathbf{Sets}}
\nc\Mod{\mathbf{Mod}}
\nc\Cat{\mathbf{Cat}}

\nc\ul{\underline}
\nc\simeqto{\overset{\simeq}{\longrightarrow }}
\nc\ct{\overset{\cong}{\longrightarrow }}
\nc\mt{\mapsto}
\nc\hr{\medskip\hrule\medskip}
\nc\trl{\triangleleft}
\nc\trr{\triangleright}

\nc\xysquare[8]{\xymatrix{
    #1 \ar[r]#5 \ar[d]#6 & #2 \ar[d]#7 \\
    #3 \ar[r]#8          & #4
  }
  }

\nc\Ob{\operatorname{Ob}}
\nc\Mor{\operatorname{Mor}}

\nc\ot{\otimes}
\nc\ott{\ot\cdots\ot}
\nc\Sp{\operatorname{Sp}}
\nc\SO{\operatorname{SO}}

\nc\HH{\mathrm{HH}}

\nc\ci{\circ}
\nc\sq{\square}

\nc\incl{\mathrm{incl}}

\nc\ol{\overline}
\nc\sqcups{\sqcup\cdots\sqcup}
\nc\congto{\overset{\cong}{\to}}

\nc\hide[1]{}
\nc\blue[1]{{\textcolor[rgb]{0,0,.9}{#1}}}
\nc\bluen[1]{\blue{[[#1]]}}
\nc\bnote{\bluen}
\nc\redn[1]{\red{[[#1]]}}

\nc\comment[1]{\marginpar{\tiny #1}}
\nc\len{\operatorname{len}}

\let\copybigwedge\bigwedge
\renewcommand\bigwedge{\copybigwedge\nolimits}
\title[Stable homology of the IA-automorphism groups of free groups]{Stable rational homology of the IA-automorphism groups of free groups} 
\author{Mai Katada}
\address{Department of Mathematics, Kyoto University, Kyoto 606-8502, Japan}
\address{Graduate School of Mathematical Sciences, University of Tokyo, Tokyo 153-8914, Japan}
\address{Faculty of Mathematics, Kyushu University, 744, Motooka, Nishi-ku, Fukuoka, Japan 819-0395}
\email{katada@math.kyushu-u.ac.jp}

\begin{document}

\begin{abstract}
The rational homology of the IA-automorphism group $\IA_n$ of the free group $F_n$ is still mysterious.
We study the quotient of the rational homology of $\IA_n$ that is obtained as the image of the map induced by the abelianization map, which we call the Albanese homology of $\IA_n$.
We obtain a representation-stable $\GL(n,\Q)$-subquotient of the Albanese homology of $\IA_n$, which conjecturally coincides with the entire Albanese homology of $\IA_n$.
In particular, we obtain a lower bound of the dimension of the Albanese homology of $\IA_n$ for each homological degree in a stable range.
Moreover, we determine the entire third Albanese homology of $\IA_n$ for $n\ge 9$.
We also study the Albanese homology of an analogue of $\IA_n$ to the outer automorphism group of $F_n$ and the Albanese homology of the Torelli groups of surfaces.
Moreover, we study the relation between the Albanese homology of $\IA_n$ and the cohomology of $\Aut(F_n)$ with twisted coefficients.

\end{abstract}
\maketitle

\setcounter{tocdepth}{1}
\tableofcontents

\section{Introduction}
 
The \emph{IA-automorphism group} $\IA_n$ of the free group $F_n$ of rank $n$ is the kernel of the canonical surjective map from the automorphism group $\Aut(F_n)$ of $F_n$ to the general linear group $\GL(n,\Z)$.
The homology of $\IA_n$ has a $\GL(n,\Z)$-module structure which is induced by the adjoint action of $\Aut(F_n)$ on $\IA_n$.
Magnus \cite{Magnus} discovered a finite set of generators for $\IA_n$.
There are several papers that have determined the first homology of $\IA_n$, and the independent results of Cohen--Pakianathan \cite{CP}, Farb \cite{Farb} and Kawazumi \cite{Kawazumi} are thought to be the first. However, the first homology of $\IA_n$ was determined by Bachmuth \cite{Bachmuth} as a free abelian group and by Formanek \cite{Formanek} as a $\GL(n,\Z)$-module.
Thanks to their works, we have an isomorphism of $\GL(n,\Z)$-modules
\begin{equation}\label{H1IAn}
    \tau:H_1(\IA_n,\Z)\xrightarrow{\cong} \Hom(H_{\Z},\bigwedge^2 H_{\Z}), \quad H_{\Z}=H_1(F_n,\Z),
\end{equation}
which is induced by the \emph{Johnson homomorphism} for $\Aut(F_n)$.

Krsti\'{c}--McCool \cite{KM} showed that $\IA_3$ is not finitely presentable and after that, Bestvina--Bux--Margalit \cite{BBM} showed that $H_2(\IA_3,\Z)$ has infinite rank.
However, it is still open whether or not $\IA_n$ is finitely presentable for $n\ge 4$.
About the second cohomology of $\IA_n$, Pettet \cite{Pettet} determined the $\GL(n,\Z)$-subrepresentation of $H^2(\IA_n,\Q)$ that is obtained by using the Johnson homomorphism. This subrepresentation is regarded as the second Albanese cohomology of $\IA_n$, which we will explain below.
Day--Putman \cite{DP} obtained a finite set of generators of $H_2(\IA_n,\Z)$ as $\GL(n,\Z)$-representations.
For $n=3$, Satoh \cite{Satoh2021} detected a nontrivial irreducible subrepresentation of $H^2(\IA_3,\Q)$, which can not be detected by the Johnson homomorphism.
However, the second homology of $\IA_n$ has not been completely determined.
It is more difficult to determine higher degree homology of $\IA_n$.

For a group $G$, we consider the following quotient group $H^A_i(G,\Q)$ of $H_i(G,\Q)$
$$H^A_i(G,\Q):=\im (\pi_*: H_i(G,\Q)\to H_i(G^{\ab},\Q)),$$
where $\pi_*$ is induced by the canonical surjection $\pi:G\twoheadrightarrow G^{\ab}$, where $G^{\ab}$ denotes the abelianization of $G$.
Church--Ellenberg--Farb \cite{CEF} called the dual notion $H_A^i(G,\Q)$ of $H^A_i(G,\Q)$ the Albanese cohomology since $H_A^i(G,\Q)$ is the subgroup of $H^i(G,\Q)$ coming from the associated Albanese variety if $G$ is the fundamental group of a compact K\"{a}hler manifold.
By following their terminology, we call $H^A_i(G,\Q)$ the \emph{Albanese homology} of $G$.
Since we have $H_i(G^{\ab},\Q)\cong \bigwedge^i H_1(G,\Q)$, where $\bigwedge^i H_1(G,\Q)$ is the $i$-th exterior power of $H_1(G,\Q)$, it is easier to determine the Albanese homology $H^A_*(G,\Q)$ than the ordinary homology $H_*(G,\Q)$.

Let $H=H_1(F_n,\Q)$ and $U=\Hom(H,\bigwedge^2 H)$.
Then the Johnson homomorphism induces a $\GL(n,\Z)$-homomorphism
$$\tau_*:H_i(\IA_n,\Q)\to H_i(U,\Q).$$
Therefore, we have $\GL(n,\Z)$-homomorphisms
$$ H_i(\IA_n,\Q)\twoheadrightarrow  H^A_i(\IA_n,\Q)\hookrightarrow H_i(U,\Q)\cong \bigwedge^i U.$$
Since $H_i(U,\Q)$ is an algebraic $\GL(n,\Q)$-representation, the $\GL(n,\Z)$-representation structure on $H^A_i(\IA_n,\Q)$ extends to an algebraic $\GL(n,\Q)$-representation structure.
In particular, since $H^A_i(\IA_n,\Q)$ is completely reducible, any subrepresentations and quotient representations of $H^A_i(\IA_n,\Q)$ can be considered as direct summands.
As $\GL(n,\Q)$-representations, $H^A_1(\IA_n,\Q)$ is determined by \eqref{H1IAn} and $H^A_2(\IA_n,\Q)$ is determined by Pettet \cite{Pettet}. 
The Albanese homology of $\IA_n$ of degree greater than $2$ is not completely determined, but there are several results which induce non-triviality \cite{Kawazumibraid, CHP} and upper bounds of dimensions \cite{CEF} of the Albanese homology of $\IA_n$, which we will explain later.

Church--Farb \cite{CFrep} introduced the notion of representation stability for a sequence 
$$V_0\xto{\phi_0} V_1\xto{\phi_1}V_2\to \cdots \to V_n\xto{\phi_n} V_{n+1}\to \cdots $$
of algebraic $\GL(n,\Q)$-representations, where $\phi_n: V_n\to V_{n+1}$ is a $\GL(n,\Q)$-homomorphism considering $V_{n+1}$ as a $\GL(n,\Q)$-representation via the canonical inclusion map $\GL(n,\Q)\hookrightarrow \GL(n+1,\Q)$.
Note that irreducible algebraic $\GL(n,\Q)$-representations are classified by \emph{bipartitions} with at most $n$ parts, which are pairs of partitions with total length at most $n$ (see Section \ref{irrrep} for details).
For a bipartition $\ul\lambda$, let $V_{\ul\lambda}(n)$ denote the irreducible $\GL(n,\Q)$-representation corresponding to $\ul\lambda$.
A sequence $\{V_n\}$ of algebraic $\GL(n,\Q)$-representations is \emph{representation stable}
if $\{V_n\}$ satisfies the following three conditions
for sufficiently large $n$:
\begin{itemize}
    \item $\phi_n: V_n\to V_{n+1}$ is injective,
    \item $\im \phi_n$ spans $V_{n+1}$ as a $\GL(n+1,\Q)$-representation,
    \item for each bipartition $\ul\lambda$, the multiplicity of $V_{\ul\lambda}(n)$ in $V_n$ is constant.
\end{itemize}
For example, $H_i(U,\Q)\cong \bigwedge^i U$ is representation stable for each $i\ge 0$.

We consider an analogue of $\IA_n$ to the outer automorphism group $\Out(F_n)$ of $F_n$.
Let $\IO_n$ denote the kernel of the canonical surjective map from $\Out(F_n)$ to $\GL(n,\Z)$.
Kawazumi \cite{Kawazumi} determined $H^A_1(\IO_n,\Q)$, and Pettet \cite{Pettet} determined $H^A_2(\IO_n,\Q)$ as $\GL(n,\Q)$-representations.
However, even non-triviality of the Albanese homology of $\IO_n$ of degree greater than $2$ does not seem to be known.

The aim of this paper is to determine the Albanese homology $H^A_i(\IA_n,\Q)$ as $\GL(n,\Q)$-representations.
We use \emph{abelian cycles} in $H_i(\IA_n,\Q)$, which are induced by $i$-tuples of mutually commuting elements of $\IA_n$.
Then we obtain a $\GL(n,\Q)$-subquotient of $H^A_i(\IA_n,\Q)$, which conjecturally coincides with the entire $H^A_i(\IA_n,\Q)$.
In particular, we obtain a lower bound of the dimension of $H^A_i(\IA_n,\Q)$ for sufficiently large $n$ with respect to $i$.
Moreover, we determine $H^A_3(\IA_n,\Q)$ for $n\ge 9$.
We also study the Albanese homology of $\IO_n$ and that of the Torelli groups of surfaces, and the relation between the Albanese homology of $\IA_n$ and the cohomology of $\Aut(F_n)$ with twisted coefficients.

\subsection{Non-triviality of $H^A_i(\IA_n,\Q)$}

Let $i\ge 1$.
First, we observe that $H^A_i(\IA_n,\Q)$ is non-trivial for $n\ge i+1$. We detect a non-trivial representation-stable quotient of $H^A_i(\IA_n,\Q)$.

For $g\ge 1$, $\IA_{2g}$ includes the \emph{Torelli group} $\I_{g,1}$ of a connected oriented surface of genus $g$ with one boundary component.
Church--Farb \cite{CFAbelJacobi} used an $\Sp(2g,\Z)$-homomorphism
$$\tau_i:H_i(\I_{g,1},\Q)\to\bigwedge^{i+2} H,$$
which was introduced by Johnson \cite{Johnsonsurvey},
to detect a non-trivial representation-stable subrepresentation in the image of $\tau_i$.
The following theorem can be regarded as an analogue of this result of Church--Farb.

\begin{theorem}[Theorem \ref{nontriviality}]\label{nontrivialintro}
  For $n\ge i+1$, we have a surjective $\GL(n,\Q)$-homomorphism 
  $$H^A_i(\IA_n,\Q) \twoheadrightarrow \Hom(H,\bigwedge^{i+1} H)\cong \left(\bigwedge^{i+1}H\right)\otimes H^*,$$
  where $H^*=\Hom_{\Q}(H,\Q)$. 
\end{theorem}

Kawazumi \cite[Corollary 4.1]{Kawazumibraid} proved that $H_A^i(\IA_n,\Q)$ surjects onto the $i$-th cohomology of the pure braid group of $n$ strings.
Cohen--Heap--Pettet \cite{CHP} obtained non-trivial subspaces of $H_A^i(\IA_n,\Q)$ whose dimensions are bounded above by a polynomial in $n$ of degree $2i$. 
However, these spaces are not $\GL(n,\Q)$-representations.
Theorem \ref{nontrivialintro} gives another proof of non-triviality of $H^A_i(\IA_n,\Q)$.

\begin{corollary}
 For $n\ge i+1$, $H^A_i(\IA_n,\Q)$ is non-trivial.
\end{corollary}

To observe the representation stability of the quotient of $H^A_i(\IA_*,\Q)$ that is detected in Theorem \ref{nontrivialintro}, we write $H$ as $H_n$ and $\Hom_{\Q}(H,\Q)$ as $H_n^*$.
The abelian cycles and the $\GL(n,\Q)$-homomorphisms that we use to prove Theorem \ref{nontrivialintro} are compatible with the canonical inclusion maps $$\IA_n\hookrightarrow \IA_{n+1},\quad H_n\hookrightarrow H_{n+1},\quad  H_n^*\hookrightarrow H_{n+1}^*,$$
which implies the representation stability of the quotient of $H^A_i(\IA_*,\Q)$.
Let 
$$\GL(\infty,\Q)=\varinjlim_{n} \GL(n,\Q),\quad \IA_{\infty}=\varinjlim_{n}\IA_{n},$$
and 
$$\bfH=\varinjlim_{n}H_n,\quad \bfH^*=\varinjlim_{n}(H_n)^{*}.$$
We can rephrase the representation stability in the sense of $\GL(\infty,\Q)$-representations as follows.

\begin{corollary}[Corollary \ref{infinitedimensional}]\label{introinfinitedimensional}
 We have a surjective $\GL(\infty,\Q)$-homomorphism
 $$H^A_i(\IA_{\infty},\Q) \twoheadrightarrow \left(\bigwedge^{i+1} \bfH\right)\otimes\bfH^*.$$
 In particular, we have $\dim_{\Q}(H^A_i(\IA_{\infty},\Q))=\dim_{\Q}(H_i(\IA_{\infty},\Q))=\infty$.
\end{corollary}

\subsection{Conjectural structure of $H^A_i(\IA_n,\Q)$}
Next, we detect a subquotient representation of $H^A_i(\IA_n,\Q)$, which is conjecturally equal to the entire $H^A_i(\IA_n,\Q)$.

For $i\ge 1$, let 
$$U_i=\Hom(H,\bigwedge^{i+1} H)\cong \left(\bigwedge^{i+1} H\right)\otimes H^*.$$
Note that we have $U=U_1$ and that $U_i$ vanishes for $n\le i$.
We have a direct sum decomposition 
$$U_i=U_i^{\tree}\oplus U_i^{\wheel},$$
where
$U_i^{\tree}$ denotes the subrepresentation of $U_i$ that is isomorphic to $V_{1^{i+1},1}$, and where $U_i^{\wheel}$ denotes the other subrepresentation that is isomorphic to $V_{1^i,0}$ for $n\ge i+1$.

For the graded $\GL(n,\Q)$-representation $U_*=\bigoplus_{i\ge 1} U_i$,
let $S^*(U_*)=\bigoplus_{i\ge 0} S^*(U_*)_i$ denote the graded-symmetric algebra of $U_*$.
We define the \emph{traceless part} $W_*=\wti{S}^* (U_*)$ of $S^*(U_*)$, which consists of elements that vanish under any contraction maps between distinct factors of $S^*(U_*)$
(see Section \ref{tracelesspartofgradedsymmetricalgebra} for details).
We can also construct $W_*$ by using an operad $\calC om$ of non-unital commutative algebras as we will see in Section \ref{KV}.

We construct a $\GL(n,\Q)$-homomorphism 
$$F_i: H_i(U,\Q)\to S^*(U_*)_i$$
by combining two kinds of contraction maps.
Then, we obtain $W_i$ as a subquotient representation of $H^A_i(\IA_n,\Q)$, which is one of our main results.

\begin{theorem}[Theorem \ref{Johnsonpart}]\label{Johnsonpartintro}
 For $n\ge 3i$, we have 
 $$F_i(H^A_i(\IA_n,\Q))\supset W_i.$$
\end{theorem}

\begin{remark}
We can decompose $W_i$ into direct summands
$$
W_i=\bigoplus_{(\mu,\nu)\in P_i}W(\mu,\nu),
$$
where $P_i$ denotes the set of pairs of partitions $(\mu,\nu)$ such that $\mu$ and $\nu$ are partitions of non-negative integers whose sum is $i$.
Recently Lindell \cite{Lindell} studied the Albanese homology $H^A_i(\I_{g,1},\Q)$ of the Torelli group $\I_{g,1}$.
Lindell's result \cite[Theorem 1.5]{Lindell} implies that for each pair of partitions $(\lambda,\mu)$ under some conditions, $H^A_i(\I_{g,1},\Q)$ contains an $\Sp(2g,\Q)$-subrepresentation $W^{\I_{g,1}}_i(\lambda,\mu)$ corresponding to $(\lambda,\mu)$.
\end{remark}

By \eqref{H1IAn}, we have $H^A_1(\IA_n,\Q)\cong W_1$ for $n\ge 3$.
By \cite{Pettet}, we have $H^A_2(\IA_n,\Q)\cong W_2$ for $n\ge 6$.
For $i=3$, we obtain the following theorem.

\begin{theorem}[Theorem \ref{thirdAlbanese}]
 For $n\ge 9$, we have a $\GL(n,\Q)$-isomorphism 
 $$F_3: H^A_3(\IA_n,\Q)\xrightarrow{\cong} W_3.$$
\end{theorem}

It seems natural to make the following conjecture.

\begin{conjecture}[Conjecture \ref{conjectureAlbanese}]\label{conjintro}
 For $n\ge 3i$, we have a $\GL(n,\Q)$-isomorphism
 $$F_i: H^A_i(\IA_n,\Q)\xrightarrow{\cong} W_i.$$ 
\end{conjecture}

 We will introduce the following relaxed version of Conjecture \ref{conjintro}.
  
  \begin{conjecture}[Conjecture \ref{conjectureAlbaneserelaxed}]\label{conjrelaxintro}
   We have a $\GL(n,\Q)$-isomorphism
  $$
   H^A_i(\IA_n,\Q) \cong  W_i
  $$
  for sufficiently large $n$ with respect to $i$.
  \end{conjecture}

\subsection{Coalgebra structure of $H^A_*(\IA_n,\Q)$}

For any group $G$, it is well known that $H_*(G,\Q)$ has a natural coalgebra structure.
Then $H^A_*(G,\Q)$ is a subcoalgebra of $H_*(G^{\ab},\Q)$. (See Section \ref{coalgebra} for details.)
We also have a coalgebra structure on $S^*(U_*)$ (see Section \ref{tracelesspartofgradedsymmetricalgebra}).
Then the graded $\GL(n,\Q)$-homomorphism $F_*=\bigoplus_{i}F_i$ is compatible with comultiplications.

\begin{proposition}[Proposition \ref{coalgebramap}]
 The graded $\GL(n,\Q)$-homomorphism
 $$F_*=\bigoplus_{i} F_i: H_*(U,\Q)\to S^*(U_*)$$
 is a coalgebra map.
\end{proposition}

For a coaugmented coalgebra $A$, let $\Prim(A)$ denote the primitive part of $A$.
We have 
$$\Prim(S^*(U_*))=U_*\subset W_*.$$
The $\GL(n,\Q)$-homomorphism $F_*$ restricts to a graded $\GL(n,\Q)$-homomorphism
$$
 F_*: \Prim(H^A_*(\IA_n,\Q))\to \Prim(S^*(U_*))=U_*.
$$
Conjecture \ref{conjintro} leads to the following conjecture, which implies that Theorem \ref{nontrivialintro} determines the primitive part of $H^A_*(\IA_n,\Q)$.

\begin{conjecture}[Conjecture \ref{conjectureprim}]
For $n\ge 3i$, the $\GL(n,\Q)$-homomorphism
$$ 
  F_i:\Prim(H^A_*(\IA_n,\Q))_i\to U_i
$$
is an isomorphism, where $\Prim(H^A_*(\IA_n,\Q))_i$ denotes the degree $i$ part of $\Prim(H^A_*(\IA_n,\Q))$.
\end{conjecture}

\subsection{Lower bound of the dimension of $H^A_i(\IA_n,\Q)$}
Church--Ellenberg--Farb \cite{CEF} introduced the theory of FI-modules and studied the representation stability of $H^A_i(\IA_n,\Q)$. They obtained the following theorem about the stability and an upper bound of the dimension of $H^A_i(\IA_n,\Q)$.

\begin{theorem}[Church--Ellenberg--Farb \cite{CEF}]\label{CEHdimension}
 For each $i\ge 0$, there exists a polynomial $P_i(T)$ of degree $\le 3i$ such that $\dim_{\Q}(H^A_i(\IA_n,\Q))=P_i(n)$ for sufficiently large $n$ with respect to $i$. 
\end{theorem}

We obtain a lower bound of $\dim_{\Q}(H^A_i(\IA_n,\Q))$.
As a consequence of Theorem \ref{Johnsonpartintro}, or directly, it can be shown that the \emph{traceless part} $H_i(U,\Q)^{\tl}$ of $H_i(U,\Q)$ is contained in $H^A_i(\IA_n,\Q)$, where $H_i(U,\Q)^{\tl}\subset H_i(U,\Q)$ is the subrepresentation that vanishes under any contraction maps. (See Section \ref{Tracelesspart}.)

\begin{theorem}[Theorem \ref{tracelessimage}] \label{tracelessimageintro}
 We have $H_i(U,\Q)^{\tl}\subset H^A_i(\IA_n,\Q)$ for $n\ge 3i$.
\end{theorem}

We have $\dim_{\Q}(H_i(U,\Q)^{\tl})=Q_i(n)$ for $n\ge 3i$, where $Q_i(T)$ is a polynomial of degree $3i$. By Theorems \ref{CEHdimension} and \ref{tracelessimageintro}, we obtain the following theorem about the dimension of $H^A_i(\IA_n,\Q)$.

\begin{theorem}[Theorem \ref{dimension}]\label{dimensionintro}
 We have $\dim_{\Q}(H^A_i(\IA_n,\Q))\ge Q_i(n)$ for $n\ge 3i$.
 Moreover, there exists a polynomial $P_i(T)$ of degree exactly $3i$ such that we have
 $\dim_{\Q}(H^A_i(\IA_n,\Q))=P_i(n)$
 for sufficiently large $n$ with respect to $i$. 
\end{theorem}

\subsection{Non-triviality of $H^A_i(\IO_n,\Q)$}

In a way similar to $\IA_n$, by using abelian cycles, we detect a non-trivial quotient representation of $H^A_i(\IO_n,\Q)$.
Note that unlike the case of $\IA_n$, there is no canonical inclusion map $\IO_n\hookrightarrow \IO_{n+1}$, so we do not consider the representation stability for $H^A_i(\IO_*,\Q)$.

\begin{theorem}[Theorem \ref{ConnectedpartIO}]\label{introconnectedpartIO}
  Let $i\ge 2$.
  For sufficiently large $n$ with respect to $i$, we have a surjective $\GL(n,\Q)$-homomorphism
  \begin{gather*}
      H^A_i(\IO_n,\Q)\twoheadrightarrow \Hom(H,\bigwedge^{i+1} H)\cong V_{1^{i+1},1} \oplus V_{1^i,0}.
  \end{gather*}
 \end{theorem}
 
 By Pettet \cite{Pettet}, we have non-triviality of $H^A_2(\IO_n,\Q)$. We obtain non-triviality of $H^A_i(\IO_n,\Q)$ for $i\ge 3$.
 
\begin{corollary}
Let $i\ge 2$. For sufficiently large $n$ with respect to $i$, $H^A_i(\IO_n,\Q)$ is non-trivial.
\end{corollary}

\subsection{Lower bound of the dimension of $H^A_i(\IO_n,\Q)$}
Unlike the case of $\IA_n$, the stability or an upper bound of the dimension of $H^A_i(\IO_n,\Q)$ is not known.
We detect a lower bound of the dimension of $H^A_i(\IO_n,\Q)$.

Let $U^O=U_1^{\tree}$. 
We have an isomorphism $H_1(\IO_n,\Q)\cong U^O$ by \cite{Kawazumi}.
The canonical projection $\pi:\IA_n\twoheadrightarrow \IO_n$ induces a  $\GL(n,\Q)$-homomorphism
$\pi_*: H^A_i(\IA_n,\Q)\to H^A_i(\IO_n,\Q)$.
By using Theorem \ref{tracelessimageintro}, we obtain a lower bound of the dimension of $H^A_i(\IO_n,\Q)$.

\begin{theorem}[Theorem \ref{dimensionIO}]\label{introdimensionIO}
 For $n\ge 3i$, $H^A_i(\IO_n,\Q)$ contains a subrepresentation which is isomorphic to $H_i(U,\Q)^{\tl}$.
 In particular, we have $\dim_{\Q}(H^A_i(\IO_n,\Q))\ge Q_i(n)$ for $n\ge 3i$.
\end{theorem}

\subsection{Conjectural structure of $H^A_i(\IO_n,\Q)$}
Here, we propose a conjectural structure of $H^A_i(\IO_n,\Q)$ and the relation between $H^A_*(\IO_n,\Q)$ and $H^A_*(\IA_n,\Q)$.

Let $U^O_1=U^O$ and $U^O_i=U_i$ for $i\ge 2$.
For the graded $\GL(n,\Q)$-representation $U^O_*=\bigoplus_{i\ge 1} U^O_i$, let $S^*(U^O_*)$ denote the graded-symmetric algebra of $U^O_*$.
Let $W^O_*=\wti{S}^*(U^O_*)$ denote the traceless part of $S^*(U^O_*)$.
Then we can decompose $W^O_i$ into direct summands
$$
 W^O_i=\bigoplus_{(\mu,\nu)\in P^O_i} W(\mu,\nu),
$$
where $P^O_i$ denotes the subset of $P_i$ consisting of pairs of partitions $(\mu,\nu)$ such that $\nu$ has no part of size $1$.

By \cite{Kawazumi}, we have $H^A_1(\IO_n,\Q)\cong W^O_1$ for $n\ge 3$. By \cite{Pettet}, we have $H^A_2(\IO_n,\Q)\cong W^O_2$ for $n\ge 6$. For $i=3$, we obtain the following theorem.

\begin{theorem}[Theorem \ref{thirdAlbaneseIO}]
 For $n\ge 9$, we have 
 $$
 H^A_3(\IO_n,\Q)\cong W^O_3
 $$
 as $\GL(n,\Q)$-representations.
\end{theorem}

It seems natural to make the following conjecture, which is an analogue of Conjecture \ref{conjintro} or \ref{conjrelaxintro}.

\begin{conjecture}[Conjecture \ref{conjectureAlbaneseIO}]\label{conjintroIO}
For sufficiently large $n$ with respect to $i$, we have a $\GL(n,\Q)$-isomorphism
$$
 H^A_i(\IO_n,\Q)\cong W^O_i.
$$
\end{conjecture}

On the relation between $H^A_*(\IO_n,\Q)$ and $H^A_*(\IA_n,\Q)$, the Hochschild--Serre spectral sequence for the exact sequence 
$$
  1\to \Inn(F_n)\to \IA_n\to \IO_n\to 1,
$$
where $\Inn(F_n)$ is the inner automorphism group of $F_n$, leads to the following proposition.

\begin{proposition}[Proposition \ref{AlbaneseIAandIO2}]\label{AlbaneseIAandIO2intro}
   For $n\ge 2$, we have a $\GL(n,\Q)$-isomorphism
  \begin{gather*}
   \begin{split}
       H^A_i(\IA_n,\Q) \xrightarrow{\cong}  H^A_i(\IO_n,\Q) \oplus (H^A_{i-1}(\IO_n,\Q)\otimes H).
   \end{split}
  \end{gather*}
 \end{proposition}
 
 By Theorem \ref{Johnsonpartintro} and Proposition \ref{AlbaneseIAandIO2intro}, we obtain the following proposition, which partially ensures Conjecture \ref{conjintroIO}.
 
\begin{proposition}[Proposition \ref{IOnWO}]
 For $n\ge 3i$, we have an injective $\GL(n,\Q)$-homomorphism
  $$W^O_i \oplus (W^O_{i-1}\otimes H)\hookrightarrow H^A_i(\IO_n,\Q) \oplus (H^A_{i-1}(\IO_n,\Q) \otimes H).$$
\end{proposition}

Now we have the following equivalence between conjectures about the structures of $H^A_*(\IA_n,\Q)$ and $H^A_*(\IO_n,\Q)$.

\begin{proposition}[Proposition \ref{AlbaneseIAandIO}]
  Conjecture \ref{conjrelaxintro} and Conjecture \ref{conjintroIO} are equivalent.
\end{proposition}

\subsection{Relation between $H^A_*(\IA_n,\Q)$ and the cohomology of $\Aut(F_n)$ with twisted coefficients}

The stable cohomology of $\Aut(F_n)$ with twisted coefficients has been studied by many authors. 
Satoh computed the first and second homology with coefficients in $H$ and $H^*$ \cite{Satoh2006, Satoh2007}.
Djament \cite{Djament19}, Vespa \cite{Vespa}
and Randal-Williams \cite{Randal-Williams} obtained the stable cohomology $H^*(\Aut(F_n),H^{\otimes p})$. 
Djament--Vespa \cite{Djament-Vespa}
and Randal-Williams \cite{Randal-Williams} obtained the stable cohomology $H^*(\Aut(F_n),(H^*)^{\otimes q})$.

Let $H^{p,q}=H^{\otimes p}\otimes (H^*)^{\otimes q}$.
Kawazumi--Vespa \cite{Kawazumi--Vespa} studied the stable cohomology $H^*(\Aut(F_n),H^{p,q})$.
Their conjecture \cite[Conjecture 6]{Kawazumi--Vespa} implies the following conjecture, where $\calC_{\calP_0^{\circlearrowright}}$ is the wheeled PROP associated to the operadic suspension $\calP_0$ of the operad $\calC om$ of non-unital commutative algebras (see Section \ref{KV} for details.)

\begin{conjecture}[Kawazumi--Vespa {\cite[Conjecture 6]{Kawazumi--Vespa}}, Conjecture \ref{conjectureKV}]\label{conjKVintro}
For $p,q\ge 0$, we stably have an isomorphism of graded $\Q[\gpS_{p}\times \gpS_{q}]$-modules
$$
  H^{*}(\Aut(F_n), H^{p,q})=
    \calC_{\calP_0^{\circlearrowright}}(p,q).
$$
\end{conjecture}

We make the following conjecture about the relation between the Albanese (co)homology of $\IA_n$ and the cohomology of $\Aut(F_n)$ with twisted coefficients, where the Albanese cohomology $H_A^i(\IA_n,\Q)$ of $\IA_n$ is isomorphic to $H^A_i(\IA_n,\Q)^*$ as $\GL(n,\Q)$-representations.

\begin{conjecture}[Conjecture \ref{cohomologyAut}]\label{cohomologyAutintro}
Let $i$ be a non-negative integer and $\ul\lambda$ a bipartition.
Then, for sufficiently large $n$, we have a linear isomorphism
$$
  H^i(\Aut(F_n),V_{\ul\lambda})\cong (H_A^i(\IA_n,\Q)\otimes V_{\ul\lambda})^{\GL(n,\Z)}.
$$
\end{conjecture}

Then we have the following relation between the conjectural structure of the Albanese homology of $\IA_n$ and the above two conjectures.

\begin{proposition}[Proposition \ref{AlbaneseandAut}]
Let $i$ be a non-negative integer.
  If two of the followings hold, then so does the third.
 \begin{enumerate}
     \item Conjecture \ref{conjrelaxintro} holds for homological degree $i$.
     \item Conjecture \ref{conjKVintro} holds for cohomological degree $i$.
     \item Conjecture \ref{cohomologyAutintro} holds for cohomological degree $i$.
 \end{enumerate}
\end{proposition}

\subsection{Conjectural structures of $H_A^*(\I_{g,1},\Q)$, $H_A^*(\I_{g},\Q)$ and $H_A^*(\I_{g}^{1},\Q)$}

Let $\I_{g}$ (resp. $\I_{g}^{1}$) denote the Torelli group of a closed surface (resp. a surface with one marked point) of genus $g$.
We propose conjectural structures of the Albanese cohomology of the Torelli groups $\I_{g}$, $\I_{g,1}$ and $\I_{g}^{1}$.
Note that the Albanese cohomology and homology of the Torelli groups are isomorphic since algebraic $\Sp(2g,\Q)$-representations are self-dual.
In a way similar to $W_*=\wti{S}^*(U_*)$ and $W^O_*=\wti{S}^*(U^O_*)$, we define $\Sp(2g,\Q)$-representations $\wti{S}^*(X''_*)$, $\wti{S}^*(Y''_*)$ and $\wti{S}^*(Z''_*)$.

\begin{conjecture}[Conjecture \ref{Torellialb}]
We stably have $\Sp(2g,\Q)$-isomorphisms
$$
   H_A^*(\I_{g},\Q)\cong \wti{S}^*(X''_*),\quad H_A^*(\I_{g,1},\Q)\cong \wti{S}^*(Y''_*),\quad 
   H_A^*(\I_{g}^{1},\Q)\cong \wti{S}^*(Z''_*).
$$
\end{conjecture}

We also study the structures of the cohomology of Lie algebras associated to $\I_{g}$, $\I_{g,1}$ and $\I_{g}^{1}$.

\subsection{Outline}
 The rest of the paper is organized as follows.
 In Section \ref{Preliminary}, we recall algebraic $\GL(n,\Q)$-representations and the notion of traceless parts.
 In Section \ref{Abeliancycles}, we recall the notion of abelian cycles. We construct an abelian cycle for $\IA_n$ corresponding to each pair of partitions.
 In Section \ref{Connectedpart}, we obtain Theorem \ref{nontrivialintro}.
 In Section \ref{Tracelesspart}, we show Theorem \ref{tracelessimageintro}, which induces Theorem \ref{dimensionintro} about the dimension of $H^A_i(\IA_n,\Q)$.
 In Section \ref{Generalpart}, we study the structure of $H^A_i(\IA_n,\Q)$ and prove Theorem \ref{Johnsonpartintro}.
 In Section \ref{coalgebra}, we recall a coalgebra structure on group homology and study the coalgebra structure of $H^A_*(\IA_n,\Q)$.
 In Section \ref{albanesecohomology}, we study the algebra structure of the Albanese cohomology $H_A^*(\IA_n,\Q)$.
 In Section \ref{Albanesehomology}, we study the Albanese homology $H^A_i(\IO_n,\Q)$ to obtain Theorems \ref{introconnectedpartIO} and \ref{introdimensionIO}. We also study the relation between $H^A_i(\IO_n,\Q)$ and $H^A_i(\IA_n,\Q)$.
 In Section \ref{ThirdhomologyIAandIO}, we determine $H^A_3(\IO_n,\Q)$ and $H^A_3(\IA_n,\Q)$ for $n\ge 9$, where we postpone somewhat technical computations to Appendix \ref{AppendixB}.
 In Section \ref{KV}, we study the relation between $H_A^*(\IA_n,\Q)$ and $H^*(\Aut(F_n),H^{p,q})$.
 In Section \ref{Sprep}, we recall algebraic $\Sp(2g,\Q)$-representations.
 In Section \ref{problems}, we discuss the Albanese cohomology of the Torelli groups and the cohomology of Lie algebras associated to the Torelli groups.
 In Appendix \ref{Albanesefunctor}, we give a brief summary of some properties of Albanese homology and cohomology of groups.
 In Appendix \ref{AppendixB}, we give the explicit computation of upper bounds of the third Albanese homology of $\IA_n$ and $\IO_n$.
 
\subsection*{Acknowledgements}
The author would like to thank her supervisor Kazuo Habiro for valuable advice. She also thanks Nariya Kawazumi and Christine Vespa for comments about topics around homology of $\IA_n$.
This work was supported by JSPS KAKENHI Grant Number JP22J14812.

\section{Preliminaries}\label{Preliminary}
 In this section, we recall representation theory of $\GL(n,\Q)$ and of $\GL(\infty,\Q)$, and introduce the notion of traceless tensor products and traceless parts of graded-symmetric algebras.

 \subsection{Notations and conventions}\label{secnotations}
 For a non-negative integer $j$, let $[j]$ denote the set $\{1,\cdots,j\}$.
 
 Let $n\ge 1$.
 Let $F_n=\langle x_1,\cdots,x_n\rangle$.
 Let $H=H(n)=H_1(F_n,\Q)=\bigoplus_{j=1}^n \Q e_j$, where we fix the basis $\{e_j\}_{j=1}^{n}$ for $H$ given by the images of the $x_j$'s in $H\cong (F_n)^{\ab}\otimes \Q$. 
 We have $H^*=H(n)^*=H^1(F_n,\Q)=\bigoplus_{j=1}^n \Q e_j^*$.
  
  In what follows, we consider representations of $\GL(H,\Q)=\GL(n,\Q)$ over $\Q$, which are the same as $\Q[\GL(n,\Q)]$-representations. Let $[A]\in \Q[\GL(n,\Q)]$ denote the canonical image of $A\in \GL(n,\Q)$.
  Sometimes we simply write ``representations'' to mean $\GL(n,\Q)$-representations.
 
  In computations, we use the following matrices in $\GL(n,\Q)$.
 Let $\id\in \GL(n,\Q)$ denote the identity matrix.
 For distinct elements $k,l\in [n]$, 
 let $E_{k,l}\in \GL(n,\Q)$ denote the matrix that maps $e_l$ to $e_k+e_l$ and fixes $e_a$ for $a\neq l$. Then $E_{k,l}$ maps the dual basis $e_k^*$ to $e_k^*-e_l^*$ and fixes $e_a^*$ for $a\neq k$.
 Let $P_{k,l}\in \GL(n,\Q)$ denote the matrix that exchanges $e_k$ and $e_l$ and fixes $e_a$ for $a\neq k,l$.
 For $k=l$, we have $P_{k,k}=\id$.
 
 \subsection{Irreducible polynomial representations of $\GL(n,\Q)$}
  Here we recall several notions from representation theory. 
  See \cite{FH} for details.
  
  Let $n\ge 1$.
  A \emph{partition} $\lambda=(\lambda_1,\lambda_2,\dots,\lambda_n)$ with at most $n$ parts is a sequence of non-negative integers such that $\lambda_1\ge\lambda_2\ge\dots\ge\lambda_n$.
  Let $l(\lambda)=\max(\{0\}\cup\{i\mid \lambda_{i}> 0\})$ denote the \emph{length} of $\lambda$ and $|\lambda|=\lambda_1+\cdots+\lambda_{l(\lambda)}$ the \emph{size} of $\lambda$.
  We write $\lambda \vdash |\lambda|$.
  A \emph{Young diagram} for a partition $\lambda$ is a diagram with $\lambda_i$ boxes in the $i$-th row such that the rows of boxes are left-aligned.
  A \emph{tableau} on a Young diagram is a numbering of the boxes by integers in $[|\lambda|]$.
  A tableau is \emph{standard} if the numbering of each row and each column is increasing. 
  The \emph{canonical tableau} on a Young diagram is a standard tableau whose numbering starts with the first row from left to right, followed by the second row from left to right, and so on.
  
  Let 
  \begin{gather*}
      a_{\lambda}= \sum_{\sigma\in R} \sigma\in \Q[\gpS_{|\lambda|}],\quad b_{\lambda}= \sum_{\tau \in C} \sgn(\tau) \tau\in \Q[\gpS_{|\lambda|}],
  \end{gather*}
  where $R$ (resp. $C$) is a subgroup of $\gpS_{|\lambda|}$ preserving rows (resp. columns) of the canonical tableau on the Young diagram corresponding to $\lambda$.
  The \emph{Young symmetrizer} $c_{\lambda}$ is defined by
  \begin{gather*}
      c_{\lambda}= b_{\lambda} a_{\lambda}\in \Q[\gpS_{|\lambda|}].
  \end{gather*}
  Let $S^\lambda=\Q[\gpS_{|\lambda|}] c_{\lambda}$ denote the \emph{Specht module} corresponding to the partition $\lambda$, which is an irreducible representation of $\gpS_{|\lambda|}$.
  
  We call a $\GL(n,\Q)$-representation $V$ \emph{polynomial} if after choosing a basis for $V$, the $(\dim V)^2$ coordinate functions of the action $\GL(n,\Q)\to \GL(V)$ are polynomial functions of the $n^2$ variables.
  Consider $H\cong \Q^n$ as the standard representation of $\GL(n,\Q)$.
  For a partition $\lambda$, let 
  $$V_{\lambda}= H^{\otimes |\lambda|}\otimes_{\Q[\gpS_{|\lambda|}]}S^\lambda.$$
  If $\lambda$ has at most $n$ parts, then $V_{\lambda}$ is an irreducible polynomial $\GL(n,\Q)$-representation, and otherwise, we have $V_{\lambda}=0$.
  For a non-negative integer $p$, by the Schur--Weyl duality, we have a decomposition of $H^{\otimes p}$ as $\GL(n,\Q)\times \gpS_{p}$-modules
  $$
   H^{\otimes p}= \bigoplus_{\lambda\vdash p \text{ with at most $n$ parts}} V_{\lambda}\otimes S^{\lambda}.
  $$
  It is well known that irreducible polynomial representations of $\GL(n,\Q)$ are classified by partitions with at most $n$ parts, that is, any irreducible polynomial $\GL(n,\Q)$-representation is isomorphic to $V_{\lambda}$ for a partition $\lambda$ with at most $n$ parts.
  
  We have the following irreducible decomposition of the tensor product of two irreducible polynomial $\GL(n,\Q)$-representations $V_{\lambda}$ and $V_{\mu}$
  $$
   V_{\lambda}\otimes V_{\mu}=\bigoplus_{\nu} V_{\nu}^{\oplus N_{\lambda\mu}^{\nu}},
  $$
  where $N_{\lambda\mu}^{\nu}$ is the Littlewood--Richardson coefficient.
  
\subsection{Irreducible algebraic representations of $\GL(n,\Q)$}\label{irrrep}
 We call a $\GL(n,\Q)$-representation $V$ \emph{algebraic} if after choosing a basis for $V$, the $(\dim V)^2$ coordinate functions of the action $\GL(n,\Q)\to \GL(V)$ are rational functions of the $n^2$ variables.
 Here we recall irreducible algebraic representations of $\GL(n,\Q)$, which generalizes irreducible polynomial representations of $\GL(n,\Q)$. See \cite{FH,Koike,Patzt} for details.

  Let $p$ and $q$ be positive integers.
  For a pair $(k,l)\in[p]\times [q]$, the \emph{contraction map} 
  $$c_{k,l}:H^{\otimes p}\otimes (H^{*})^{\otimes q}\rightarrow H^{\otimes p-1}\otimes (H^{*})^{\otimes q-1}$$
  is defined for $v_1\otimes\cdots\otimes v_{p}\in H^{\otimes p}$ and $f_1\otimes\cdots\otimes f_{q} \in (H^{*})^{\otimes q}$ by
  $$c_{k,l}((v_1\otimes\cdots\otimes v_{p}) \otimes(f_1\otimes\cdots\otimes f_{q}))
  = \langle v_k,f_l\rangle
    (v_1\otimes\cdots\hat{v}_{k}\cdots\otimes v_{p})\otimes (f_1\otimes\cdots\hat{f}_{l}\cdots\otimes f_{q}),
  $$
  where 
  $\langle -,-\rangle:H\otimes H^{*}\rightarrow \Q$
  denotes the dual pairing,
  and where $\hat{v}_{k}$ (resp. $\hat{f}_{l}$) denotes the omission of $v_{k}$ (resp. $f_{l}$).    
  
  For $p,q\ge 0$, let $T_{p,q}$ denote the \emph{traceless part} of $H^{\otimes p}\otimes (H^{*})^{\otimes q}$, which is defined by
  $$T_{p,q}=\bigcap_{(k,l)\in [p]\times [q]} \ker c_{k,l}\subset H^{\otimes p}\otimes (H^{*})^{\otimes q}.$$
  Note that we have $T_{0,q}=(H^*)^{\otimes q}$ and $T_{p,0}=H^{\otimes p}$.
  
  A \emph{bipartition} $\ul{\lambda}=(\lambda^{+},\lambda^{-})$ with at most $n$ parts is a pair of two partitions $\lambda^{+}$ and $\lambda^{-}$ such that $l(\ul\lambda)=l(\lambda^{+})+l(\lambda^{-})\le n$.
  The \emph{size} $|\ul\lambda|$ of $\ul\lambda$ is defined by $|\lambda^{+}|+|\lambda^{-}|$. 
  
  For a bipartition $\ul\lambda$, set $p=|\lambda^{+}|$ and $q=|\lambda^{-}|$. 
  Let
  $$
    V_{\ul\lambda}=T_{p,q}\otimes_{\Q[\gpS_{p}\times\gpS_{q}]} (S^{\lambda^{+}}\otimes S^{\lambda^{-}}).
  $$
  If $\ul\lambda$ has at most $n$ parts, then $V_{\ul\lambda}$ is an irreducible algebraic representation, and otherwise, we have $V_{\ul\lambda}=0$.
  We have the following decomposition of $T_{p,q}$ as a $\GL(n,\Q)\times (\gpS_p\times \gpS_q)$-modules, which generalizes the Schur--Weyl duality for $H^{\otimes p}$
  \begin{gather}\label{Tpq}
    T_{p,q}=\bigoplus_{\substack{\ul\lambda: \text{bipartition with at most $n$ parts} \\ |\lambda^{+}|=p,\;|\lambda^{-}|=q}} V_{\ul\lambda}\otimes (S^{\lambda^{+}}\otimes S^{\lambda^{-}}).
  \end{gather}(See \cite[Theorem 1.1]{Koike}.)
  It is well known that irreducible algebraic representations of $\GL(n,\Q)$ are classified by bipartitions with at most $n$ parts, that is, any irreducible algebraic $\GL(n,\Q)$-representation is isomorphic to $V_{\ul\lambda}$ for a bipartition $\ul\lambda$ with at most $n$ parts.
  (See also \cite{Patzt} for the correspondence between irreducible algebraic representations and bipartitions.)
  
   For a bipartition $\ul\lambda=(\lambda^{+},\lambda^{-})$ with at most $n$ parts, we have
  $V_{\ul\lambda}=V_{\mu}\otimes \det^{k}$,
  where $\mu$ is a partition and $k$ is an integer satisfying
  $$(\lambda^{+}_1,\cdots, \lambda^{+}_{l(\lambda^{+})},0,\cdots,0,-\lambda^{-}_{l(\lambda^{-})},\cdots,-\lambda^{-}_1)= (\mu_1+k,\cdots, \mu_n+k),$$
  and where $\det$ denotes the $1$-dimensional determinant representation.
  
  The \emph{dual} of $\ul\lambda$ is defined by $\ul\lambda^*=(\lambda^{-},\lambda^{+})$.
  Note that we have a $\GL(n,\Q)$-isomorphism $(V_{\ul\lambda})^*\cong V_{\ul\lambda^*}$.
  
  For two bipartitions $\ul\lambda, \ul\mu$ such that $n\ge l(\ul\lambda)+l(\ul\mu)$, we have the following irreducible decomposition of the tensor product of $V_{\ul\lambda}$ and $V_{\ul\mu}$ 
  $$
  V_{\ul\lambda}\otimes V_{\ul\mu}=\bigoplus_{\ul\nu} V_{\ul\nu}^{\oplus N_{\ul\lambda \ul\mu}^{\ul\nu}},
  $$
 where $N_{\ul\lambda \ul\mu}^{\ul\nu}=
 \sum_{\alpha\beta\theta\delta}
 (\sum_{\kappa}N_{\kappa\alpha}^{\lambda^+}N_{\kappa\beta}^{\mu^-})
 (\sum_{\epsilon}N_{\epsilon\theta}^{\lambda^-}N_{\epsilon\delta}^{\mu^+})
 N_{\alpha\delta}^{\nu^+}N_{\beta\theta}^{\nu^-}$ (see \cite{Koike} for details).

 \subsection{Generator of the traceless part $T_{p,q}$}
   Here we give a generator of the traceless part $T_{p,q}$ of $H^{\otimes p}\otimes (H^*)^{\otimes q}$.
   
   We have the following explicit highest weight vectors of $T_{p,q}$.
  
  \begin{lemma}[Theorems 2.7 and 2.11 of \cite{BCHLLS}]\label{hwv}
   Let $p,q\ge 0$.
   For a bipartition $\ul\lambda$ with at most $n$ parts such that $|\lambda^{+}|=p$ and $|\lambda^{-}|=q$, let $$e(\ul\lambda)=(e_1^{\otimes \lambda^{+}_1} \otimes \cdots \otimes e_{l(\lambda^{+})}^{\otimes \lambda^{+}_{l(\lambda^{+})}}) \otimes ((e_n^*)^{\otimes \lambda^{-}_1} \otimes \cdots \otimes (e_{n-l(\lambda^{-})+1}^{*})^{\otimes \lambda^{-}_{l(\lambda^{-})}})
   \in H^{\otimes p}\otimes (H^*)^{\otimes q}.$$
   Let $St_{\lambda^+}$ (resp. $St_{\lambda^-}$) denote the subset of $\gpS_p$ (resp. $\gpS_q$) consisting of permutations which send the canonical tableau to standard tableaux on the Young diagram corresponding to $\lambda^+$ (resp. $\lambda^-$).
   Then the set 
   $$\{(\pi b_{\lambda^{+}}\otimes \id) (\id\otimes \rho b_{\lambda^{-}}) e(\ul\lambda)\mid |\lambda^{+}|=p,\: |\lambda^{-}|=q,\: \pi\in St_{\lambda^+},\: \rho\in St_{\lambda^-}\}$$
   is a basis for the space of all highest weight vectors of $T_{p,q}$.
   
   In particular, for each $\ul\lambda$, $(\pi b_{\lambda^{+}}\otimes \id) (\id\otimes \rho b_{\lambda^{-}})  e(\ul\lambda)$ generates an irreducible subrepresentation of $T_{p,q}$ which is isomorphic to $V_{\ul\lambda}$.
  \end{lemma}

  \begin{lemma}\label{tracelessgen}
   Let $n\ge p+q$.
   Then the traceless part $T_{p,q}$ of $H^{\otimes p}\otimes (H^*)^{\otimes q}$ is generated by 
   $$e_{p,q}=e_1\otimes \cdots \otimes e_p \otimes e_n^*\otimes \cdots \otimes e_{n-q+1}^* \in H^{\otimes p}\otimes (H^*)^{\otimes q}$$
   as a $\GL(n,\Q)$-representation.
  \end{lemma}
  
  \begin{proof}
   The proof is analogous to that of \cite[Lemma 2.1]{Lindell}.
   By Lemma \ref{hwv}, the traceless part $T_{p,q}$ is generated by 
   $$\{(\pi b_{\lambda^{+}}\otimes \id) (\id\otimes \rho b_{\lambda^{-}}) e(\ul\lambda)\mid |\lambda^{+}|=p,\: |\lambda^{-}|=q,\: \pi\in St_{\lambda^+},\: \rho\in St_{\lambda^-}\}.$$
   Therefore, it suffices to show that 
   $$
    e_J:= e_{j_1}\otimes \cdots \otimes e_{j_p} \otimes e_{j_{n}}^*\otimes \cdots \otimes e_{j_{n-q+1}}^* \in \Q[\GL(n,\Q)] e_{p,q}
   $$
   for $j_1, \cdots, j_p \in [p]$, $j_{n},\cdots,j_{n-q+1}\in \{n,\cdots,n-q+1\}$.
   
   Let 
   $$K_p=\{k\in [p]\mid e_{j_k}=e_{j_l} \text{ for some }l<k\}$$
   and 
   $$K'_q=\{k\in [q]\mid e_{j_{n-k+1}}^*=e_{j_{n-l+1}}^* \text{ for some } l<k\}.$$
   We can reorder the basis $\{e_1,\cdots,e_n\}$ for $H$ by using permutation matrices in such a way that $e_{j_k}=e_k$ for $k\in [p]\setminus K_p$, that is, we may assume that $e_J$ coincides with
   $e_1\otimes \cdots\otimes e_p\otimes  e_{j_{n}}^*\otimes \cdots \otimes e_{j_{n-q+1}}^*$
   except for tensor factors where $e_J$ has repeating elements.
   We can also assume that $e_J$ coincides with $e_{p,q}$ except for tensor factors where $e_J$ has repeating elements
   by reordering the basis $\{e_1,\cdots,e_n\}$ again in such a way that $e_{j_{n-k+1}}^*=e_{n-k+1}^*$ for $k\in [q]\setminus K'_q$.
   
   Now, we have to check that $e_{J}\in \Q[\GL(n,\Q)] e_{p,q}$. 
   Define 
   $$E_J:=\prod_{k\in K_p} ([E_{j_k,k}]-[\id])\prod_{k\in K_q}([\id]-[E_{n-k+1,j_{n-k+1}}])\in \Q[\GL(n,\Q)].$$ 
   Then we have
   \begin{gather*}
      e_{J}=E_J(e_{p,q})\in \Q[\GL(n,\Q)] e_{p,q},
   \end{gather*}
   which completes the proof.
  \end{proof}
  
  \begin{corollary}\label{projgen}
   Let $n\ge p+q$.
   Let $X$ be a $\GL(n,\Q)$-representation with a projection $\pi:T_{p,q}\twoheadrightarrow X$.
   Then $X$ is generated by $\pi(e_{p,q})$ as a $\GL(n,\Q)$-representation.
  \end{corollary}
 
 \subsection{Traceless tensor products}\label{tracelesswedge}
  Here we define traceless tensor products of algebraic $\GL(n,\Q)$-representations.
  
  By Lemma \ref{hwv}, we have a subrepresentation $V_{\ul\lambda}\subset H^{\otimes |\lambda^+|}\otimes (H^*)^{\otimes |\lambda^-|}$ which is generated by $(b_{\lambda^{+}}\otimes \id) (\id\otimes b_{\lambda^{-}})  e(\ul\lambda)$ for each bipartition $\ul\lambda$.
  
  Let $\ul\lambda$ and $\ul\mu$ be two bipartitions. Let $p=|\lambda^{+}|, q=|\lambda^{-}|, r=|\mu^{+}|$ and $s=|\mu^{-}|$.
  Define the \emph{traceless tensor product} $V_{\ul\lambda}\wti{\otimes} V_{\ul\mu}$ of $V_{\ul\lambda}$ and $V_{\ul\mu}$ as 
  $$V_{\ul\lambda}\wti{\otimes} V_{\ul\mu}=(V_{\ul\lambda}\otimes V_{\ul\mu})\cap T_{p+r,q+s}\subset  H^{\otimes (p+r)}\otimes (H^*)^{\otimes (q+s)}.$$
  In other words, $V_{\ul\lambda}\wti{\otimes} V_{\ul\mu}$ is a subrepresentation of $V_{\ul\lambda}\otimes V_{\ul\mu}$ which vanishes under any contraction maps.
  Then we have for $n\ge l(\ul\lambda)+l(\ul\mu)$,
  $$V_{\ul\lambda}\wti{\otimes} V_{\ul\mu}\cong  \bigoplus_{|\ul\nu|= |\ul\lambda|+|\ul\mu|} V_{\ul\nu}^{\oplus N_{\ul\lambda \ul\mu}^{\ul\nu}}\subset
  \bigoplus_{\ul\nu} V_{\ul\nu}^{\oplus N_{\ul\lambda \ul\mu}^{\ul\nu}}
  \cong V_{\ul\lambda}\otimes V_{\ul\mu}.$$
  
  Let $M$ be an algebraic $\GL(n,\Q)$-representation. For each bipartition $\ul\lambda$, define a vector space $$M_{\ul\lambda}=\Hom_{\GL(n,\Q)}(V_{\ul\lambda},M).$$
  Since the category of algebraic $\GL(n,\Q)$-representations is split semisimple, we have a natural isomorphism 
  $$
  \iota_M: M\xrightarrow{\cong} \bigoplus_{\ul\lambda}V_{\ul\lambda}\otimes M_{\ul\lambda}.
  $$
  For two algebraic $\GL(n,\Q)$-representations $M$ and $N$, we have 
   $$
  \iota_M\otimes \iota_N: M\otimes N \xrightarrow{\cong} \left(\bigoplus_{\ul\lambda}V_{\ul\lambda}\otimes M_{\ul\lambda}\right) 
  \otimes 
  \left(\bigoplus_{\ul\mu}V_{\ul\mu}\otimes N_{\ul\mu}\right)
  \cong 
  \bigoplus_{\ul\lambda,\ul\mu}(V_{\ul\lambda}\otimes V_{\ul\mu})\otimes (M_{\ul\lambda}\otimes N_{\ul\mu}).
  $$
  Define the \emph{traceless tensor product} $M \wti{\otimes} N$ of $M$ and $N$ as 
  $$
  M \wti{\otimes} N=(\iota_{M}\otimes \iota_{N})^{-1} \left(\bigoplus_{\ul\lambda,\ul\mu}(V_{\ul\lambda}\wti{\otimes} V_{\ul\mu})\otimes (M_{\ul\lambda}\otimes N_{\ul\mu})\right)\subset M\otimes N.
  $$
  Let $M^{\wti{\otimes} 0}=\Q$ and $M^{\wti{\otimes} 1}=M$.
  For $n\ge 2$, we define $M^{\wti{\otimes}n}=M^{\wti{\otimes}n-1}\wti{\otimes} M\subset M^{\otimes n}$ iteratively, and define the \emph{traceless part} $\wti{T}^*M$ of the tensor algebra $T^*M$ as
  $$\wti{T}^*M=\bigoplus_{n\ge 0} M^{\wti{\otimes}n}\subset T^*M.$$
  
  Let $\bigwedge^* M$ denote the exterior algebra of $M$ and $\Sym^* M$ the symmetric algebra of $M$.
  We have canonical projections $T^*M\twoheadrightarrow \bigwedge^* M$ and $T^*M \twoheadrightarrow \Sym^* M$.
  We define the \emph{traceless part} $\wti{\bigwedge}^* M$ of $\bigwedge^* M$ as the image of $\wti{T}^*M$ under the above projection, and the \emph{traceless part} $\wti{\Sym}^* M$ of $\Sym^* M$ in a similar way.

 \subsection{Traceless parts of graded-symmetric algebras}\label{tracelesspartofgradedsymmetricalgebra}
 
  Let $M_*=\bigoplus_{i\ge 1} M_i$ be a graded algebraic $\GL(n,\Q)$-representation. 
  Let 
  $S^*(M_*)=\bigoplus_{i\ge 0}S^*(M_*)_i$, where\begin{gather*}
      S^* (M_*)_i =\bigoplus_{k_1+2k_2 +\cdots + p k_p=i} S^{k_1}(M_{1})  \otimes\cdots \otimes S^{k_p}(M_{p}),
  \end{gather*}
  denote the graded-symmetric algebra of $M_*$.
  That is, we have the following graded-commutativity 
  $$
  y x =(-1)^{ij} x y
  $$
  for $x\in M_i$ and $y\in M_j$.
  Then we have $S^k (M_i)=\Sym^k(M_i)$ for even $i$ and $S^k (M_i)=\bigwedge^k(M_i)$ for odd $i$.
  We have a canonical projection $T^*(M_*)\twoheadrightarrow S^*(M_*)$.
  Define the \emph{traceless part} $\wti{S}^* (M_*)=\bigoplus_{i\ge 0} \wti{S}^* (M_*)_i$ of the graded-symmetric algebra $S^*(M_*)$ as the image of $\wti{T}^*(M_*)$ under the projection.
  Then $\wti{S}^* (M_*)$ is a graded $\GL(n,\Q)$-representation satisfying
  \begin{gather*}
      \wti{S}^* (M_*)_i =\bigoplus_{k_1+2k_2 +\cdots + p k_p=i} \left(\wti{S}^{k_1}(M_{1})\right)  \wti{\otimes}\cdots \wti{\otimes} \left(\wti{S}^{k_p}(M_{p})\right).
  \end{gather*}
  
  The graded-symmetric algebra $S^*(M_*)$ has a coalgebra structure defined as follows.
 For an element $x=x_1\cdots x_k\in S^*(M_*)$, where $x_j\in M_{i_j}$, and for $\sigma\in \gpS_k$, let $\sgn(\sigma; x)\in \{1,-1\}$ denote the sign satisfying
 $$
 x_{\sigma(1)}\cdots x_{\sigma(k)}= \sgn(\sigma; x)x_1\cdots x_k.
 $$
 Then the comultiplication $\Delta$ is defined by
 $$
 \Delta(x_1\cdots x_k)=\sum_{p=0}^k \sum_{\sigma\in \Sh(p,k-p)}\sgn(\sigma; x) x_{\sigma(1)}\cdots x_{\sigma(p)} \otimes x_{\sigma(p+1)}\cdots x_{\sigma(k)},
 $$
 where $\Sh(p,k-p)\subset \gpS_k$ denotes the set of $(p,k-p)$-shuffles.
 Then we can check that the coalgebra structure of $S^*(M_*)$ induces a subcoalgebra structure of the traceless part $\wti{S}^*(M_*)$. We can also check that the primitive part of $S^*(M_*)$ is $M_*$.
 
 \begin{remark}
  The traceless part $\wti{S}^*(M_*)$ does not inherit the algebra structure of $S^*(M_*)$. However, we can consider an algebra structure on $\wti{S}^*(M_*)$ in the symmetric monoidal category that we introduce below.
  Let $\mathbf{Rep}^{\text{alg}}(\GL(n,\Q))$ denote the category of algebraic $\GL(n,\Q)$-representations and $\GL(n,\Q)$-homomorphisms.
  The traceless tensor product $\wti{\otimes}$ and the symmetry $\wti{\tau}_{V,W}: V\wti{\otimes}W\to W\wti{\otimes}V$ which is the restriction of the usual symmetry $\tau_{V,W}:V\otimes W\xrightarrow{\cong} W\otimes V$
  form a symmetric monoidal structure $(\mathbf{Rep}^{\text{alg}}(\GL(n,\Q)),\wti{\otimes}, \wti{\tau})$.
  Then $\wti{S}^*(M_*)$ is a bialgebra in $(\mathbf{Rep}^{\text{alg}}(\GL(n,\Q)),\wti{\otimes}, \wti{\tau})$.
 \end{remark}
 
  \subsection{$\GL(\infty,\Q)$-representations}\label{prestable}
   Let $n\ge 0$.
   We have an inclusion $\GL(n,\Q)\hookrightarrow \GL(n+1,\Q)$ sending $A\in \GL(n,\Q)$ to $A\oplus {1}\in \GL(n+1,\Q)$.
   Let $\GL(\infty,\Q)=\varinjlim_{n}\GL(n,\Q)=\bigcup_{n\ge 1} \GL(n,\Q)$.
   Here we recall representation theory of $\GL(\infty,\Q)$. See \cite{SamSnowden} for details.
   
   We have sequences of algebraic $\GL(n,\Q)$-representations
   \begin{gather*}
       H(0)\hookrightarrow H(1)\hookrightarrow\dots \hookrightarrow H(n)\hookrightarrow H(n+1)\hookrightarrow\dots,\\
       H(0)^*\hookrightarrow H(1)^*\hookrightarrow\dots \hookrightarrow H(n)^*\hookrightarrow H(n+1)^*\hookrightarrow\dots,
   \end{gather*}
   where the map $H(n)^{(*)}\hookrightarrow H(n+1)^{(*)}$ is the canonical inclusion map that sends the basis vector $e_j^{(*)}$ to $e_j^{(*)}$ for $j\in [n]$, which is a $\GL(n,\Q)$-homomorphism if $H(n+1)^{(*)}$ is considered as a $\GL(n,\Q)$-representation.
   Let $\bfH=\varinjlim_{n} H(n)=\bigcup_{n\ge 0} H(n)$, and $\bfH^*=\varinjlim_{n} (H(n)^*)=\bigcup_{n\ge 0}(H(n)^*)$.
   The group $\GL(\infty,\Q)$ acts on the tensor product $\bfH^{\otimes p}\otimes (\bfH^*)^{\otimes q}$ for any $p,q\ge 0$.
   We call a $\GL(\infty,\Q)$-representation \emph{algebraic} if it is a subquotient of a finite direct sum of tensor products $\bfH^{\otimes p}\otimes (\bfH^*)^{\otimes q}$ for $p,q\ge 0$.
   
   We define the contraction maps $c_{k,l}$ and the traceless part $T_{p,q}$ for $\GL(\infty,\Q)$ as in the case of $\GL(n,\Q)$.
   For any bipartition $\ul\lambda$ such that $|\lambda^{+}|=p$ and $|\lambda^{-}|=q$, 
   define a $\GL(\infty,\Q)$-representation $V_{\ul\lambda}$ as
   $$
   V_{\ul\lambda}=T_{p,q}\otimes_{\Q[\gpS_p\times\gpS_q]}(S^{\lambda^{+}}\otimes S^{\lambda^{-}}).
   $$
   Then as in \eqref{Tpq}, we have the following decomposition of the traceless part $T_{p,q}$ as $\GL(\infty,\Q)$-representations
   $$
    T_{p,q}=\bigoplus_{\substack{\ul\lambda=(\lambda^{+},\lambda^{-}) \\ |\lambda^{+}|=p,\;|\lambda^{-}|=q}} V_{\ul\lambda}\otimes (S^{\lambda^{+}}\otimes S^{\lambda^{-}}).
   $$
   Moreover, as in the case of $\GL(n,\Q)$, it is known that 
   irreducible algebraic representations of $\GL(\infty,\Q)$ are classified by bipartitions. (See Proposition 3.14 in \cite{SamSnowden}.)

   \begin{remark}
    Note that the category of algebraic $\GL(\infty,\Q)$-representations is not split semisimple.
    For example, we do not have a decomposition $\bfH\otimes \bfH^*\cong T_{1,1}\oplus\Q$ as a $\GL(\infty,\Q)$-representation.
    For each $n\ge 0$, we have a split short exact sequence of $\GL(n,\Q)$-representations
    \begin{gather*}
        0\to T_{1,1}(=V_{1,1})\to H(n)\otimes H(n)^*\overset{c_{1,1}}{\underset{s_n}{\rightleftarrows}} \Q\to 0,
    \end{gather*}
    where
    $s_n(1)=\sum_{j=1}^{n}e_j\otimes e_j^*.$
    Since the $\GL(n,\Q)$-module maps $s_n$ do not give rise to a morphism between sequences of $\GL(n,\Q)$-representations, however, we do not have a section in the case of $\GL(\infty,\Q)$-representations.
   \end{remark}

\section{Abelian cycles in $H_i(\IA_n,\Q)$}\label{Abeliancycles}
   Here we recall the notion of abelian cycles and construct abelian cycles in $H_i(\IA_n,\Q)$.
   
   \subsection{Definition of abelian cycles}
    Let $i\ge 1$ and $\Z^i=\bigoplus_{k=1}^i \Z z_k$, where $\{z_k\}_k$ is the canonical basis for $\Z^i$.
    Let $(\phi_1,\cdots,\phi_i)$ be an $i$-tuple of mutually commuting elements of $\IA_n$.
    Then we have a group homomorphism $\Z^i\to \IA_n$ mapping $z_k$ to $\phi_k$ for $1\le k\le i$,
    which induces a $\Q$-linear map $$H_i(\Z^i,\Q)\to H_i(\IA_n,\Q).$$
    We have $H_i(\Z^i,\Q)\cong H_i(T^i,\Q)\cong \Q$, where $T^i$ is the $i$-dimensional torus.
    Let $A(\phi_1,\cdots,\phi_i)\in H_i(\IA_n,\Q)$ denote the image of the fundamental class of $H_i(T^i,\Q)$ and call it the \emph{abelian cycle} determined by $(\phi_1,\cdots,\phi_i)$.
    
   \subsection{Abelian cycles in $H_i(\IA_n,\Q)$}\label{secabeliancycle} 
   
    Magnus's set of generators of $\IA_n$ is
    $$
   \{g_{a,b}\mid 1\le a,b \le n,\; a\ne b \}\cup\{f_{a,b,c}\mid 1\le a,b,c\le n,\; a< b,\; a\ne c\ne b\},
   $$
   where $g_{a,b}$ and $f_{a,b,c}$ are defined by 
    \begin{gather*}
        g_{a,b}(x_b)=x_a x_b x_a^{-1}, \quad g_{a,b}(x_d)=x_d \text{ for }d\neq b,\\
        f_{a,b,c}(x_c)=x_c[x_a,x_b], \quad f_{a,b,c}(x_d)=x_d \text{ for }d\neq c.
    \end{gather*}
   
    For $k>l$, let 
    $$
    h_{k,l}:=g_{k,l} g_{k,l+1}\cdots g_{k,k-1}\in\IA_n.
    $$
    Let $\Aut(F_n)(k,l)$ denote the image of the canonical injective map
    $$\Aut(F_{k-l})\hookrightarrow \Aut(F_n),\quad x_j\mapsto x_{j+l-1}.$$
    Since $h_{k,l}$ fixes $x_j$ for $j\in \{1,\cdots,l-1,k,\cdots,n\}$ and maps $x_j$ to $x_k x_j x_k^{-1}$ for $j\in \{l,l+1,\cdots,k-1\}$, it follows that $h_{k,l}$ commutes with any element of $\Aut(F_n)(k,l)$.
    
    For positive integers $a$ and $r$, set
    $$\mathbf{g}_{r,a}=(g_{a+1,a}=h_{a+1,a}, h_{a+2,a},\cdots, h_{a+r,a})$$
    and
    $$\mathbf{f}_{r,a}=(f_{a,a+1,a+2}, h_{a+3,a},h_{a+4,a},\cdots, h_{a+r+1,a}).$$
    
    \begin{lemma}\label{abeliancycle}
     For $n\ge a+r$, $\mathbf{g}_{r,a}$ is an $r$-tuple of mutually commuting elements.
     For $n\ge a+r+1$, $\mathbf{f}_{r,a}$ is an $r$-tuple of mutually commuting elements.
    \end{lemma}
    \begin{proof}
    For $\mathbf{g}_{r,a}$, it suffices to check that $h_{a+i,a}$ and $h_{a+j,a}$ commute for $1\le i<j\le r$.
    Since $h_{a+j,a}$ commutes with any elements of $\Aut(F_n)(a+j,a)$ and since we have $h_{a+i,a}\in \Aut(F_n)(a+j,a)$, it follows that $h_{a+j,a}$ commutes with $h_{a+i,a}$.
    
     We can also check that $\mathbf{f}_{r,a}$ is an $r$-tuple of mutually commuting elements by using the fact that 
     $f_{a,a+1,a+2}\in \Aut(F_n)(a+j,a)$ for any $3\le j\le r+1$.
    \end{proof}
    
    Let $i\ge 0$. 
    We call $(\mu,\nu)$ \emph{a pair of partitions of total size $i$}, denoted by $(\mu,\nu)\vdash i$, if $\mu$ and $\nu$ are partitions with $|\mu|+|\nu|=i$. 
    (Note that pairs of partitions are the same as bipartitions, but we regard them as different notions.)
    For $n\ge i+2l(\mu)+l(\nu)$,
    define an $i$-tuple $\mathbf{h}_{(\mu,\nu)}$ of elements of $\IA_n$ by
    $$
    \mathbf{h}_{(\mu,\nu)}=(\mathbf{f}_{\mu_1,s(1)},\cdots, \mathbf{f}_{\mu_{l(\mu)},s(l(\mu))}, \mathbf{g}_{\nu_1,t(1)},  \cdots, \mathbf{g}_{\nu_{l(\nu)},t(l(\nu))}),
    $$
    where 
    $$s(j)=1+\sum_{p=1}^{j-1}(\mu_p +2),\quad
    t(k)=1+|\mu|+2l(\mu)+\sum_{p=1}^{k-1}(\nu_p +1).$$
    Then we can check that $\mathbf{h}_{(\mu,\nu)}$ consists of mutually commuting elements.
    For example, we have
    \begin{gather*}
     \begin{split}
        \mathbf{h}_{(2, 0)}&=\mathbf{f}_{2,1}=(f_{1,2,3},\; h_{4,1}),\\
        \mathbf{h}_{(0,2)}&=\mathbf{g}_{2,1}=(h_{2,1},\; h_{3,1}),\\
        \mathbf{h}_{(1^2,0)}&=(\mathbf{f}_{1,1},\mathbf{f}_{1,4})=(f_{1,2,3},\; f_{4,5,6}),\\
        \mathbf{h}_{(0,1^2)}&=(\mathbf{g}_{1,1},\mathbf{g}_{1,3})=(h_{2,1},\; h_{4,3}),\\
        \mathbf{h}_{(1,1)}&=(\mathbf{f}_{1,1},\mathbf{g}_{1,4})=(f_{1,2,3},\; h_{5,4}).
     \end{split}
    \end{gather*}
    Then we obtain an abelian cycle corresponding to each pair of partitions.
    
    \begin{definition}\label{abeliancyclepairpart}
     Let $(\mu,\nu)\vdash i$ be a pair of partitions.
     For $n\ge i+2l(\mu)+l(\nu)$, define $\alpha_{(\mu,\nu)}=A(\mathbf{h}_{(\mu,\nu)})\in H_i(\IA_n,\Q)$ as the abelian cycle corresponding to the $i$-tuple $\mathbf{h}_{(\mu,\nu)}$ of mutually commuting elements.
    \end{definition}
    
    \begin{remark}
     The above construction gives an abelian cycle $\alpha_{(\mu,\nu)}$ in $H_i(\IA_n,\Z)$ for each $(\mu,\nu)\vdash i$.
    \end{remark}
   
\section{Non-triviality of $H^A_i(\IA_n,\Q)$}\label{Connectedpart}
 In this section, we construct two types of contraction maps each of which detects an irreducible $\GL(n,\Q)$-quotient representation of $H^A_i(\IA_n,\Q)$.
 
 Let 
 $$U=\Hom(H,\bigwedge^2 H)\cong \left(\bigwedge^2 H\right)\otimes H^{*}.$$
 For $a,b,c\in [n]$, 
 let 
 $$e_{a,b}^{c}:=(e_a\wedge e_b) \otimes e_c^*\in U.$$
 Then we have the following basis for $U$ induced by Magnus's set of generators of $\IA_n$
 \begin{gather*}
     \{e_{a,b}^{b}\mid 1\le a,b \le n,\; a\ne b \}\cup\{e_{a,b}^{c}\mid 1\le a,b,c\le n,\; a< b,\; a\ne c\ne b\}\\
   =\{e_{a,b}^{c}\mid 1\le a,b,c\le n,\; a< b\}.
 \end{gather*}
 Here, we have
 \begin{gather*}
     \tau(g_{a,b})=e_{a,b}^{b},\quad \tau(f_{a,b,c})=e_{a,b}^{c}
 \end{gather*}
for the Johnson homomorphism $\tau:\IA_n\to U$.

  \subsection{Contraction maps}\label{subseccontraction}
    For $i\ge 1$, let
    $$U_i=\Hom(H,\bigwedge^{i+1} H)\cong \left(\bigwedge^{i+1} H\right)\otimes H^*.$$
    Then we have $U=U_1$. Note that $U_i$ vanishes for $n\le i$.
    For $n\ge i+1$, we have a direct sum decomposition 
    $$U_i=U_i^{\tree}\oplus U_i^{\wheel},$$
    where
    $U_i^{\tree}$ denotes the subrepresentation of $U_i$ that is isomorphic to $V_{1^{i+1},1}$, and where $U_i^{\wheel}$ denotes the other subrepresentation that is isomorphic to $V_{1^i,0}$.
    In what follows, we assume $n\ge i+1$ and we identify $U_i^{\wheel}$ with $V_{1^i,0}$.

    \begin{remark}
    Here we can graphically describe an element of $U_i$ as follows.
    A basis element of $U_i$ corresponds to a tree with one root colored by a basis element of $H^*$ and $(i+1)$ leaves colored by basis elements of $H$.
    Then $U_i^{\tree}\cong V_{1^{i+1},1}$ is spanned by trees colored by traceless tensors in $H^{\otimes i+1}\otimes H^*$.
    On the other hand, $U_i^{\wheel}\cong V_{1^i,0}$ is spanned by trees colored by tensors with nontrivial image under a contraction map. The image under a contraction map can be represented by a wheel with $i$ legs colored by elements of $H$.
    (In Section \ref{KV}, we will use this graphical interpretation in the context of wheeled PROPs.)
    \end{remark}

    We will construct two contraction maps
    \begin{gather*}
     c^{\tree}_{i}: \bigwedge^{i} U\to U_i^{\tree}, \quad c^{\wheel}_{i}: \bigwedge^{i} U\to U_i^{\wheel}
    \end{gather*}
    as follows.
    Let 
    $$M_i=(H^{\otimes 2}\otimes H^{*})^{\otimes i}\cong H^{\otimes 2i}\otimes (H^*)^{\otimes i}.$$
    Let $\iota_{i}: \bigwedge^{i} U \hookrightarrow M_i$ be the canonical injective map defined by
    \begin{gather*}
      \iota_i(u_1\wedge\cdots\wedge u_i)=
      \sum_{\sigma\in \gpS_i} 
    \sgn(\sigma) \varphi(u_{\sigma(1)})\otimes\cdots\otimes \varphi(u_{\sigma(i)}),
    \end{gather*}
    where 
    $$
    \varphi:U \hookrightarrow H^{\otimes 2}\otimes H^*,\quad (a_1\wedge a_2)\otimes b^* \mapsto (a_1\otimes a_2- a_2\otimes a_1)\otimes b^*.
    $$
    Define a $\GL(n,\Q)$-homomorphism
    \begin{gather*}
     \tilde{c}_{i}: M_i\to U_i
    \end{gather*}
    by
    $$\tilde{c}_{i}\left(\bigotimes_{j=1}^i (a_{2j-1}\otimes a_{2j} \otimes b_j^*)\right)=\left(\prod_{j=2}^i b_j^*(a_{2(j-1)})\right) a_1\wedge a_3\wedge\cdots\wedge a_{2i-1}\wedge a_{2i} \otimes b_1^*$$
    and let
     \begin{gather*}
     \tilde{c}^{\tree}_{i}: M_i\xrightarrow{\tilde{c}_i} U_i\overset{\pr^{\tree}_i}{\twoheadrightarrow} U_i^{\tree}
    \end{gather*}
    be the composition, where the map $\pr^{\tree}_i$ is the canonical projection.
    Then set 
    \begin{gather*}
        c^{\tree}_{i}: \bigwedge^{i} U \xrightarrow{\iota_i} M_i\xrightarrow{\tilde{c}^{\tree}_i} U_i^{\tree}.
    \end{gather*}
    Explicitly, we have
    \begin{gather}\label{eqcitree}
    \begin{split}
        c^{\tree}_{i}(\bigwedge_{j=1}^{i} ((a_{2j-1}\wedge &a_{2j})\otimes b_j^*))\\
        =\pr^{\tree}_i \Bigg(\sum_{\substack{\sigma\in \gpS_i \\ \rho\in \prod_{j=1}^{i}\gpS_2}} &\sgn(\sigma)\sgn(\rho)
        \left(\prod_{j=2}^{i}b_{\sigma(j)}^*(a_{\rho(2\sigma(j-1))})\right) \\
        &a_{\rho(2\sigma(1)-1)}\wedge a_{\rho(2\sigma(2)-1)}\wedge \dots \wedge a_{\rho(2\sigma(i)-1)}\wedge a_{\rho(2\sigma(1))}\otimes b_{\sigma(1)}^* \Bigg).
    \end{split}
    \end{gather}
    Define another $\GL(n,\Q)$-homomorphism
    \begin{gather*}
     \tilde{c}^{\wheel}_{i}: M_i\to U_i^{\wheel} \cong V_{1^i,0}
    \end{gather*}
    by
    $$\tilde{c}^{\wheel}_{i}\left(\bigotimes_{j=1}^i (a_{2j-1}\otimes a_{2j} \otimes b_j^*)\right)= \left(\prod_{j=1}^i b_j^*(a_{2(j-1)})\right) a_1\wedge a_3\wedge\cdots\wedge a_{2i-1},$$
    where $a_{0}$ means $a_{2i}$, and let
    \begin{gather*}
        c^{\wheel}_{i}: \bigwedge^i U \xrightarrow{\iota_i} M_i\xrightarrow{\tilde{c}^{\wheel}_{i}} U_i^{\wheel}.
    \end{gather*}
     Explicitly, we have
    \begin{gather}\label{eqciwheel}
    \begin{split}
        &c^{\wheel}_{i}(\bigwedge_{j=1}^{i} ((a_{2j-1}\wedge a_{2j})\otimes b_j^*))\\
        &=\sum_{\substack{\sigma\in \gpS_i \\ \rho\in \prod_{j=1}^{i}\gpS_2}} \sgn(\sigma)\sgn(\rho)
        \left(\prod_{j=1}^{i}b_{\sigma(j)}^*(a_{\rho(2\sigma(j-1))})\right) 
        \bigwedge_{j=1}^{i} a_{\rho(2\sigma(j)-1)},
    \end{split}
    \end{gather}
    where $a_{\rho(2\sigma(0))}=a_{\rho(2\sigma(i))}$.
    
  \subsection{Non-trivial quotient representation of $H^A_i(\IA_n,\Q)$}
    Here we detect two irreducible quotient representations of $H^A_i(\IA_n,\Q)$ by using $c^{\tree}_{i}$ and $c^{\wheel}_{i}$.
    
    Recall that in Definition \ref{abeliancyclepairpart}, we defined the abelian cycle $\alpha_{(0,i)}\in H_i(\IA_n,\Q)$ corresponding to the pair of partition $(0,i)$ for $n\ge i+1$, and that we have the following $\GL(n,\Z)$-homomorphisms
    \begin{gather*}
    \xymatrix{
    H_i(\IA_n,\Q)\ar[r]^-{\tau_*} 
    &H_i(U,\Q)=\bigwedge^i U \ar[rr]^-{c^{\wheel}_i}\ar[rrd]^-{c^{\tree}_i}
    && U_i^{\wheel}\\
    &&& U_i^{\tree},
    }
    \end{gather*}
    where $\tau_* $ is induced by the Johnson homomorphism.
    
    \begin{lemma}\label{lemconnected}
     For $n\ge i+1$, we have
     \begin{gather*}
         c^{\wheel}_{i} \tau_* (\alpha_{(0,i)}) = i!\: e_2\wedge e_3\wedge\cdots \wedge e_{i+1}\in
        U_i^{\wheel}\setminus \{0\},
     \end{gather*}
     and 
     \begin{gather*}
         \tilde{c}_{i} \iota_{i} \tau_* (\alpha_{(0,i)})= (-1)^i (i+1)!\: e_1\wedge e_2\wedge\cdots \wedge e_{i+1}\otimes e_{1}^*.
     \end{gather*}
     Therefore, for $n\ge i+2$, we have
     \begin{gather*}
         ([\id]- [E_{1,n}])c^{\tree}_{i} \tau_* (\alpha_{(0,i)}) =(-1)^i (i+1)!\: e_1\wedge e_2\wedge \cdots \wedge e_{i+1}\otimes e_{n}^* \in U_i^{\tree}\setminus \{0\}.
     \end{gather*}
    \end{lemma}
    
    \begin{proof}
    We have 
    \begin{gather}\label{taualpha0i}
        \tau_* (\alpha_{(0,i)})=
        e_{2,1}^{1} \wedge (e_{3,1}^{1}+e_{3,2}^2)
        \wedge \cdots \wedge (\sum_{j=1}^i e_{i+1,j}^{j})
        \in \bigwedge^{i} U.
    \end{gather}
    
    First, we consider the image under $c^{\wheel}_{i}$.
    By \eqref{taualpha0i} and \eqref{eqciwheel}, 
    we have 
    \begin{gather*}
    \begin{split}
        c^{\wheel}_{i} \tau_* (\alpha_{(0,i)})
        &=  c^{\wheel}_{i} ( e_{2,1}^{1} \wedge e_{3,1}^{1}\wedge \cdots \wedge e_{i+1,1}^{1}) \\
        &= i!\: e_2\wedge e_3\wedge\cdots \wedge e_{i+1}\in U_i^{\wheel}\setminus \{0\}
    \end{split}
    \end{gather*}
    since the product $\prod_{j=1}^{i}b_{\sigma(j)}^*(a_{\rho(2\sigma(j-1))})$ in \eqref{eqciwheel} vanishes for any other wedges with at least one component $e_{k,j}^{j} \; (j\neq 1)$ and for $\rho\in \prod_{j=1}^{i}\gpS_2\setminus \{\id\}$.
    
    Next, we consider the image under the contraction map  $c^{\tree}_{i}$.
    For $K\in [n]^{3i}$ regarded as a map $[3i]\to [n]$, let $e_{K}=\bigotimes_{j=1}^i (e_{K(j)}\otimes e_{K(2i+1-j)}\otimes e_{K(2i+j)}^*)$.
    Then $\{e_{K} \mid K\in [n]^{3i} \}$
    forms a basis for $M_i$.
    We write 
    \begin{gather*}
         \iota_{i} \tau_* (\alpha_{(0,i)})=\sum_{K\in [n]^{3i}} a_K e_K \in M_i,
    \end{gather*}
    where $a_K\in \Z$.
    Then by \eqref{taualpha0i}, we can easily check the following properties of the coefficient $a_K$
    \begin{itemize}
        \item for each $K\in [n]^{3i}$, we have $a_K\in \{0,1,-1\}$.
        \item for $K\notin [i+1]^{3i}$, we have $a_K=0$.
        \item for $K\in [i+1]^{3i}$ such that $K(2i+j)\neq \min\{K(j),K(2i+1-j)\}$ for some $1\le j\le i$, we have $a_K=0$.
    \end{itemize}
    
    By the second property of $a_K$ and the definition of $\tilde{c}_i$, if we have $\tilde{c}_{i}(a_K e_K)\in U_i\setminus \{0\}$, then there exists a permutation $\sigma\in \gpS_{i+1}$ such that for any $j\in [i+1]$, we have $K(j)=\sigma(j)$.
    
    \begin{claim}\label{exist}
     For each $\sigma\in \gpS_{i+1}$, there uniquely exists $K\in [i+1]^{3i}$
     such that 
     \begin{itemize}
         \item $K(j)=\sigma(j)$ for $j\in [i+1]$
         \item $\tilde{c}_{i}(a_K e_K)\neq 0$.
     \end{itemize} 
     In particular, we have 
     $$\tilde{c}_{i}(a_K e_K)= a_K (e_{\sigma(1)}\wedge e_{\sigma(2)}\wedge \cdots\wedge e_{\sigma(i+1)})\otimes e_1^{*}.$$
    \end{claim}
    
    \begin{proof}[Proof of Claim \ref{exist}]
    Define $K\in [i+1]^{3i}$ by
    \begin{gather*}
        K(j)=
        \begin{cases}
        \sigma(j) & (1\le j\le i+1)\\
        \min \{K(2i+2-j), \cdots , K(i+1)\} & (i+2\le j \le 2i)\\
        \min \{K(j-2i),\cdots , K(i+1) \} & (2i+1\le j\le 3i).
        \end{cases}
    \end{gather*}
    Then we can check that 
    \begin{gather*}
        \tilde{c}_{i}(a_K e_K)=a_K (e_{\sigma(1)}\wedge e_{\sigma(2)}\wedge \cdots\wedge e_{\sigma(i+1)})\otimes e_1^{*}
    \end{gather*}
    since we have $K(2i+1)=\min \{K(1),\cdots , K(i+1)\}=1$.
    The uniqueness follows from the third property of $a_{K}$ and the construction of $K$.
    \end{proof}
    
    \begin{claim}\label{sgn}
     For each $\sigma\in \gpS_{i+1}$, let $K(\sigma)$ denote the unique element of $[i+1]^{3i}$ that we constructed in the proof of Claim \ref{exist}.
     Then we have $a_{K(\sigma)}=(-1)^i \sgn (\sigma)$.
    \end{claim}
    
    \begin{proof}[Proof of Claim \ref{sgn}]
     We use the induction on $i$.
     For $i=1$, we have 
     $$a_{K(\id)}=-1=(-1)^1 \sgn (\id), \quad a_{K((12))}=1=(-1)^1 \sgn((12)).$$
     Assume that the statement holds for $i\ge 1$.
     For any $\sigma\in \gpS_{i+1}$, define $\tau\in \gpS_{i+1}$ as
     \begin{gather*}
         \tau(j)=
         \begin{cases}
         \sigma(1) & (j=1)\\
         j-1 & (2\le j\le \sigma(1))\\
         j &(\sigma(1)+1\le j\le i+1).
         \end{cases}
     \end{gather*}
     Let $\ti{\tau}=\tau^{-1}\sigma$.
     Then $\ti{\tau}$ maps $1$ to $1$, so we regard $\ti{\tau}\in \gpS_i$ by identifying $\{2,\cdots,i+1\}$ with $\{1,\cdots,i\}$ via the order preserving bijection.
     We will prove that 
     \begin{gather}\label{aKsigma}
         a_{K(\sigma)}=(-1)^{\sigma(1)}a_{K(\ti{\tau})}.
     \end{gather}
     If $\sigma(1)=1$, then we have $K(1)=1, K(2i)=2, K(2i+1)=1$ by the definition of $K$ in Claim \ref{exist}.
     Therefore, by \eqref{taualpha0i}, we have $a_{K(\sigma)}=(-1)a_{K(\ti{\tau})}$.
     In a similar way, if $\sigma(1)\neq 1$, then we have $K(1)=\sigma(1), K(2i)=1, K(2i+1)=1$.
     Therefore, we have $a_{K(\sigma)}=(-1)^{(\sigma(1)-2)}a_{K(\ti{\tau})}$, where the sign $(-1)^{(\sigma(1)-2)}$ comes from the permutation which brings the factor $e_{\sigma(1),1}^1$ in a summand of \eqref{taualpha0i} to the first factor.
     Hence, we have \eqref{aKsigma}.
     
     Since we have $\sgn(\tau)=(-1)^{\sigma(1)-1}$, by induction hypothesis, we have 
     \begin{gather*}
      \begin{split}
         (-1)^{\sigma(1)} a_{K(\ti{\tau})}
          &=(-1)^{\sigma(1)}(-1)^i \sgn(\ti{\tau})
          =(-1)^{i+1} \sgn(\tau) \sgn(\ti{\tau})
          =(-1)^{i+1}\sgn(\sigma),
      \end{split}   
     \end{gather*}
     which completes the proof.
    \end{proof}
    
    By using Claims \ref{exist} and \ref{sgn}, we have 
    \begin{gather*}
     \begin{split}
        \tilde{c}_{i} \iota_{i} \tau_* (\alpha_{(0,i)}) 
        &=\sum_{K\in [n]^{3i}} \tilde{c}_{i} (a_K e_K)\\
        &=\sum_{\sigma\in \gpS_{i+1}} a_{K(\sigma)} (e_{\sigma(1)}\wedge e_{\sigma(2)}\wedge \cdots\wedge e_{\sigma(i+1)})\otimes e_1^{*}\\
        &=\sum_{\sigma\in \gpS_{i+1}} a_{K(\sigma)} \sgn(\sigma) (e_1\wedge e_2\wedge\cdots\wedge e_{i+1})\otimes e_1^{*}\\
        &=(-1)^i (i+1)!\: e_1\wedge e_2\wedge\cdots \wedge e_{i+1}\otimes e_{1}^*.
     \end{split}
    \end{gather*}
    Since for $n\ge i+2$, we have 
    $$([\id]- [E_{1,n}])\tilde{c}_{i} \iota_{i} \tau_* (\alpha_{(0,i)}) =(-1)^i (i+1)!\: e_1\wedge e_2\wedge \cdots \wedge e_{i+1}\otimes e_{n}^* \in U_i^{\tree}\setminus \{0\},$$
    the proof is complete.
    \end{proof}
    
    \begin{theorem}\label{nontriviality}
     For $n\ge i+1$, we have a surjective $\GL(n,\Q)$-homomorphism
     $$H^A_i(\IA_n,\Q) \twoheadrightarrow \Hom(H,\bigwedge^{i+1} H)\cong V_{1^i,0} \oplus V_{1^{i+1},1}.$$
     Therefore, $H^A_i(\IA_n,\Q)$ includes a $\GL(n,\Q)$-subrepresentation which is isomorphic to $\Hom(H,\bigwedge^{i+1} H)$.
    \end{theorem}
    
    \begin{proof}
    This directly follows from Lemma \ref{lemconnected}.
    \end{proof}
    
 \subsection{A representation-stable quotient of $H^A_i(\IA_n,\Q)$}
 
  Here we study the sequence $H_i^A(\IA_*,\Q)$ of $H^A_i(\IA_n,\Q)$.
  
  We have a canonical inclusion map $\IA_n\hookrightarrow \IA_{n+1}$ sending $f\in\IA_n$ to $f\in \IA_{n+1}$ which is the same as $f$ on $F_n\subset F_{n+1}$ and fixes $x_{n+1}$.
  Let $H^A_i(\IA_*,\Q)$ denote the sequence of group homomorphisms $H^A_i(\IA_n,\Q)\to H^A_i(\IA_{n+1},\Q)$ induced by the inclusion maps.
  Since the inclusion map $\IA_n\hookrightarrow \IA_{n+1}$ sends
  $g_{a,b}\in \IA_n$ to $g_{a,b}\in \IA_{n+1}$ and $f_{a,b,c}\in \IA_n$ to $f_{a,b,c}\in \IA_{n+1}$,
  the abelian cycle $\alpha_{(0,i)}\in H_i(\IA_n,\Q)$ is sent to $\alpha_{(0,i)}\in H_i(\IA_{n+1},\Q)$.
  Therefore, the element $\tau_*(\alpha_{(0,i)})\in H^A_i(\IA_n,\Q)$ is sent to $\tau_*(\alpha_{(0,i)})\in H^A_i(\IA_{n+1},\Q)$.
  
  The canonical inclusion maps $H(n)\hookrightarrow H(n+1)$ and $H(n)^*\hookrightarrow H(n+1)^*$ that we have observed in Section \ref{prestable} send $e_1\wedge e_2\wedge\cdots \wedge e_{i+1}\otimes e_1^*\in \bigwedge^{i+1} H(n)\otimes H(n)^*$ to $e_1\wedge e_2\wedge\cdots\wedge e_{i+1}\otimes e_1^*\in \bigwedge^{i+1} H(n+1) \otimes H(n+1)^*$.
  Let $\left(\bigwedge^{i+1} H\right)\otimes H^*$ denote the sequence of group homomorphisms $\left(\bigwedge^{i+1} H(n)\right)\otimes H(n)^*\hookrightarrow \left(\bigwedge^{i+1} H(n+1)\right)\otimes H(n+1)^*$.
  Then it is obvious that $\left(\bigwedge^{i+1} H\right)\otimes H^*$ is representation stable.

  Let $\{V_n,\phi_n:V_n\to V_{n+1}\}$ and $\{W_n,\psi_n:W_n\to W_{n+1}\}$ be sequences of algebraic $\GL(n,\Q)$-representations.
  By a morphism from $\{V_n,\phi_n\}$ to $\{W_n,\psi_n\}$, we mean a family $\{f_n\}$ of $\GL(n,\Q)$-module maps $f_n:V_n\to W_n$ satisfying $\psi_n f_n=f_{n+1}\phi_n$.
  Then we have the following proposition by Lemma \ref{lemconnected} and Theorem \ref{nontriviality}.
  \begin{proposition}\label{representationstabilityofconnectedpart}
  We have a surjective morphism of sequences of algebraic $\GL(n,\Q)$-representations
  $$H^A_i(\IA_*,\Q)\twoheadrightarrow\left(\bigwedge^{i+1} H\right) \otimes H^*,$$
  where $\left(\bigwedge^{i+1} H\right) \otimes H^*$ is representation stable.
  \end{proposition}
  
  By the representation stability, we can consider an analogue of Theorem \ref{nontriviality} in the category of algebraic $\GL(\infty,\Q)$-representations.
  Let $\IA_{\infty}=\varinjlim_{n}\IA_n$ denote the direct limit.
  Both group homology and Albanese homology preserve direct limits (see Proposition \ref{preservedirectlimit}), that is, we have
  $$
   H_i(\IA_{\infty},\Q)\cong \varinjlim_{n}H_i(\IA_n,\Q),\quad  H^A_i(\IA_{\infty},\Q)\cong \varinjlim_{n}H^A_i(\IA_n,\Q).
  $$
  Since $\alpha_{(0,i)}$ is preserved under $H_i(\IA_n,\Q)\to H_i(\IA_{n+1},\Q)$, we obtain a cycle $\alpha_{(0,i)}$ in $H_i(\IA_{\infty},\Q)$ and an element $\tau_*(\alpha_{(0,i)})\in H^A_i(\IA_{\infty},\Q)$.
  Therefore, by Lemma \ref{lemconnected}, we have a $\GL(\infty,\Q)$-homomorphism
  $$\tilde{c}_{i} \iota_{i}: H^A_i(\IA_{\infty},\Q) \to \left(\bigwedge^{i+1} \bfH\right)\otimes\bfH^*$$
  such that
  \begin{gather*}
    \tilde{c}_{i} \iota_{i} \tau_* (\alpha_{(0,i)})= (-1)^i (i+1)!\: e_1\wedge e_2\wedge\cdots \wedge e_{i+1}\otimes e_{1}^*.
  \end{gather*}
  
  \begin{corollary}\label{infinitedimensional}
  The $\GL(\infty,\Q)$-homomorphism
  $$
  \tilde{c}_{i} \iota_{i}: H^A_i(\IA_{\infty},\Q)\to\left(\bigwedge^{i+1} \bfH\right)\otimes\bfH^*
  $$
  is surjective.
  In particular, for each $i\ge 1$, $H_i(\IA_{\infty},\Q)$ is infinite dimensional.
  \end{corollary}
  
  \begin{remark}
   We have a non-split exact sequence
   $$0\to V_{1^{i+1},1}\to \left(\bigwedge^{i+1} \bfH\right)\otimes\bfH^*\to V_{1^i,0}\to 0,$$
   which generalizes the observation in Section \ref{prestable} that $\bfH\otimes \bfH^*\to \Q$ does not split.
   Therefore, $H^A_i(\IA_{\infty},\Q)$ is not semisimple.
  \end{remark}

\section{The traceless part of $H^A_i(\IA_n,\Q)$}\label{Tracelesspart}
   Let $H_i(U,\Q)^{\tl}$ denote the \emph{traceless part} of $H_i(U,\Q)=\bigwedge^i U$ defined by
   $$H_i(U,\Q)^{\tl}=\pi(T_{2i,i}),$$
   where $\pi: H^{\otimes 2i}\otimes (H^{*})^{\otimes i}\twoheadrightarrow \bigwedge^{i} U$ is the canonical projection defined by
   $$
   \left(\bigotimes_{j=1}^i (a_j\otimes b_j)\right) \otimes \left(\bigotimes_{j=1}^i c_j^*\right)\mapsto \bigwedge_{j=1}^i ((a_j\wedge b_j)\otimes c_j^*).
   $$
   In this section, we show that $H_i(U,\Q)^{\tl}$ is included in $H^A_i(\IA_n,\Q)$.
  
   \begin{theorem}\label{tracelessimage}
   Let $n\ge 3i$.
   We have $H_i(U,\Q)^{\tl}\subset H^A_i(\IA_n,\Q)$.
   \end{theorem}
   
   \begin{proof}
   By Corollary \ref{projgen}, the traceless part $H_i(U,\Q)^{\tl}$ is generated by 
   $$\pi(e_{2i,i})=e_{1,2}^{n}\wedge e_{3,4}^{n-1}\wedge \dots \wedge e_{2i-1,2i}^{n-i+1}
   $$
   as a $\GL(n,\Q)$-representation.
   Therefore, we have to show that $\pi(e_{2i,i})\in H^A_i(\IA_n,\Q)$.
   For this purpose, we use the abelian cycle $\alpha_{(0,1^i)}\in H_i(\IA_n,\Q)$ that we defined in Definition \ref{abeliancyclepairpart}.
   Since $\alpha_{(0,1^i)}$ corresponds to the $i$-tuple $(g_{2,1},g_{4,3},\dots, g_{2i,2i-1})$, we have
   \begin{gather*}
       \tau_*(\alpha_{(0,1^i)})=e_{2,1}^1 \wedge e_{4,3}^3 \wedge\dots\wedge e_{2i,2i-1}^{2i-1}
       \in H^A_i(\IA_n,\Q).
   \end{gather*}
   We can transform $\tau_*(\alpha_{(0,1^i)})$ into $\pi(e_{2i,i})$ by the action of $$\prod_{j=1}^i ([E_{2j-1,n-j+1}]-[\id])\in \Q[\GL(n,\Q)].$$
   Therefore, we have $\pi(e_{2i,i})\in H^A_i(\IA_n,\Q)$.
   \end{proof}
   
   Church--Ellenberg--Farb \cite[Theorem 7.2.3]{CEF} proved that for each $i\ge 0$,
   there exists a polynomial $P_i(T)$ of degree $\le 3i$ such that $\dim_{\Q}(H^A_i(\IA_n,\Q))=P_i(n)$ for sufficiently large $n$ with respect to $i$. 
   On the other hand, we have the following decomposition of $H_i(U,\Q)$
   \begin{gather}\label{decompositionofwedgeU}
       H_i(U,\Q)=\bigwedge^i U=\bigoplus_{\lambda \vdash i} (\repS_{\lambda}(\repS_{1^2} H)\otimes \repS_{\lambda'}(H^*)),
   \end{gather}
   where $\repS_{\lambda}$ denotes the \emph{Schur functor}, which sends a vector space $V$ to $\repS_{\lambda}(V)=V_{\lambda}$, and
   where $\lambda'$ denotes the \emph{conjugate} partition $\lambda'$ to $\lambda$, which is obtained from the Young diagram corresponding to $\lambda$ by interchanging rows and columns.
   (See \cite[Exercise 6.11]{FH}.)
   In particular, we have $V_{1^{2i},1^{i}} \subset H_i(U,\Q)^{\tl}$.
   By \eqref{decompositionofwedgeU}, we obtain $\dim_{\Q}(H_i(U,\Q)^{\tl})=Q_i(n)$ for $n\ge 3i$, where $Q_i(T)$ is a polynomial of degree $3i$. 
   By Theorem \ref{tracelessimage}, we obtain the following theorem.
   
   \begin{theorem}\label{dimension}
   For $n\ge 3i$, we have $\dim_{\Q}(H^A_i(\IA_n,\Q))\ge Q_i(n)$.
   Moreover, there exists a polynomial $P_i(T)$ of degree $3i$ such that $\dim_{\Q}(H^A_i(\IA_n,\Q))=P_i(n)$ for sufficiently large $n$ with respect to $i$.
   \end{theorem}

\section{The structure of $H^A_i(\IA_n,\Q)$}\label{Generalpart}

In this section, we introduce a graded $\GL(n,\Q)$-representation $W_*$.
The degree $i$ part $W_i$ is graded over the set of pairs of partitions of total size $i$. 
We show that $W_i$ is a subquotient of $H^A_i(\IA_n,\Q)$.
The non-trivial quotient representations that we detected in Theorem \ref{nontriviality} are the degree $(i,0)$ and $(0,i)$ parts, and the traceless part $H_i(U,\Q)^{\tl}$ is the direct sum of the degree $(1^j,1^k)$ parts for $j+k=i$.

 \subsection{The subquotient representation $W_i$ of $H^A_i(\IA_n,\Q)$}
 
Recall that for $n\ge i+1$, we have
$$U_i=\Hom(H,\bigwedge^{i+1} H)=U_i^{\tree}\oplus U_i^{\wheel}\cong  V_{1^{i+1},1}\oplus V_{1^i,0}.$$
Let $U_*=\bigoplus_{i\ge 1} U_i$, which is a graded algebraic $\GL(n,\Q)$-representation.
Define $W_*=\wti{S}^* (U_*)$ as the traceless part of the graded-symmetric algebra $S^* (U_*)$ of $U_*$, which we defined in Section \ref{tracelesswedge}.
We can also construct $W_*$ by using an operad $\calC om$ of non-unital commutative algebras as we explain in Section \ref{KV}.
 
 For $i\ge 0$, let $(\mu,\nu)\vdash i$ be a pair of partitions.
 If we write 
 $$\mu=(\mu_1^{k_1'},\cdots,\mu_r^{k_r'}),\quad \nu=(\nu_1^{k_1''},\cdots,\nu_s^{k_s''}),$$
 then it means that $\mu_1>\mu_2>\cdots>\mu_r$ and $\mu_j$ appears $k_j'>0$ times in $\mu$ and that $\nu_1>\nu_2>\cdots>\nu_s$ and $\nu_j$ appears $k_j''>0$ times in $\nu$.
 Then the length of $\mu$ (resp. $\nu$) is $l(\mu)=\sum_{j=1}^r k_j'$ (resp. $l(\nu)=\sum_{j=1}^s k_j''$).
 Let
 $$U^{\tree}_{\mu}=\bigotimes_{j=1}^r S^{k_j'} (U_{\mu_j}^{\tree}),\quad
 U^{\wheel}_{\nu}=\bigotimes_{j=1}^s S^{k_j''} (U_{\nu_j}^{\wheel}).$$
 Then we have 
 $$S^*(U_*)= \bigoplus_{i\ge 0} \bigoplus_{(\mu,\nu)\vdash i} U^{\tree}_{\mu}\otimes U^{\wheel}_{\nu}.$$
 Since 
 $W_i=\wti{S}^* (U_*)_i$ by the definition of $W_*$,
 we have
  $$W_i= \bigoplus_{(\mu,\nu)\vdash i} W(\mu,\nu)\subset S^*(U^*)_i,$$
  where
 $$W(\mu,\nu)=
  \wti{\bigotimes}_{j=1}^r \wti{S}^{k_j'}\left(U_{\mu_j}^{\tree}\right)  \wti{\otimes}\;
  \wti{\bigotimes}_{j=1}^s \wti{S}^{k_j''}\left(U_{\nu_j}^{\wheel}\right)\subset U^{\tree}_{\mu}\otimes U^{\wheel}_{\nu}.$$
 
 For example, we have for $n\ge 3$,
 $$W_1=W(1,0)\oplus W(0,1)\cong V_{1^2,1}\oplus V_{1,0},$$
 and we have for $n\ge 6$, 
 \begin{gather*}
 \begin{split}
      W_2&=W(2,0)\oplus W(0,2)\oplus W(1^2,0)\oplus W(1,1)\oplus W(0,1^2)\\
      &\cong V_{1^3,1}\oplus V_{1^2,0} \oplus (V_{1^4,1^2}\oplus V_{2 1^2,2}\oplus V_{2^2,1^2})\oplus (V_{21,1}\oplus V_{1^3,1})\oplus V_{1^2,0}.
 \end{split}
 \end{gather*}

 Here, we observe that the subrepresentation $W(\mu,\nu)$ of $W_i$ is an image of the traceless part $T_{i+l(\mu),l(\mu)}$ under a projection.
 Let 
 $$\pi_j^{\mu}: H^{\otimes k_j'(1+\mu_j)}\otimes (H^*)^{\otimes k_j'}=(H^{\otimes (1+\mu_j)}\otimes H^*)^{\otimes k_j'}\twoheadrightarrow (V_{1^{1+\mu_j}, 1})^{\otimes k_j'}\twoheadrightarrow S^{k_j'} (V_{1^{1+\mu_j}, 1})
 $$
 and 
 $$
 \pi_j^{\nu}:  H^{\otimes k_j''\nu_j}= (H^{\otimes \nu_j})^{\otimes k_j''} \twoheadrightarrow (V_{1^{\nu_j}})^{\otimes k_j''}\twoheadrightarrow S^{k_j''} (V_{1^{\nu_j}})
 $$
 be the composition of two canonical projections, and let
 \begin{gather*}
  \begin{split}
    \pi^{(\mu,\nu)}=\left(\bigotimes_{j=1}^r \pi_j^{\mu} \right) \otimes \left(\bigotimes_{j=1}^s \pi_j^{\nu}\right):
     H^{\otimes (i+l(\mu))} \otimes (H^*)^{\otimes l(\mu)}\to 
     U^{\tree}_{\mu}\otimes U^{\wheel}_{\nu}.
  \end{split}
 \end{gather*}
 Then we have
 $$W(\mu,\nu)=\pi^{(\mu,\nu)}(T_{i+l(\mu), l(\mu)}).$$
 
 Define 
 \begin{gather*}
     \begin{split}
         &F_{(\mu,\nu)}:
         H_i(U,\Q)=\bigwedge^i U
         \to
         U^{\tree}_{\mu}\otimes U^{\wheel}_{\nu}
     \end{split}
 \end{gather*}
 to be the composition of the inclusion 
 $\iota_{i}: \bigwedge^i U\hookrightarrow M_i$ 
 and 
 $$\left(\bigotimes_{j=1}^r \frac{1}{k_j'!} \prod^{k_j'}\tilde{c}^{\tree}_{\mu_j} \right)
 \otimes 
 \left(\bigotimes_{j=1}^s \frac{1}{k_j''!} \prod^{k_j''}\tilde{c}^{\wheel}_{\nu_j}
 \right)
 : M_i
 \to 
 U^{\tree}_{\mu}\otimes U^{\wheel}_{\nu}.$$

 Then we have a $\GL(n,\Q)$-homomorphism
 \begin{gather*}
     F_i:= \bigoplus_{(\mu,\nu)\vdash i} F_{(\mu,\nu)} : H_i(U,\Q)\to S^*(U_*)_i= \bigoplus_{(\mu,\nu)\vdash i} 
     U^{\tree}_{\mu}\otimes U^{\wheel}_{\nu}.
 \end{gather*}
 
 \begin{theorem}\label{Johnsonpart}
  For $n\ge 3i$, we have 
  $$F_i (H^A_i(\IA_n,\Q)) \supset W_i.$$
 \end{theorem}
  
  We will prove Theorem \ref{Johnsonpart} in the rest of this section by using the abelian cycles $\alpha_{(\mu,\nu)}$.
  
  It seems natural to make the following conjecture, which holds for $i=1$ \cite{CP, Farb, Kawazumi}, for $i=2$ \cite{Pettet} and for $i=3$ as we will observe in Theorem  \ref{thirdAlbanese}.
  
  \begin{conjecture}\label{conjectureAlbanese}
  For $n\ge 3i$, $F_i$ restricts to a $\GL(n,\Q)$-isomorphism
  $$
   F_i : H^A_i(\IA_n,\Q) \xrightarrow{\cong}  W_i.
  $$
  \end{conjecture}
  
  Conjecture \ref{conjectureAlbanese} implies that as a $\GL(n,\Q)$-representation, the Albanese homology $H^A_i(\IA_n,\Q)$ is generated by the images of abelian cycles of $H_i(\IA_n,\Q)$ under $\tau_*$.

  We will introduce the following relaxed version of Conjecture \ref{conjectureAlbanese}.
  
  \begin{conjecture}\label{conjectureAlbaneserelaxed}
   We have a $\GL(n,\Q)$-isomorphism
  $$
   H^A_i(\IA_n,\Q) \cong  W_i
  $$
  for sufficiently large $n$ with respect to $i$.
  \end{conjecture}
  
 We would like to realize $W_i$ as a subrepresentation of $H^A_i(\IA_n,\Q)$.
 
 \begin{problem}
   Construct a lift
   $$
   W_i\to H^A_i(\IA_n,\Q)
   $$
   of the inclusion map $W_i\hookrightarrow F_i(H^A_i(\IA_n,\Q))$
   along $H^A_i(\IA_n,\Q)\xto{F_i} F_i(H^A_i(\IA_n,\Q))$
   as $\GL(n,\Q)$-representations for sufficiently large $n$ with respect to $i$.
 \end{problem}
 
   We would like to realize the Albanese homology $H^A_i(\IA_n,\Q)$ as a subrepresentation of $H_i(\IA_n,\Q)$.
 
 \begin{problem}
   Construct a section of 
   $$
     H_i(\IA_n,\Q)\overset{\tau_*}{\twoheadrightarrow} H^A_i(\IA_n,\Q) 
   $$
   as $\GL(n,\Z)$-representations for sufficiently large $n$ with respect to $i$.
 \end{problem}

 \subsection{Computation of the contraction maps}
 
  In a way similar to Lemma \ref{lemconnected}, we obtain the following lemma.
 
 \begin{lemma}\label{lemcontractioncomputation}
  For $n\ge i+2$, we have
     \begin{gather*}
        c^{\wheel}_{i} \tau_* (\alpha_{(i,0)}) = 0,
     \end{gather*}
     and
     \begin{gather*}
        c^{\tree}_{i} \tau_* (\alpha_{(i,0)}) = (-1)^{i-1}(i+1)! \: e_1\wedge e_2\wedge e_4\wedge \cdots \wedge e_{i+2}\otimes e_{3}^* \in U_i^{\tree}\setminus \{0\}.
     \end{gather*}
 \end{lemma}
 
 \begin{proof}
  For the abelian cycle $\alpha_{(i,0)}$, we have 
    \begin{gather*}
        \tau_* (\alpha_{(i,0)})=
        e_{1,2}^{3}
        \wedge (\sum_{j=1}^{3} e_{4,j}^{j}) 
        \wedge \cdots \wedge (\sum_{j=1}^{i+1} e_{i+2,j}^{j}).
    \end{gather*}

  By using \eqref{eqciwheel}, we can check that
  \begin{gather}\label{vanishingunderwheel}
      c^{\wheel}_{i}(e_{a_1,b_1}^{d_1} \wedge  e_{a_2,b_2}^{b_2}\wedge\cdots \wedge e_{a_i,b_i}^{b_i})=0
  \end{gather}
  if $d_1\notin \{a_1,\cdots, a_i, b_1\}$.
  Therefore, we have $c^{\wheel}_{i} \tau_* (\alpha_{(i,0)}) = 0$.
    
  We will compute $c^{\tree}_3\tau_*(\alpha_{(3,0)})$ for example.
  We have
  \begin{gather*}
      c^{\tree}_3(e_{1,2}^{3} \wedge e_{4,i}^{i}\wedge e_{5,j}^{j})
      =\begin{cases}
          0 & (i,j)=(1,2),(2,1),(3,4),\\
          4 e_1\wedge e_2\wedge e_4\wedge e_5 \otimes e_3^*,& (i,j)=(1,1),(2,2),(3,3),\\
          2 e_1\wedge e_2\wedge e_4\wedge e_5 \otimes e_3^*,& (\text{otherwise})
      \end{cases}
  \end{gather*}
  and thus we have 
  \begin{gather*}
  \begin{split}
      c^{\tree}_3\tau_*(\alpha_{(3,0)})
      &=c^{\tree}_3(e_{1,2}^{3} \wedge  (e_{4,1}^{1}+e_{4,2}^2+ e_{4,3}^3)\wedge (e_{5,1}^{1}+e_{5,2}^2+e_{5,3}^3+e_{5,4}^{4}))\\
      &=24 e_1\wedge e_2\wedge e_4\wedge e_5 \otimes e_3^*.
  \end{split}
  \end{gather*}
  In general, it is easy to see that 
  $$c^{\tree}_{i} (e_{1,2}^{3} \wedge  e_{4,b_2}^{b_2}\wedge\cdots \wedge e_{i+2,b_i}^{b_i}) \in \Z e_1\wedge e_2\wedge e_4\wedge \cdots \wedge e_{i+2}\otimes e_{3}^*$$
  by \eqref{eqcitree}.
  The computation of the coefficient in $c^{\tree}_{i}(\tau_* (\alpha_{(i,0)}))$ is similar to that of the coefficient in $c^{\tree}_{i}(\tau_* (\alpha_{(0,i)}))$ in Lemma \ref{lemconnected}.
 \end{proof}
 
 We call elements of the form
 $$
  e_{a_1,b_1}^{d_1}\wedge e_{a_2,b_2}^{d_2}\wedge\cdots \wedge e_{a_i,b_i}^{d_i}\in \bigwedge^i U
 $$
 the \emph{basis elements} of $\bigwedge^i U$.
 Two basis elements $x=e_{a_1,b_1}^{d_1}\wedge e_{a_2,b_2}^{d_2}\wedge\cdots \wedge e_{a_i,b_i}^{d_i}$ of $\bigwedge^i U$ and $y=e_{p_1,q_1}^{r_1}\wedge e_{p_2,q_2}^{r_2}\wedge\cdots \wedge e_{p_j,q_j}^{r_j}$ of $\bigwedge^j U$ are said to be \emph{disjoint} if $\{a_k,b_k,d_k\}_{k=1}^{i}\cap \{p_l,q_l,r_l\}_{l=1}^{j}=\emptyset$.
 We also call linear sums $x=\sum_{k=1}^{m} a_k x_k\in \bigwedge^i U$ and $y=\sum_{l=1}^{n} b_l y_l\in \bigwedge^j U$ of basis elements \emph{disjoint} if any pair $(x_k,y_l)$ of basis elements is disjoint.
 If basis elements $x\in \bigwedge^i U$ and $y\in \bigwedge^j U$ are disjoint, then it is easy to see that \begin{equation}\label{disjoint}
     c^{\wheel}_{i+j}(x\wedge y)=c^{\tree}_{i+j}(x\wedge y)=0.
 \end{equation}
 Note that we have
 \begin{gather*}
     c^{\wheel}_{i}(\tau_*(\alpha_{(\mu,\nu)}))=c^{\tree}_{i}(\tau_*(\alpha_{(\mu,\nu)}))=0
 \end{gather*}
 for $(\mu,\nu)\vdash i$ with $l(\mu)+l(\nu)\ge 2$ by \eqref{disjoint}.
 
 Let $P_i$ denote the set of pairs of partitions of total size $i$.
 For $l\in\{0,\cdots,i\}$, let $P_i^l$ denote the subset of $P_i$ consisting of elements $(\mu,\nu)$ with $l(\mu)=l$.
 In what follows, we consider a sufficient condition for $(\mu,\nu), (\xi,\eta)\in P_i^l$ that $F_{(\mu,\nu)}(\tau_*(\alpha_{(\xi,\eta)}))$ vanishes.
 In order to state the condition, we will put a partial order on $P_i^l$ as follows.
 For $(\xi,\eta), (\mu,\nu) \in P_i^l$, we write 
 $(\xi,\eta)\ge (\mu,\nu)$
 if there exist a decomposition $L_1\sqcup\cdots\sqcup L_{l+l(\eta)} = \{1,\cdots,l(\nu)\}$ and $\sigma\in \gpS_l$ such that 
 \begin{gather}\label{partialorder1}
      \xi_j-\mu_{\sigma(j)}=\sum_{k\in L_j}\nu_k \quad(1\le j\le l), 
 \end{gather}
 \begin{gather}\label{partialorder2}
     \eta_j=\sum_{k\in L_{l+j}}\nu_k\quad (1\le j\le l(\eta)).
 \end{gather}
 We can check that $(P_i^l,\ge)$ is a partially ordered set with the minimum element $(1^l, 1^{i-l})$ and with the maximal elements $(\mu,0)$ for $\mu\vdash i$ with $l(\mu)=l$.

\begin{lemma}\label{lempairpartcontraction}
 For $(\xi,\eta),(\mu,\nu)\in P_i^l$,
 we have
 \begin{gather*}
      F_{(\mu,\nu)}(\tau_*(\alpha_{(\xi,\eta)}))=0 \quad \text{unless } (\xi,\eta)\geq (\mu,\nu).
 \end{gather*}
\end{lemma}

\begin{proof}
 We have 
 \begin{gather}\label{taualpha}
  \tau_*(\alpha_{(\xi,\eta)})=
  \tau(\xi,1)\wedge\cdots\wedge \tau(\xi,l)\wedge \tau(\eta,1)\wedge\cdots\wedge \tau(\eta,l(\eta))\in \bigwedge^i U,
 \end{gather}
 where $\tau(\xi,j)$ (resp. $\tau(\eta,j)$) is obtained from $\tau_*(\alpha_{(\xi_j,0)})$ (resp. $\tau_*(\alpha_{(0,\eta_j)})$) by the shift that is appeared in the definition of $\alpha_{(\xi,\eta)}$.
 Note that $\{\tau(\xi,j)\mid 1\le j\le l\}\cup \{\tau(\eta,j)\mid 1\le j\le l(\eta)\}$ is a set of mutually disjoint elements.
 Therefore, by \eqref{disjoint} and by the definition of $F_{(\mu,\nu)}$, we have  $F_{(\mu,\nu)}(\tau_*(\alpha_{(\xi,\eta)}))=0$ unless
 there exist decompositions $L_1\sqcup\cdots\sqcup L_{l+l(\eta)} = \{1,\cdots,l(\nu)\}$ and $M_1\sqcup\cdots\sqcup M_{l+l(\eta)} = \{1,\cdots,l\}$ such that 
 \begin{gather*}
     \xi_j=\sum_{k\in M_j}\mu_k+ \sum_{k\in L_j}\nu_k\quad (1\le j\le l),\\
     \eta_j=\sum_{k\in M_{l+j}}\mu_k+ \sum_{k\in L_{l+j}}\nu_k\quad (1\le j\le l(\eta)).
 \end{gather*}
 Moreover, it follows from Lemma \ref{lemcontractioncomputation} (more specifically from the vanishing property \eqref{vanishingunderwheel}) that we have  $F_{(\mu,\nu)}(\tau_*(\alpha_{(\xi,\eta)}))=0$ unless $M_j\neq \emptyset$ for $1\le j\le l$. 
 Since a decomposition $M_1\sqcup\cdots\sqcup M_{l+l(\eta)} = \{1,\cdots,l\}$ with $M_j\neq \emptyset$ for $1\le j\le l$ is equivalent to a permutation $\sigma \in \gpS_l$, we have $F_{(\mu,\nu)}(\tau_*(\alpha_{(\xi,\eta)}))=0$ unless $(\xi,\eta)\geq (\mu,\nu)$.
\end{proof}

 \subsection{Proof of Theorem \ref{Johnsonpart}}
 
 To prove Theorem \ref{Johnsonpart}, we use the following lemma, which is an analogue of \cite[Theorem 1.5]{Lindell}.

 \begin{lemma}\label{Johnsonpartmunu}
 Let $(\mu,\nu)\vdash i$.
 Then for $n\ge i+2l(\mu)+l(\nu)$, we have 
 $$F_{(\mu,\nu)} (H^A_i(\IA_n,\Q)) \supset W(\mu,\nu).$$
 \end{lemma}
 
 \begin{proof}
  Our proof is analogous to the proof of \cite[Theorem 1.5]{Lindell}.
  By Corollary \ref{projgen}, $W(\mu,\nu)=\pi^{(\mu,\nu)}(T_{i+l(\mu), l(\mu)})$ is generated by the element $\pi^{(\mu,\nu)}(e_{i+l(\mu),l(\mu)})$.
  Therefore, it suffices to show that 
  there is an element $x \in \Q[\GL(n,\Q)]$ such that 
  $x (F_{(\mu,\nu)}(\tau_*(\alpha_{(\mu,\nu)})))=\pi^{(\mu,\nu)}(e_{i+l(\mu),l(\mu)})$.
  
  Here, we write 
  $$\mu=(\mu_1,\cdots,\mu_{l(\mu)})=(\ti{\mu}_1^{k_1'},\cdots, \ti{\mu}_r^{k_r'}),\quad \nu=(\nu_1,\cdots, \nu_{l(\nu)})=(\ti{\nu}_1^{k_1''},\cdots, \ti{\nu}_s^{k_s''}).$$
  As in \eqref{taualpha}, we have
  $$
   \tau_*(\alpha_{(\mu,\nu)})=
  \tau(\mu,1)\wedge\cdots\wedge \tau(\mu,l(\mu))\wedge \tau(\nu,1)\wedge\cdots\wedge \tau(\nu,l(\nu))\in \bigwedge^i U.
  $$
  For $1\le j\le r$ and $1\le m\le k_j'$, let $\tau(\mu,j,m)=\tau(\mu,\phi_{\mu}(j,m))$, where $\phi_{\mu}(j,m)=k_1'+\cdots + k_{j-1}'+m$. In a similar way, we define $\phi_{\nu}(j,m)$ and $\tau(\nu,j,m)$.
  Then we have 
   \begin{gather*}
    \begin{split}
    &F_{(\mu,\nu)}(\tau_*(\alpha_{(\mu,\nu)}))\\
    &=\left(\bigotimes_{j=1}^r \frac{1}{k_j'!} \prod_{m=1}^{k_j'}c^{\tree}_{\ti{\mu}_j}(\tau(\mu,j,m)) \right)\otimes 
    \left(\bigotimes_{j=1}^s \frac{1}{k_j''!} \prod_{m=1}^{k_j''}c^{\wheel}_{\ti{\nu}_j}  (\tau(\nu,j,m)) \right).
    \end{split}
  \end{gather*}
  By Lemma \ref{lemcontractioncomputation}, we have
  $$c^{\tree}_{\ti{\mu}_j} (\tau(\mu,j,m))=(-1)^{\ti{\mu}_j-1}(\ti{\mu}_j+1)! \: s(\phi_{\mu}(j,m))(e_{1}\wedge e_{2}\wedge e_{4}\wedge \cdots \wedge e_{\ti{\mu}_j+2}\otimes e_{3}^*),$$
  where $s(j)$ is the function that we used in Section \ref{secabeliancycle} to define the abelian cycle $\alpha_{(\mu,\nu)}$, and where $s(\phi_{\mu}(j,m))$ denotes the shift homomorphism by $s(\phi_{\mu}(j,m))$.
  By Lemma \ref{lemconnected}, we also have 
  $$c^{\wheel}_{\ti{\nu}_j}  (\tau(\nu,j,m))=(\ti{\nu}_j)!\: t(\phi_{\tau}(j,m))(e_2\wedge e_3\wedge\cdots \wedge e_{\ti{\nu}_j+1}).$$
  Therefore, we can take an element $x \in \Q[\GL(n,\Q)]$ to satisfy
  $x (F_{(\mu,\nu)}(\tau_*(\alpha_{(\mu,\nu)})))=\pi^{(\mu,\nu)}(e_{i+l(\mu),l(\mu)})$, which completes the proof.
 \end{proof}
 
 \begin{remark}
  Let $\mu\cup \nu$ denote the partition of $i$ that is obtained from $\mu$ and $\nu$ by reordering the parts.
  We can also use the abelian cycle $\alpha_{(0,\mu\cup\nu)}$ to prove Lemma \ref{Johnsonpartmunu}. That is, we can also show that $\tau_*(\alpha_{(0,\mu\cup\nu)})$ generates $W(\mu,\nu)$.
  However, we need the abelian cycle $\alpha_{(\mu,\nu)}$ to show that $F_{(\mu,\nu)}(H^A_i(\IA_n,\Q))$ includes the direct sum of $W(\mu,\nu)$.
 \end{remark}

 \begin{proof}[Proof of Theorem \ref{Johnsonpart}] 
  Let $F_{i}^{l}:=\bigoplus_{(\mu,\nu)\in P_i^l} F_{(\mu,\nu)}$.
  
  For $(\mu,\nu)\in P_i^l$ and $(\xi,\eta)\in P_i^{l'}$ with $l\neq l'$, any irreducible component of $W(\mu,\nu)$ and any irreducible component of $W(\xi,\eta)$ are not isomorphic. Therefore, it suffices to show that we have 
  $$\bigoplus_{(\mu,\nu)\in P_i^l} W(\mu,\nu)\subset F_{i}^{l} (H^A_i(\IA_n,\Q))$$
  for any $l\in\{0,\cdots,i\}$.
  
  By Lemma \ref{Johnsonpartmunu}, it suffices to show that for each $(\xi,\eta)\in P_i^l$, the element $$F_{(\xi,\eta)}(\tau_*(\alpha_{(\xi,\eta)}))\in W(\xi,\eta)\subset \bigoplus_{(\mu,\nu)\in P_i^l} W(\mu,\nu)$$
  is included in $F_{i}^{l} (H^A_i(\IA_n,\Q))$.
  In what follows, we identify an element of $W(\xi,\eta)$ with the image under the canonical inclusion $W(\xi,\eta)\hookrightarrow \bigoplus_{(\mu,\nu)\in P_i^l} W(\mu,\nu)$.
  
  We use the induction with respect to the partial order $\ge$ of $P_i^l$.
  For the minimum element $(1^l,1^{i-l})$ of $P_i^l$, by Lemma \ref{lempairpartcontraction}, we have 
  \begin{gather*}
  \begin{split}
       F_{(1^l,1^{i-l})}(\tau_*(\alpha_{(1^l,1^{i-l})}))=
       F_{i}^{l}(\tau_*(\alpha_{(1^l,1^{i-l})}))\in F_{i}^{l} (H^A_i(\IA_n,\Q)).
  \end{split}
  \end{gather*}
  For $(\xi,\eta)\in P_i^l$, suppose that for any $(\zeta,\epsilon)\le(\xi,\eta)$, we have 
  $$
  F_{(\zeta,\epsilon)}(\tau_*(\alpha_{(\zeta,\epsilon)}))\in F_{i}^{l} (H^A_i(\IA_n,\Q)).$$
  Then by Lemma \ref{lempairpartcontraction}, we have
  \begin{gather*}
   \begin{split}
   F_{i}^{l} (\tau_* (\alpha_{(\xi,\eta)}))&=F_{(\xi,\eta)}(\tau_*(\alpha_{(\xi,\eta)}))+ X(\xi,\eta)
   \end{split}
  \end{gather*}
  where $X(\xi,\eta)$ is an element of $\bigoplus_{(\zeta,\epsilon)\le (\xi,\eta)} W(\zeta,\epsilon)$.
  Therefore, by induction hypothesis, we have
  $$
  F_{(\xi,\eta)}(\tau_*(\alpha_{(\xi,\eta)}))\in F_{i}^{l} (H^A_i(\IA_n,\Q)).
  $$
  This completes the proof.
 \end{proof}

\section{Coalgebra structure of $H^A_*(\IA_n,\Q)$}\label{coalgebra}
 Here we recall the coalgebra structure of the rational homology of groups. We show that the map $F_*=\bigoplus_{i\ge 0} F_i: H_*(U,\Q) \to S^*(U_*)$, which we constructed in Section \ref{Generalpart}, is a coalgebra map. 
 
 \subsection{Coalgebra structure of $H_*(G,\Q)$}
 
 Let $G$ be a group.
 We briefly recall the graded-cocommutative coalgebra structure $(H_*(G,\Q), \Delta_*^G,\epsilon_*)$. (See \cite{Brown} for details.)
 
 The rational homology $H_*(G,\Q)$ is defined by
 $$
  H_*(G,\Q)=H_*(F\otimes_{G}\Q),
 $$
 where $F$ is a projective resolution of $\Z$ over $\Z[G]$.
 Here, we take the bar resolution.
 The diagonal map $\Delta^G: G\to G\times G,\;g\mapsto (g,g)$ induces a homomorphism
 $$\Delta_*^G: H_*(G,\Q)\to H_*(G,\Q)\otimes H_*(G,\Q),$$
 which coincides with the map induced by the Alexander--Whitney map
 $$\Delta: F\to F\otimes F,\quad (g_0,\cdots,g_n)\mapsto \sum_{p=0}^{n}(g_0,\cdots,g_p)\otimes (g_p,\cdots,g_n).$$
 (See Brown \cite[Section 1 of Chapter 5]{Brown}.)
 Then the induced map $\Delta_*^G$ can be written explicitly as follows:
 $$\Delta_*^G([x_1\otimes \cdots \otimes x_i])=\sum_{p=0}^i [x_1\otimes \cdots \otimes x_p] \otimes [x_{p+1}\otimes \cdots \otimes x_i]$$
 for $x_1,\cdots, x_i\in G$.
 We also have a trivial map $\epsilon: G\to 1$, which induces
 $$\epsilon_*: H_*(G,\Q)\to \Q.$$
 
 The canonical projection $\pi^G: G\twoheadrightarrow G^{\ab}$ induces a coalgebra map 
 $$\pi^G_*: H_*(G,\Q)\to H_*(G^{\ab},\Q).$$
 Therefore, the coalgebra structure of $H_*(G^{\ab},\Q)$ induces a subcoalgebra structure on $H^A_*(G,\Q)$.
 
 \subsection{Coalgebra structure of $H^A_*(\IA_n,\Q)$}
 As we saw in the previous subsection,
 we have a coalgebra structure of $H_*(U,\Q)$, which is compatible with the graded $\GL(n,\Q)$-representation structure.
 We consider the coalgebra structure on $S^*(U_*)$ that we observed in Section \ref{tracelesspartofgradedsymmetricalgebra}.
 Then the two coalgebra structures are compatible in the sense of the following proposition.
 
\begin{proposition}\label{coalgebramap}
 The graded $\GL(n,\Q)$-homomorphism 
 $$F_*=\bigoplus_{i\ge 0} F_i: H_*(U,\Q) \to S^*(U_*)$$
 is a coalgebra map.
\end{proposition}

\begin{proof}
 It suffices to show that we have
 $$
   (\bigoplus_{p=0}^{i}(F_p\otimes F_{i-p})) \Delta_i^U(x)=\Delta F_i (x)
 $$
 for $x=x_1\wedge \cdots \wedge x_i\in \bigwedge^i U\cong H_i(U,\Q)$.
 
 We have 
 $$
  \Delta_i^U (x)=\sum_{p=0}^i\sum_{\sigma\in \Sh(p,i-p)} \sgn(\sigma) (x_{\sigma(1)}\wedge \cdots\wedge x_{\sigma(p)})\otimes (x_{\sigma(p+1)}\wedge \cdots\wedge x_{\sigma(i)}).
 $$
 
 Let $p\in \{0,\cdots,i\}$. For pairs of partitions $(\xi,\eta)\vdash p, (\zeta,\epsilon)\vdash i-p$, we write $\xi=(\xi_1^{k_1'},\cdots, \xi_r^{k_r'})$, $\eta=(\eta_1^{k_1''},\cdots, \eta_s^{k_s''})$ and $\zeta=(\zeta_1^{m_1'},\cdots, \zeta_t^{m_t'})$, $\epsilon=(\epsilon_1^{m_1''},\cdots, \epsilon_u^{m_u''})$.
 Then we have
 \begin{gather*}
     \begin{split}
         &F_{(\xi,\eta)}(x_{\sigma(1)}\wedge \cdots\wedge x_{\sigma(p)})\\
         &= \left(\bigotimes_{j=1}^r \frac{1}{k_{j}'!}\prod^{k_{j}'}\tilde{c}^{\tree}_{\xi_j} \otimes \bigotimes_{j=1}^s \frac{1}{k_{j}''!} \prod^{k_{j}''}\tilde{c}^{\wheel}_{\eta_j}\right)\iota_p(x_{\sigma(1)}\wedge \cdots\wedge x_{\sigma(p)})\\
         &=\sum_{\rho\in \gpS_p}\sgn(\rho)\left(\bigotimes_{j=1}^r \frac{1}{k_{j}'!}\prod^{k_{j}'}\tilde{c}^{\tree}_{\xi_j} \otimes \bigotimes_{j=1}^s \frac{1}{k_{j}''!} \prod^{k_{j}''}\tilde{c}^{\wheel}_{\eta_j}\right)(\varphi(x_{\sigma\rho(1)})\otimes \cdots\otimes \varphi(x_{\sigma\rho(p)})).
     \end{split}
 \end{gather*}
 In a similar way, we have
  \begin{gather*}
     \begin{split}
         &F_{(\zeta,\epsilon)}(x_{\sigma(p+1)}\wedge \cdots\wedge x_{\sigma(i)})\\
         &=\sum_{\theta\in \gpS_{i-p}}\sgn(\theta)\left(\bigotimes_{j=1}^t \frac{1}{m_{j}'!}\prod^{m_{j}'}\tilde{c}^{\tree}_{\zeta_j} \otimes \bigotimes_{j=1}^u \frac{1}{m_{j}''!} \prod^{m_{j}''}\tilde{c}^{\wheel}_{\epsilon_j}\right)\\
         &\quad\quad
         (\varphi(x_{\sigma(p+\theta(1))})\otimes \cdots\otimes \varphi(x_{\sigma(p+\theta(i-p))})).
     \end{split}
 \end{gather*}
Therefore, we have
\begin{gather*}
    \begin{split}
        &(\bigoplus_{p=0}^{i}(F_p\otimes F_{i-p}))\Delta_i^U(x)\\
        &=\sum_{p=0}^{i}(F_p\otimes F_{i-p})(\sum_{\sigma\in \Sh(p,i-p)} \sgn(\sigma) (x_{\sigma(1)}\wedge \cdots\wedge x_{\sigma(p)})\otimes (x_{\sigma(p+1)}\wedge \cdots\wedge x_{\sigma(i)}))\\
        &=\sum_{p=0}^{i}\sum_{\substack{(\xi,\eta)\vdash p\\ (\zeta,\epsilon)\vdash i-p}}
        \left(\sum_{\sigma\in \Sh(p,i-p)}\sgn(\sigma)\right) \left(\sum_{\rho\in \gpS_p}\sgn(\rho)\right)
        \left(\sum_{\theta\in \gpS_{i-p}}\sgn(\theta)\right)\\
        &\quad \left(\bigotimes_{j=1}^r \frac{1}{k_{j}'!}\prod^{k_{j}'}\tilde{c}^{\tree}_{\xi_j} \otimes \bigotimes_{j=1}^s \frac{1}{k_{j}''!} \prod^{k_{j}''}\tilde{c}^{\wheel}_{\eta_j} \right)(\varphi(x_{\sigma\rho(1)})\otimes \cdots\otimes \varphi(x_{\sigma\rho(p)}))\\
        &\quad\otimes
        \left(
        \bigotimes_{j=1}^t \frac{1}{m_{j}'!}\prod^{m_{j}'}\tilde{c}^{\tree}_{\zeta_j} \otimes \bigotimes_{j=1}^u \frac{1}{m_{j}''!} \prod^{m_{j}''}\tilde{c}^{\wheel}_{\epsilon_j} \right) (\varphi(x_{\sigma(p+\theta(1))})\otimes \cdots\otimes \varphi(x_{\sigma(p+\theta(i-p))}))\\
        &=\sum_{p=0}^{i}\sum_{\substack{(\xi,\eta)\vdash p\\ (\zeta,\epsilon)\vdash i-p}}
        \sum_{\tau\in \gpS_i}\sgn(\tau)\\
        &\quad\left(\bigotimes_{j=1}^r \frac{1}{k_{j}'!}\prod^{k_{j}'}\tilde{c}^{\tree}_{\xi_j} \otimes \bigotimes_{j=1}^s \frac{1}{k_{j}''!} \prod^{k_{j}''}\tilde{c}^{\wheel}_{\eta_j} \right)(\varphi(x_{\tau(1)})\otimes \cdots\otimes  \varphi(x_{\tau(p)}))\\
        &\quad\otimes
        \left(\bigotimes_{j=1}^t \frac{1}{m_{j}'!}\prod^{m_{j}'}\tilde{c}^{\tree}_{\zeta_j} \otimes \bigotimes_{j=1}^u \frac{1}{m_{j}''!} \prod^{m_{j}''}\tilde{c}^{\wheel}_{\epsilon_j} \right)(\varphi(x_{\tau(p+1)})\otimes \cdots\otimes  \varphi(x_{\tau(i)})).
    \end{split}
\end{gather*}
 On the other hand, we have 
 \begin{gather*}
     \begin{split}
          \Delta F_i(x)
          &=\Delta \left(\sum_{(\mu,\nu)\vdash i} \bigotimes_{j=1}^v \frac{1}{n_{j}'!}\prod^{n_{j}'}\tilde{c}^{\tree}_{\mu_j} \otimes \bigotimes_{j=1}^w \frac{1}{n_{j}''!}\prod^{n_{j}''}\tilde{c}^{\wheel}_{\nu_j}\right)\iota_i(x_1\wedge \cdots \wedge x_i)\\
          &=\sum_{\sigma\in \gpS_i}\sgn(\sigma)\sum_{(\mu,\nu)\vdash i}\Delta\left(\bigotimes_{j=1}^v \frac{1}{n_{j}'!}\prod^{n_{j}'}\tilde{c}^{\tree}_{\mu_j} \otimes \bigotimes_{j=1}^w \frac{1}{n_{j}''!}\prod^{n_{j}''}\tilde{c}^{\wheel}_{\nu_j}\right)\\
          &\quad\quad(\varphi(x_{\sigma(1)})\otimes \cdots\otimes  \varphi(x_{\sigma(i)})),
     \end{split}
 \end{gather*}
 where we write $\mu=(\mu_1^{n_1'},\dots,\mu_v^{n_v'})$ and $\nu=(\nu_1^{n_1''},\dots, \nu_w^{n_w''})$.
 By the definition of $\Delta$ in Section \ref{tracelesspartofgradedsymmetricalgebra}, we have 
 \begin{gather*}
     \begin{split}
     &\Delta\left(\bigotimes_{j=1}^v \frac{1}{n_{j}'!}\prod^{n_{j}'}\tilde{c}^{\tree}_{\mu_j} \otimes \bigotimes_{j=1}^w \frac{1}{n_{j}''!}\prod^{n_{j}''}\tilde{c}^{\wheel}_{\nu_j}\right)(\varphi(x_{\sigma(1)})\otimes \cdots\otimes  \varphi(x_{\sigma(i)}))\\
     &=\prod_{j=1}^v\frac{1}{n_j'!}\prod_{j=1}^w\frac{1}{n_j''!}\sum_{\substack{\mu=\mu(1)\cup \mu(2)\\\nu=\nu(1)\cup \nu(2)}} \sum_{\rho}\sgn(\rho;c\varphi(x))\\
     &\quad \left(\bigotimes_{j=1}^v \prod^{n_{j,1}'}\tilde{c}^{\tree}_{\mu_j} \otimes \bigotimes_{j=1}^w \prod^{n_{j,1}''}\tilde{c}^{\wheel}_{\nu_j}\right)(\varphi(x_{\sigma\rho(1)})\otimes \cdots\otimes  \varphi(x_{\sigma\rho(|\mu(1)|+|\nu(1)|)}))\\
     &\quad \otimes \left(\bigotimes_{j=1}^v \prod^{n_{j,2}'}\tilde{c}^{\tree}_{\mu_j} \otimes \bigotimes_{j=1}^w \prod^{n_{j,2}''}\tilde{c}^{\wheel}_{\nu_j}\right)(\varphi(x_{\sigma\rho(|\mu(1)|+|\nu(1)|+1)})\otimes \cdots\otimes  \varphi(x_{\sigma\rho(i)})),
     \end{split}
 \end{gather*}
 where $\mu(p)=(\mu_1^{n_{1,p}'},\dots,\mu_v^{n_{v,p}'})$, $\nu(p)=(\nu_1^{n_{1,p}''},\dots,\nu_w^{n_{v,p}''})$ ($p=1,2$) are partitions such that $n_{j}'=n_{j,1}'+n_{j,2}'$, $n_{j}''=n_{j,1}''+n_{j,2}''$ (note that $n_{j,p}', n_{j,p}''$ are possibly zero), and where $\rho$ runs over the elements of $\gpS_i$ that are induced by combinations of $(n_{j,1}', n_{j,2}')$-shuffles and $(n_{j,1}'', n_{j,2}'')$-shuffles via the following correspondence
 \begin{gather*}
     \begin{split}
    &(\mu_1^{n_1'},\dots, \mu_v^{n_v'},\nu_1^{n_1''},\dots, \nu_w^{n_w''})\\
    &\mapsto (\mu_1^{n_{1,1}'},\dots, \mu_v^{n_{v,1}'},\nu_1^{n_{1,1}''},\dots, \nu_w^{n_{w,1}''}, \mu_1^{n_{1,2}'},\dots, \mu_v^{n_{v,2}'},\nu_1^{n_{1,2}''},\dots, \nu_w^{n_{w,2}''}),
     \end{split}
 \end{gather*}
 and where $\sgn(\rho;c\varphi(x))$ denotes the sign with regard to $\rho$ and $x_j$ as in the definition of $\Delta$ in Section \ref{tracelesspartofgradedsymmetricalgebra}.
 Since the number of $\rho$ satisfying the above condition is $\prod_{j=1}^{v}\binom{n_j'}{n_{j,1}'}\prod_{j=1}^{w}\binom{n_j''}{n_{j,1}''}$, and since we have 
 \begin{gather*}
    \prod_{j=1}^{v}\frac{1}{n_j'!} \binom{n_j'}{n_{j,1}'}
    \prod_{j=1}^{w}\frac{1}{n_j''!}\binom{n_j''}{n_{j,1}''}
    =\prod_{j=1}^{v}\frac{1}{n_{j,1}'!\; n_{j,2}'!}
     \prod_{j=1}^{w}\frac{1}{n_{j,1}''!\; n_{j,2}''!},
 \end{gather*}
 it follows that
 \begin{gather*}
     \begin{split}
          \Delta F_i(x)
          &=\sum_{\sigma\in \gpS_i}\sum_{\substack{(\mu,\nu)\vdash i\\ \mu=\mu(1)\cup\mu(2)\\\nu=\nu(1)\cup \nu(2)}}\sgn(\sigma)\prod_{j=1}^{v}\frac{1}{n_{j,1}'!\; n_{j,2}'!}
          \prod_{j=1}^{w}\frac{1}{n_{j,1}''!\; n_{j,2}''!}\\
          & \left(\bigotimes_{j=1}^v \prod^{n_{j,1}'}\tilde{c}^{\tree}_{\mu_j} \otimes \bigotimes_{j=1}^w \prod^{n_{j,1}''}\tilde{c}^{\wheel}_{\nu_j}\right)(\varphi(x_{\sigma(1)})\otimes \cdots\otimes  \varphi(x_{\sigma(|\mu(1)|+|\nu(1)|)}))\\
          & \otimes \left(\bigotimes_{j=1}^v \prod^{n_{j,2}'}\tilde{c}^{\tree}_{\mu_j} \otimes \bigotimes_{j=1}^w \prod^{n_{j,2}''}\tilde{c}^{\wheel}_{\nu_j}\right)(\varphi(x_{\sigma(|\mu(1)|+|\nu(1)|+1)})\otimes \cdots\otimes  \varphi(x_{\sigma(i)})).
     \end{split}
 \end{gather*}
 By letting $(\xi,\eta):=(\mu(1),\nu(1))$ and $(\zeta,\epsilon):=(\mu(2),\nu(2))$, we obtain
 \begin{gather*}
     \begin{split}
        \Delta F_i(x)
        &=\sum_{\sigma\in \gpS_i}\sum_{p=0}^{i}\sum_{\substack{(\xi,\eta)\vdash p\\ (\zeta,\epsilon)\vdash i-p}}
        \sgn(\sigma)\\
        &\quad\left(\bigotimes_{j=1}^r \frac{1}{k_{j}'!}\prod^{k_{j}'}\tilde{c}^{\tree}_{\xi_j} \otimes \bigotimes_{j=1}^s \frac{1}{k_{j}''!} \prod^{k_{j}''}\tilde{c}^{\wheel}_{\eta_j} \right)(\varphi(x_{\sigma(1)})\otimes \cdots\otimes  \varphi(x_{\sigma(p)}))\\
        &\quad\otimes
        \left(\bigotimes_{j=1}^t \frac{1}{m_{j}'!}\prod^{m_{j}'}\tilde{c}^{\tree}_{\zeta_j} \otimes \bigotimes_{j=1}^u \frac{1}{m_{j}''!} \prod^{m_{j}''}\tilde{c}^{\wheel}_{\epsilon_j} \right)(\varphi(x_{\sigma(p+1)})\otimes \cdots\otimes  \varphi(x_{\sigma(i)}))\\
        &=(\bigoplus_{p=0}^{i}(F_p\otimes F_{i-p})) \Delta_i^U(x).
     \end{split}
 \end{gather*}
\end{proof}

By Proposition \ref{coalgebramap}, the subcoalgebra $H^A_*(\IA_n,\Q)\subset H_*(U,\Q)$ is mapped to a subcoalgebra of $S^*(U_*)$, which includes $W_*$ as a subcoalgebra.
For a coaugmented coalgebra $A$, let $\Prim(A)$ denote the \emph{primitive part} of $A$.
We can check that 
$$\Prim(S^*(U_*))=U_*\subset W_*.$$
Since coalgebra maps preserve the primitive part, the graded $\GL(n,\Q)$-homomorphism $F_*$ restricts to 
$$F_*: \Prim(H^A_*(\IA_n,\Q))\to \Prim(S^*(U_*))=U_*.$$
If Conjecture \ref{conjectureAlbanese} is true, then the following conjecture is true.

\begin{conjecture}\label{conjectureprim}
 For $n\ge 3i$, the $\GL(n,\Q)$-homomorphism 
 $$F_i:\Prim(H^A_*(\IA_n,\Q))_i\to U_i$$
 is an isomorphism, where $\Prim(H^A_*(\IA_n,\Q))_i$ denotes the degree $i$ part of $\Prim(H^A_*(\IA_n,\Q))$.
\end{conjecture}

\section{Albanese cohomology of $\IA_n$}\label{albanesecohomology}
 In this section, we study the subalgebra of the rational cohomology algebra $H^*(\IA_n,\Q)$ that Church--Ellenberg--Farb \cite{CEF} called the Albanese cohomology of $\IA_n$.
 
 \subsection{Albanese cohomology of groups}
 
 For a group $G$, the \emph{Albanese cohomology} $H_A^i(G,\Q)$ of $G$ is defined by
 $$
  H_A^i(G,\Q)=\im (\pi^*:H^i(G^{\ab},\Q)\to H^i(G,\Q)),
 $$
 where $\pi:G\twoheadrightarrow G^{\ab}$ is the abelianization map.
 
 We have a linear isomorphism
 $$
  H_A^i(G,\Q)\xrightarrow{\cong} (H^A_i(G,\Q))^*=\Hom_{\Q}(H^A_i(G,\Q),\Q).
 $$
 (See Lemma \ref{dualityisom}.)
 
 It is well known that $H^*(G,\Q)$ is a graded-symmetric algebra with the cup product as a multiplication, which is the dual of the comultiplication of $H_*(G,\Q)$.
 Then $H_A^*(G,\Q)$ is a subalgebra of $H^*(G,\Q)$.
 Since we have $H^i(G^{\ab},\Q))\cong \bigwedge^i H^1(G,\Q)$, the algebra $H_A^*(G,\Q)$ is generated by $H^1(G,\Q)$ as an algebra.
 
 \subsection{Albanese cohomology of $\IA_n$}
 
 As we saw in the previous subsection, the Albanese cohomology $H_A^*(\IA_n,\Q)$ has a graded-symmetric algebra structure and is generated by $H^1(\IA_n,\Q)$.
 Moreover, the linear isomorphism 
 $$
  H_A^i(G,\Q)\xrightarrow{\cong} (H^A_i(G,\Q))^*
 $$
 is a $\GL(n,\Q)$-isomorphism (see Proposition \ref{dualityisomIA}).
 
 Let $(S^*(U_*))^*$ (resp. $(H_*(U,\Q))^*$) denote the graded dual of $S^*(U_*)$ (resp. $H_*(U,\Q)$) and let 
 $$F^i: (S^*(U_*)_i)^*\xrightarrow{F_i^*} (H_i(U,\Q))^*\cong H^i(U,\Q)\overset{\pi^*}{\twoheadrightarrow} H_A^i(\IA_n,\Q).$$
 Then by Proposition \ref{coalgebramap}, $F^*=\bigoplus_{i\ge 0}F^i$ is an algebra map.
 
 \begin{proposition}
   The graded $\GL(n,\Q)$-homomorphism 
   $$F^*: (S^*(U_*))^*\to H_A^*(\IA_n,\Q)$$
   is an algebra map.
 \end{proposition}

Let $\langle R_2\rangle$ denote the ideal of $H^*(U,\Q)$ generated by $R_2 = \ker (\tau^*: H^2(U,\Q)\to H^2(\IA_n,\Q))$.
We have a surjective $\GL(n,\Q)$-homomorphism 
$$
  H^*(U,\Q)/\langle R_2\rangle \twoheadrightarrow H_A^*(\IA_n,\Q).
$$

\begin{conjecture}\label{quadratic}
   The Albanese cohomology algebra $H_A^*(\IA_n,\Q)$ is stably quadratic, that is, the surjective $\GL(n,\Q)$-homomorphism
  $$
  H^*(U,\Q)/\langle R_2\rangle \twoheadrightarrow H_A^*(\IA_n,\Q)
  $$
  is an isomorphism for sufficiently large $n$ with respect to the cohomological degree.
\end{conjecture}

Conjecture \ref{quadratic} holds for $*=3$. See Remark \ref{quadratic3}.

\section{Albanese homology of $\IO_n$}\label{Albanesehomology}

The \emph{inner automorphism group} $\Inn(F_n)$ of $F_n$ is the normal subgroup of $\Aut(F_n)$ consisting of $\{\sigma_x \mid x\in F_n\}$, where $\sigma_x(y)=xyx^{-1}$ for any $y\in F_n$.
The \emph{outer automorphism group} $\Out(F_n)$ of $F_n$ is the quotient group of $\Aut(F_n)$ by $\Inn(F_n)$.
Since we have $\Inn(F_n)\subset \IA_n$, we have a surjection
$\Out(F_n)\twoheadrightarrow \GL(n,\Z)$. 
Let $\IO_n$ denote its kernel.
That is, we have exact sequences
$$
1\to \IO_n \to \Out(F_n)\to \GL(n,\Z) \to 1
$$
and
$$
1\to \Inn(F_n)\to \IA_n \xrightarrow{\pi} \IO_n\to 1.
$$
Note that $\IO_2=1$ since we have $\Inn(F_2)=\IA_2$. In what follows, we will consider $\IO_n$ for $n\ge 3$.

As in the case of $\Aut(F_n)$, the Johnson homomorphism for $\Out(F_n)$ induces an isomorphism on the first homology \cite{Kawazumi}
$$
\tau^O : H_1(\IO_n,\Z)\xrightarrow{\cong} \Hom(H_{\Z}, \bigwedge^2 H_{\Z})/H_{\Z},
$$
where $H_{\Z}$ is embedded into $\Hom(H_{\Z}, \bigwedge^2 H_{\Z})$ via the map $x\mapsto x\wedge -$ for $n\ge 2$.

We can also consider $H_i(\IO_n,\Q)$ as a $\GL(n,\Z)$-representation, and the Johnson homomorphism preserves the $\GL(n,\Z)$-action.
Then we have 
$$H_1(\IO_n,\Q)\cong\Hom(H, \bigwedge^2 H)/H \cong V_{1^2,1}.$$
Let $U^O=\Hom(H, \bigwedge^2 H)/H$.
The Johnson homomorphism induces a $\GL(n,\Z)$-homomorphism on homology
$$
\tau^O_*: H_i(\IO_n,\Q)\to H_i(U^O,\Q).
$$
In this section, we study the Albanese homology of $\IO_n$ and observe some relation between $H^A_i(\IO_n,\Q)$ and $H^A_i(\IA_n,\Q)$.

\subsection{A set of generators for $H_1(\IO_n,\Q)$}\label{generatorforIOn}
 Here we obtain a set of generators for $H_1(\IO_n,\Q)$, which is induced by Magnus's set of generators for $\IA_n$.
 
 The projection $\pi: \IA_n\twoheadrightarrow \IO_n$ induces a surjective $\GL(n,\Q)$-homomorphism
 $$
 \pi_*: H_1(\IA_n,\Q) \to H_1(\IO_n,\Q).
 $$
 Then we have the following commutative diagram of $\GL(n,\Q)$-representations:
 \begin{gather*}
   \xymatrix{
    H_1(\IA_n,\Q)\ar@{->>}[r]^{\pi_*}\ar[d]^{\tau}_{\cong}
    &
     H_1(\IO_n,\Q)\ar[d]^{\tau^O}_{\cong}
      \\
     U=\Hom(H,\bigwedge^2 H) \ar@{->>}[r]^-{\pr}
    &
    U^O=\Hom(H,\bigwedge^2 H)/ H,
   }
 \end{gather*}
 where the bottom map is the projection.
 
 Magnus's set of generators for $\IA_n$ induces the following set of generators for $U^O$:
 $$\{\overline{g_{a,b}}\mid 1\le a,b \le n,\; a\ne b\}\cup \{\overline{f_{a,b,c}}\mid 1\le a,b,c\le n,\; a< b,\; a\ne c\ne b\},$$
 where
 $$\overline{g_{a,b}}=\pr\tau(g_{a,b})= \pr(e_{a,b}^{b})\in U^O,\quad
 \overline{f_{a,b,c}}=\pr\tau(f_{a,b,c})=\pr(e_{a,b}^{c})\in U^O.$$
 We have an isomorphism $U^O \xrightarrow{\cong} U_{1}^{\tree}$ which maps
 $$\overline{f_{a,b,c}} \mapsto e_{a,b}^{c},\quad \overline{g_{a,b}} \mapsto e_{a,b}^{b}- \frac{1}{n-1} (\sum_{j=1}^n e_{a,j}^{j})$$
 since we have $e_{a,b}^{c}\in U_{1}^{\tree}$ and 
 \begin{gather*}
 \begin{split}
     e_{a,b}^{b}- \frac{1}{n-1} (\sum_{j=1}^n e_{a,j}^{j})
     &=\frac{1}{n-1} \sum_{j\neq a,b}  (e_{a,b}^{b} -e_{a,j}^{j})\\
     &=\frac{1}{n-1} \sum_{j\neq a,b}([\id]- [E_{j,b}]- [P_{j,b}]) (e_{a,b}^{j}) \in U_{1}^{\tree},
 \end{split}
 \end{gather*}
 where the matrices $P_{j,b}$ and $E_{j,b}$ are defined in Section \ref{secnotations}.
 In what follows, we identify $U^O$ with the subrepresentation $U_{1}^{\tree}$ of $U=U_{1}^{\tree}\oplus U_{1}^{\wheel}$.

\subsection{Computation of the contraction maps}
 The inclusion map 
 $\iota_{i}:\bigwedge^i U \hookrightarrow M_i$ that we defined in Section \ref{subseccontraction}
 restricts to an inclusion map 
 \begin{gather*}
      \iota_{i}: \bigwedge^{i} U^O \hookrightarrow \bigwedge^{i} U\hookrightarrow M_i.
 \end{gather*}
 Then we can consider the composition of $\iota_{i}$ and each of the two contraction maps $\tilde{c}^{\wheel}_{i}$ and $\tilde{c}^{\tree}_{i}$, which we defined in Section \ref{subseccontraction}.
 The abelian cycle $\alpha_{(0,i)}$ of $H_i(\IA_n,\Q)$ induces an abelian cycle $\pi_*(\alpha_{(0,i)})$ of $H_i(\IO_n,\Q)$.
 Here we compute the two contraction maps for $\tau^O_*\pi_*(\alpha_{(0,i)})$ as in Lemma \ref{lemconnected}.
 
 \begin{lemma}\label{lemcontractionout0i}
  (1)  For $i=1$, we have $\tilde{c}^{\wheel}_{1} \iota_{1} \tau^O_* \pi_*(\alpha_{(0,1)}) = 0.$ \\
  (2)  For $i\ge 2$, we have $\tilde{c}^{\wheel}_{i} \iota_{i} \tau^O_* \pi_*(\alpha_{(0,i)})\neq 0$ for sufficiently large $n$.\\
  (3)  For $i\ge 1$, we have $\tilde{c}^{\tree}_{i} \iota_{i} \tau^O_* \pi_*(\alpha_{(0,i)})\neq 0$ for sufficiently large $n$.
 \end{lemma}
 
 \begin{proof}
  The proof is similar to that of Lemma \ref{lemconnected}.
  We have
  \begin{gather*}
   \begin{split}
        \tau^O_* \pi_*(\alpha_{(0,i)})
        &=\overline{g_{2,1}}\wedge \left(\overline{g_{3,1}}+\overline{g_{3,2}}\right) \wedge \cdots\wedge \left(\sum_{j=1}^i \overline{g_{i+1,j}}\right)
        \\
        &=
        \bigwedge_{k=1}^{i}
        \left(\sum_{j=1}^{k}\frac{n-k-1}{n-1} e_{k+1,j}^{j}
        -\frac{k}{n-1}\sum_{j=k+2}^n e_{k+1,j}^{j}\right).
   \end{split}
  \end{gather*}
  
  For $i=1$, we have 
  \begin{gather*}
   \begin{split}
    \tilde{c}^{\wheel}_{1} \iota_{1} \tau^O_* \pi_*(\alpha_{(0,1)}) &=\tilde{c}^{\wheel}_{1} \iota_{1} \left( \frac{n-2}{n-1}e_{2,1}^{1} -\frac{1}{n-1}\sum_{j=3}^n e_{2,j}^{j}\right)\\
    &=\frac{n-2}{n-1}e_2-\frac{1}{n-1}(n-2)e_2=0,
   \end{split}
  \end{gather*}
  which proves (1).
    
   For $i\ge 2$, $n\ge i+1$, we have
     \begin{gather*}
      \begin{split}
          &\tilde{c}^{\wheel}_{i} \iota_{i} \tau^O_* \pi_*(\alpha_{(0,i)}) \\
          &\quad = \frac{i! (n-2)(n-3)\cdots(n-i-1)+ f_i(n)}{(n-1)^i} e_2\wedge e_3\wedge \cdots \wedge e_{i+1}\in V_{1^i,0},
      \end{split}
     \end{gather*}
     where $f_i(n)$ is a polynomial of degree less than $i$.
    Therefore, we have (2).
   
    For $i\ge 1$, $n\ge i+1$, we have    
     \begin{gather*}
      \begin{split}
        &\tilde{c}^{\tree}_{i} \iota_{i} \tau^O_* \pi_*(\alpha_{(0,i)})\\
        &\quad= \frac{(i+1)!(n-2)(n-3)\cdots(n-i-1)+g_i(n)}{(n-1)^i}  e_2\wedge e_3\wedge \cdots \wedge e_{i+1} \wedge e_{1}\otimes e_1^*\\
        &\quad + \sum_{j=i+2}^{n} \frac{h_{i}(n)}{(n-1)^i} e_2\wedge e_3\wedge \cdots \wedge e_{i+1} \wedge e_{j}\otimes e_j^*
        \in V_{1^{i+1},1},
      \end{split}
     \end{gather*}
     where $g_{i}(n)$ is a polynomial of degree less than $i$ and $h_{i}(n)$ is a polynomial of degree $i-1$.
     Therefore, we have (3).
 \end{proof}
 
 \begin{theorem}\label{ConnectedpartIO}
  Let $i\ge 2$.
  For sufficiently large $n$ with respect to $i$, we have a surjective $\GL(n,\Q)$-homomorphism
  \begin{gather*}
      H^A_i(\IO_n,\Q)\twoheadrightarrow \Hom(H,\bigwedge^{i+1} H) \cong V_{1^{i+1},1} \oplus V_{1^i,0}.
  \end{gather*}
 \end{theorem}
 
 \begin{proof}
  This directly follows from Lemma \ref{lemcontractionout0i}.
 \end{proof}
 
\subsection{The traceless part of $H^A_i(\IO_n,\Q)$}

We have a commutative diagram
\begin{gather*}
   \xymatrix{
    H_i(\IA_n,\Q)\ar[r]^-{\pi_*}\ar[d]^-{\tau_*}
    &
     H_i(\IO_n,\Q)\ar[d]^-{\tau^O_*}
      \\
     H_i(U,\Q)\ar@{->>}[r]_-{\pr_*}
    &
    H_i(U^O,\Q).
   }
\end{gather*}
Therefore, we have
\begin{gather}\label{AlbaneseIOandIA}
\pr_*(H^A_i(\IA_n,\Q))\subset H^A_i(\IO_n,\Q). 
\end{gather}

\begin{question}
 Is $\pr_*: H^A_i(\IA_n,\Q) \to H^A_i(\IO_n,\Q)$ surjective for any $n$?
\end{question}

We have the following decomposition of $H_i(U,\Q)$ as $\GL(n,\Q)$-representations
\begin{gather*}
 \begin{split}
    H_i(U,\Q)&=\bigwedge^i U= \left(\bigwedge^i U^O\right) \oplus Y_i,
 \end{split}
\end{gather*}
where $Y_i$ is a subrepresentation of $H_i(U,\Q)$ whose irreducible decomposition does not include $V_{\ul\lambda}$ for any $\ul\lambda$ such that $|\ul\lambda|=3i$.
Therefore, by Theorem \ref{tracelessimage}, we obtain the following theorem.
Recall that we have $\dim_{\Q}(H_i(U,\Q)^{\tl})=Q_i(n)$ for $n\ge 3i$, where $Q_i(T)$ is a polynomial of degree $3i$.

\begin{theorem}\label{dimensionIO}
 For $n\ge 3i$, we have 
 $$H_i(U,\Q)^{\tl}\subset H^A_i(\IO_n,\Q).$$
 In particular, we have $\dim_{\Q}(H^A_i(\IO_n,\Q))\ge Q_i(n)$ for $n\ge 3i$.
\end{theorem}

\begin{conjecture}
There is a polynomial $P^O_i(T)$ of degree $3i$ such that we have $\dim_{\Q}(H^A_i(\IO_n,\Q))=P^O_i(n)$ for sufficiently large $n$ with respect to $i$.
\end{conjecture}

Moreover, by Theorem \ref{Johnsonpart}, we have a direct summand $\wti{W_i}$ of $H^A_i(\IA_n,\Q)$ which is isomorphic to $W_i$.
Therefore, by \eqref{AlbaneseIOandIA}, we have for $n\ge 3i$,
$$\pr_*(\wti{W_i})\subset H^A_i(\IO_n,\Q).$$

\subsection{Conjectural structure of $H^A_i(\IO_n,\Q)$}
 Here we propose a conjectural structure of $H^A_i(\IO_n,\Q)$.
 
 Define $U^O_i$ by
 $$U^O_i=
 \begin{cases}
  U^O \cong U_1^{\tree} & (i=1)\\
  U_i & (i\ge 2).
 \end{cases}
 $$
 For the graded $\GL(n,\Q)$-representation $U^O_*=\bigoplus_{i\ge 1} U^O_i$, let
 $S^* (U^O_*)$ denote the graded-symmetric algebra of $U^O_*$.
 Let
 $W^O_*=\wti{S}^* (U^O_*)$ denote the traceless part of $S^*(U^O_*)$.
 Let 
 $P^O_i\subset P_i$ denote the subset of $P_i$ consisting of pairs of partitions $(\mu,\nu)$ of total size $i$ such that $\nu$ has no part of size $1$.
 Then we have
 $$
 W^O_i=\bigoplus_{(\mu,\nu)\in P^O_i} W(\mu,\nu).
 $$
 
 We make the following conjecture, which is true for $i=1$ by Kawazumi \cite{Kawazumi}, for $i=2$ by Pettet \cite{Pettet} and for $i=3$ as we will observe in Theorem \ref{thirdAlbaneseIO}.
 
 \begin{conjecture}\label{conjectureAlbaneseIO}
  For sufficiently large $n$ with respect to $i$, we have a $\GL(n,\Q)$-isomorphism
  $$H^A_i(\IO_n,\Q) \xrightarrow{\cong} W^O_i.$$
 \end{conjecture}
 
 Since we have a direct sum decomposition
 $$S^*(U^O_*)_i=\bigoplus_{(\mu,\nu)\in P^O_i} 
    U_{\mu}^{\tree} \otimes U_{\nu}^{\wheel},$$
 there is a canonical surjective $\GL(n,\Q)$-homomorphism
 $$ \Pr : S^*(U_*)_i=
    \bigoplus_{(\mu,\nu)\in P_i} 
      U_{\mu}^{\tree} \otimes U_{\nu}^{\wheel}\twoheadrightarrow S^*(U^O_*)_i=\bigoplus_{(\mu,\nu)\in P^O_i} 
    U_{\mu}^{\tree} \otimes U_{\nu}^{\wheel}.$$
Then we have a $\GL(n,\Q)$-homomorphism 
$$G_{i}: H_i(U^O,\Q)=\bigwedge^i U^O \hookrightarrow \bigwedge^i U \xrightarrow{F_i} S^*(U_*)_i\xrightarrow{\Pr} S^*(U^O_*)_i.$$ 
We expect that $G_i$ restricts to a $\GL(n,\Q)$-isomorphism
$G_i: H^A_i(\IO_n,\Q) \xrightarrow{\cong} W^O_i$, which implies that Conjecture \ref{conjectureAlbaneseIO} is true.

\subsection{Structures of $H^A_i(\IO_n,\Q)$ and $H^A_i(\IA_n,\Q)$}
 Here we study the relation between the structures of $H^A_i(\IO_n,\Q)$ and $H^A_i(\IA_n,\Q)$.
 
 Recall that we have an exact sequence of groups
\begin{gather}\label{IAIOexact}
    1\to \Inn(F_n)\to \IA_n\to \IO_n\to 1.
\end{gather}
 Since we have $\Inn(F_n)\cong F_n$ for $n\ge 2$, we identify $\Inn(F_n)$ with $F_n$.
 By Proposition \ref{cospectralsequence3}, we obtain the following proposition.
 
 \begin{proposition}\label{AlbaneseIAandIO2}
      For $n\ge 2$, we have a $\GL(n,\Q)$-isomorphism
  \begin{gather*}
   \begin{split}
       H^A_i(\IA_n,\Q) \xrightarrow{\cong}  H^A_i(\IO_n,\Q)\oplus (H^A_{i-1}(\IO_n,\Q)\otimes H).
   \end{split}
  \end{gather*}
 \end{proposition}

 \begin{proof}
     The exact sequence \eqref{IAIOexact} induces an exact sequence
     \begin{gather*}
         0\to (F_n)^{\ab}\to \IA_n^{\ab}\to \IO_n^{\ab}\to 0,
     \end{gather*}
     which splits as $\Z[\GL(n,\Z)]$-modules.
     It is obvious that $\IO_n$ acts trivially on $H^*(F_n,\Q)$, $E_2^{p,q}=H^p(\IO_n,H^q(F_n,\Q))\cong H^p(\IO_n,\Q)\otimes H^q(F_n,\Q)=0$ for $q\ge 2$, $d_2^{0,1}=0$ and that $H^*(\IA_n^{\ab},\Q)$ is a semisimple $\Q[\GL(n,\Z)]$-module.
     Therefore, by Proposition \ref{cospectralsequence3}, we have
     \begin{gather*}
         H_A^*(\IA_n,\Q)\cong H_A^*(\IO_n,\Q)\otimes H_A^*(F_n,\Q)
     \end{gather*}
     as $\Q[\GL(n,\Z)]$-modules.
     Since the Albanese homology of $\IA_n$ is algebraic, by the computation of $H_A^*(F_n,\Q)$, we obtain the statement.
 \end{proof}
 
 To study the relation between $H^A_*(\IO_n,\Q)$ and $H^A_*(\IA_n,\Q)$, we use the following lemma.
 
  \begin{lemma}\label{WiWoi}
  We have $W_i\cong W^O_i \oplus (W^O_{i-1}\otimes H)$ as $\GL(n,\Q)$-representations for $n\ge 3i$.
 \end{lemma}
 
 \begin{proof}
  By the definitions of $W_i$ and $W^O_i$, it suffices to show that 
  \begin{gather}\label{WiWOiisom}
        \bigoplus_{(\mu,\nu)\in P^O_{i-1}}W(\mu,\nu) \otimes H \cong \bigoplus_{(\xi,\eta)\in P_i\setminus P^O_i} W(\xi,\eta).
  \end{gather}
 For $(\mu,\nu)\in P^O_{i-1}$, we write $\mu=(\mu_1^{k_1'},\cdots,\mu_r^{k_r'})$ and $\nu=(\nu_1^{k_1''},\cdots,\nu_s^{k_s''})$ with $\nu_s>1$.
 Let $\mu-\mu_{j}=(\mu_1^{k_1'},\cdots, \mu_j^{k_j'-1},\cdots,\mu_r^{k_r'})$ for $1\le j\le r$ 
 and $\nu+1^a=(\nu_1^{k_1''},\cdots,\nu_s^{k_s''},1^a)$ for $a\ge 1$.
 Since we have
 \begin{gather*}
     (U_{\nu_i}^{\wheel}\;\wti{\otimes}\; U_{\mu_j}^{\tree})\otimes H \cong (U_{\nu_i}^{\wheel}\;\wti{\otimes}\; U_{\mu_j}^{\tree}\;\wti{\otimes}\; H ) \oplus  (U_{\nu_i}^{\wheel}\;\wti{\otimes}\; \wti{S}^{1+\mu_j} (U_1^{\wheel})),
 \end{gather*}
  it follows that
 \begin{gather*}
    W(\mu,\nu)\otimes H 
    \cong  (W(\mu,\nu)\;\wti{\otimes}\; H) \oplus \left(\bigoplus_{j=1}^r W(\mu-\mu_j,\nu)\;\wti{\otimes} \; \wti{S}^{1+\mu_j} (U_1^{\wheel})\right).
 \end{gather*}
 Since $\nu$ has no part of size $1$, we have
 $$
  W(\mu,\nu) \;\wti{\otimes}\; H\cong W(\mu,\nu+1)
 $$
 and 
 $$
 W(\mu-\mu_j,\nu)\;\wti{\otimes} \; \wti{S}^{1+\mu_j} (U_1^{\wheel})\cong W(\mu-\mu_{j},\nu+1^{1+\mu_{j}}).
 $$
 Therefore, we have 
  \begin{gather}\label{WiandWOi}
    W(\mu,\nu)\otimes H \cong W(\mu, \nu+1)\oplus \bigoplus_{j=1}^r W(\mu-\mu_{j},\nu+1^{1+\mu_{j}}).
 \end{gather}
 Hence, for each $(\mu,\nu)\in P^O_{i-1}$, we have
  $$
   W(\mu,\nu) \otimes H \subset \bigoplus_{(\xi,\eta)\in P_i\setminus P^O_i} W(\xi,\eta).
 $$
 Moreover, we have 
 \begin{gather}\label{WiWOiinclu}
     \bigoplus_{(\mu,\nu)\in P^O_{i-1}}W(\mu,\nu) \otimes H \subset \bigoplus_{(\xi,\eta)\in P_i\setminus P^O_i} W(\xi,\eta)
 \end{gather}
since for distinct pairs $(\mu,\nu)\neq (\mu',\nu')$, the pairs $(\mu, \nu+1),(\mu', \nu'+1), (\mu-\mu_{j},\nu+1^{1+\mu_{j}}),(\mu'-\mu'_{j},\nu'+1^{1+\mu'_{j}})$ for $1\le j\le r$ are distinct.

Let $(\xi,\eta)\in P_i\setminus P^O_i$. If $\eta$ has only one part of size $1$, then we have $(\mu,\nu)\in P^O_{i-1}$ such that $(\mu,\nu+1)=(\xi,\eta)$.
Otherwise, we have $(\mu,\nu)\in P^O_{i-1}$ such that $(\mu-\mu_j,\nu+1^{1+\mu_j})=(\xi,\eta)$.
Therefore, by the decomposition \eqref{WiandWOi}, we obtain 
$$
   \bigoplus_{(\mu,\nu)\in P^O_{i-1}}W(\mu,\nu) \otimes H \supset \bigoplus_{(\xi,\eta)\in P_i\setminus P^O_i} W(\xi,\eta).
 $$
 Therefore, by using \eqref{WiWOiinclu}, we obtain \eqref{WiWOiisom}.
 \end{proof}
 
 By Theorem \ref{Johnsonpart} and Proposition \ref{AlbaneseIAandIO2}, we obtain the following proposition, which partially ensures Conjecture \ref{conjectureAlbaneseIO}.
 
\begin{proposition}\label{IOnWO}
  For $n\ge 3i$, we have an injective $\GL(n,\Q)$-homomorphism 
  $$W^O_i \oplus (W^O_{i-1} \otimes H)\hookrightarrow H^A_i(\IO_n,\Q) \oplus (H^A_{i-1}(\IO_n,\Q)\otimes H).$$
\end{proposition}

\begin{proof}
Let $n\ge 3i$.
By Proposition \ref{AlbaneseIAandIO2}, we have 
$$H^A_i(\IA_n,\Q)\xrightarrow{\cong} H^A_i(\IO_n,\Q) \oplus (H^A_{i-1}(\IO_n,\Q)\otimes H).$$
By Theorem \ref{Johnsonpart}, $H^A_i(\IA_n,\Q)$ contains a subrepresentation which is isomorphic to $W_i$.
Since we have 
$$W^O_i \oplus (W^O_{i-1}\otimes H)\cong W_i$$
by Lemma \ref{WiWoi}, we have an injective $\GL(n,\Q)$-homomorphism
$$W^O_i \oplus (W^O_{i-1}\otimes H)\hookrightarrow H^A_i(\IO_n,\Q) \oplus (H^A_{i-1}(\IO_n,\Q)\otimes H).$$
\end{proof}
 
  Recall that Conjecture \ref{conjectureAlbaneseIO} states that we stably have $H^A_*(\IO_n,\Q)\cong W^O_*$.
  Then we obtain the equivalence between the conjectures about the structures of $H^A_*(\IA_n,\Q)$ and $H^A_*(\IO_n,\Q)$.
 
 \begin{proposition}\label{AlbaneseIAandIO}
   Conjecture \ref{conjectureAlbaneserelaxed} and Conjecture \ref{conjectureAlbaneseIO} are equivalent.
 \end{proposition}
 
  \begin{proof}
  If Conjecture \ref{conjectureAlbaneseIO} holds, then Conjecture \ref{conjectureAlbaneserelaxed} holds by Proposition \ref{AlbaneseIAandIO2} and Lemma \ref{WiWoi}.
  
  Suppose that Conjecture \ref{conjectureAlbaneserelaxed} holds.
  We will show that Conjecture \ref{conjectureAlbaneseIO} also holds by induction on the homological degree $i$. 
  Suppose that we have $W^O_i\cong H^A_i(\IO_n,\Q)$ for sufficiently large $n$ with respect to $i$.
  For sufficiently large $n$ with respect to $i+1$, by using Lemma \ref{WiWoi} and Proposition \ref{AlbaneseIAandIO2}, we have 
  $$
  W^O_{i+1}\oplus (W^O_{i}\otimes H)\cong W_{i+1} \cong H^A_{i+1}(\IA_n,\Q)\cong H^A_{i+1}(\IO_n,\Q) \oplus (H^A_i(\IO_n,\Q)\otimes H).
  $$
  Therefore, since we have $H^A_i(\IO_n,\Q)\cong W^O_{i}$ by induction hypothesis, we have $H^A_{i+1}(\IO_n,\Q)\cong W^O_{i+1}$, which completes the proof.
 \end{proof}
 
 \begin{remark}
 Let $n\ge 3i$.
 Since we have $W_i\cong \bigoplus_{p+q=i} W^O_p\wti{\otimes} \bigwedge^q H$ by the definitions of $W_*$ and $W^O_*$, by Lemma \ref{WiWoi}, we have
 $$
  W^O_i \oplus (W^O_{i-1}\otimes H)\cong
  \bigoplus_{p+q=i} W^O_{p}\wti{\otimes} \bigwedge^q H.
 $$
 Therefore, if Conjecture \ref{conjectureAlbaneseIO} holds, then we stably have
 $$
 H^A_*(\IO_n,\Q)\otimes H^A_*(F_n,\Q)\cong  H^A_*(\IO_n,\Q)\wti{\otimes} H^A_*(F_n^{\ab},\Q).
 $$
 \end{remark}

\section{The third Albanese homology of $\IA_n$ and $\IO_n$}\label{ThirdhomologyIAandIO}
 In this section, we compute the third Albanese homology of $\IA_n$ and $\IO_n$, and prove that Conjectures \ref{conjectureAlbanese} and \ref{conjectureAlbaneseIO} hold for $i=3$.

  Hain \cite{Hain} and Sakasai \cite{Sakasai} computed the Albanese cohomology of the Torelli groups of closed surfaces of degree $2$ and $3$, respectively.
  Pettet \cite{Pettet} applied Hain's and Sakasai's methods to determine $H_A^2(\IA_n,\Q)$ and $H_A^2(\IO_n,\Q)$. 
  The method (applied to $H^A_3(\IA_n,\Q)$) is as follows.
  Let $R_3=\ker (\tau^*: H^3(U,\Q)\to H^3(\IA_n,\Q))$.
  Suppose that we have subrepresentations $S\subset H^A_3(\IA_n,\Q)$ and $T\subset R_3$.
  Then, we have 
  $$S^*\oplus T\subset H_A^3(\IA_n,\Q) \oplus R_3=\im \tau^*\oplus \ker \tau^*\cong H^3(U,\Q).$$
  If we can show that $H^3(U,\Q)\cong S^*\oplus T$ by counting multiplicities of irreducible components, then we obtain $H_A^3(\IA_n,\Q)\cong S^*$ and $R_3=T$. 
  Therefore, we have $H^A_3(\IA_n,\Q)= S$.
  We will use the above method to determine $H^A_3(\IA_n,\Q)$ and $H^A_3(\IO_n,\Q)$.

 \subsection{Irreducible decompositions of $H_3(U,\Q)$ and $H_3(U^O,\Q)$}
   We begin with the computation of decompositions of $H_3(U,\Q)$ and $H_3(U^O,\Q)$ into irreducible $\GL(n,\Q)$-representations. 
   
   We can check the following lemma by hand and by using SageMath.
   
   \begin{lemma}\label{H3U}
    Let $n\ge 9$.
    We have the following decomposition of $H_3(U,\Q)$ into the direct sum of $61$ irreducible $\GL(n,\Q)$-representations:
    \begin{gather*}
    \begin{split}
        H_3(U,\Q)
        &\cong \bigwedge^3 (\bigwedge^2 H\otimes H^*)\\
        &\cong V_{3^2, 1^3}
        \oplus V_{3 2 1, 2 1}
        \oplus V_{3 1^3, 3}
        \oplus V_{2^3, 3}
        \oplus V_{2^2 1^2, 2 1}
        \oplus V_{2^2 1^2, 1^3}
        \oplus V_{2 1^4, 2 1}
        \oplus V_{1^6, 1^3}\\
        &
        \oplus V_{3 2, 2}
        \oplus V_{3 2, 1^2}^{\oplus 2}
        \oplus V_{3 1^2, 2}^{\oplus 2}
        \oplus V_{3 1^2, 1^2}
        \oplus V_{2^2 1, 2}^{\oplus 3}
        \oplus V_{2^2 1, 1^2}^{\oplus 3}
        \oplus V_{2 1^3, 2}^{\oplus 3}
        \oplus V_{2 1^3, 1^2}^{\oplus 3}\\
        &
        \oplus V_{1^5, 2}
        \oplus V_{1^5, 1^2}^{\oplus 2}
        \oplus V_{3 1, 1}^{\oplus 2}
        \oplus V_{2^2, 1}^{\oplus 6}
        \oplus V_{2 1^2, 1}^{\oplus 7}
        \oplus V_{1^4, 1}^{\oplus 6}
        \oplus V_{3,0}
        \oplus V_{2 1,0}^{\oplus 4}
        \oplus V_{1^3,0}^{\oplus 6},
    \end{split}
    \end{gather*}
    and the following decomposition of
    $H_3(U^O,\Q)$ into the direct sum of $36$ irreducible $\GL(n,\Q)$-representations:
    \begin{gather*}
    \begin{split}
        H_3(U^O,\Q)
        &\cong \bigwedge^3 (V_{1^2, 1})\\
        &\cong V_{3^2, 1^3}
        \oplus V_{3 2 1, 2 1}
        \oplus V_{3 1^3, 3}
        \oplus V_{2^3, 3}
        \oplus V_{2^2 1^2, 2 1}
        \oplus V_{2^2 1^2, 1^3}
        \oplus V_{2 1^4, 2 1}
        \oplus V_{1^6, 1^3}\\
        &
        \oplus V_{3 2, 2}
        \oplus V_{3 2, 1^2}
        \oplus V_{3 1^2, 2}
        \oplus V_{3 1^2, 1^2}
        \oplus V_{2^2 1, 2}^{\oplus 2}
        \oplus V_{2^2 1, 1^2}^{\oplus 2}
        \oplus V_{2 1^3, 2}^{\oplus 2}
        \oplus V_{2 1^3, 1^2}^{\oplus 2}\\
        &
        \oplus V_{1^5, 2}
        \oplus V_{1^5, 1^2}
        \oplus V_{3 1, 1}
        \oplus V_{2^2, 1}^{\oplus 3}
        \oplus V_{2 1^2, 1}^{\oplus 3}
        \oplus V_{1^4, 1}^{\oplus 3}
        \oplus V_{3,0}
        \oplus V_{2 1,0}
        \oplus V_{1^3,0}^{\oplus 2}.
    \end{split}
    \end{gather*}
   \end{lemma}

 \subsection{The third Albanese homology of $\IA_n$ and $\IO_n$}
  For $n\ge 9$, we have
  \begin{gather*}
    \begin{split}
      &W(3,0)\cong V_{1^4,1}, \quad 
      W(0, 3)\cong V_{1^3},\\
      &W(21, 0)\cong V_{2^2 1, 2}\oplus V_{2^2 1, 1^2}\oplus V_{2 1^3, 2}\oplus V_{2 1^3, 1^2}\oplus V_{1^5, 2} \oplus V_{1^5, 1^2},\\
      &W(2, 1)\cong  V_{2 1^2, 1} \oplus V_{1^4, 1},\quad 
      W(1, 2)\cong  V_{2^2, 1} \oplus V_{2 1^2, 1}\oplus V_{1^4, 1},\\
      &W(0, 21)\cong  V_{2 1} \oplus V_{1^3},\\
      &W(1^3, 0)\cong   V_{3^2, 1^3}\oplus V_{3 2 1, 2 1}\oplus V_{3 1^3, 3}\oplus V_{2^3, 3}\oplus V_{2^2 1^2, 2 1}\oplus V_{2^2 1^2, 1^3}\oplus V_{2 1^4, 2 1}\oplus V_{1^6, 1^3},\\
      &W(1^2, 1)\cong V_{3 2, 1^2}\oplus V_{3 1^2, 2}\oplus V_{2^2 1, 2}\oplus V_{2^2 1, 1^2}\oplus V_{2 1^3, 2}\oplus V_{2 1^3, 1^2}\oplus V_{1^5, 1^2},\\
      &W(1, 1^2)\cong  V_{2^2, 1} \oplus V_{2 1^2, 1} \oplus V_{1^4, 1},\quad 
      W(0, 1^3)\cong V_{1^3}.
    \end{split}
  \end{gather*}

  We will prove the following theorems in the rest of this section. 
  
  \begin{theorem}\label{thirdAlbanese}
  Let $n\ge 9$.
  $H^A_3(\IA_n,\Q)$ is decomposed into $34$ irreducible $\GL(n,\Q)$-representations:
  \begin{gather*}
  \begin{split}
      H^A_3(\IA_n,\Q)&\cong W_3\\
      &= W(3,0)\oplus W(0,3)\oplus W(21,0)\oplus W(2,1)\oplus W(1,2)\oplus W(0,21)\\
      &\quad\oplus W(1^3,0)\oplus W(1^2,1)\oplus W(1,1^2)\oplus W(0,1^3)\\
      &\cong
               V_{3^2, 1^3}
        \oplus V_{3 2 1, 2 1}
        \oplus V_{3 1^3, 3}
        \oplus V_{2^3, 3}
        \oplus V_{2^2 1^2, 2 1}
        \oplus V_{2^2 1^2, 1^3}
        \oplus V_{2 1^4, 2 1}
        \oplus V_{1^6, 1^3}\\
        &
        \oplus V_{3 2, 1^2}
        \oplus V_{3 1^2, 2}
        \oplus V_{2^2 1, 2}^{\oplus 2}
        \oplus V_{2^2 1, 1^2}^{\oplus 2}
        \oplus V_{2 1^3, 2}^{\oplus 2}
        \oplus V_{2 1^3, 1^2}^{\oplus2}
        \oplus V_{1^5, 2}
        \oplus V_{1^5, 1^2}^{\oplus 2}\\
        &
        \oplus V_{2^2, 1}^{\oplus 2}
        \oplus V_{2 1^2, 1}^{\oplus 3}
        \oplus V_{1^4, 1}^{\oplus 4}
        \oplus V_{2 1,0}
        \oplus V_{1^3,0}^{\oplus 3}.
    \end{split}
  \end{gather*}
  \end{theorem}
 
 \begin{theorem}\label{thirdAlbaneseIO}
 Let $n\ge 9$.
 $H^A_3(\IO_n,\Q)$ is decomposed into $19$ irreducible $\GL(n,\Q)$-representations:
 \begin{gather*}
  \begin{split}
      H^A_3(\IO_n,\Q)&\cong  W^O_3
      =W(3,0)\oplus W(0,3)\oplus W(21,0)\oplus W(1,2)\oplus W(1^3,0)\\
      &\cong
               V_{3^2, 1^3}
        \oplus V_{3 2 1, 2 1}
        \oplus V_{3 1^3, 3}
        \oplus V_{2^3, 3}
        \oplus V_{2^2 1^2, 2 1}
        \oplus V_{2^2 1^2, 1^3}
        \oplus V_{2 1^4, 2 1}
        \oplus V_{1^6, 1^3}\\
        &
        \oplus V_{2^2 1, 2}
        \oplus V_{2^2 1, 1^2}
        \oplus V_{2 1^3, 2}
        \oplus V_{2 1^3, 1^2}
        \oplus V_{1^5, 2}
        \oplus V_{1^5, 1^2}\\
        &
        \oplus V_{2^2, 1}
        \oplus V_{2 1^2, 1}
        \oplus V_{1^4, 1}^{\oplus 2}
        \oplus V_{1^3,0}.
    \end{split}
  \end{gather*}
 \end{theorem}

\subsection{Proof of Theorems \ref{thirdAlbanese} and \ref{thirdAlbaneseIO}}

Here we will give the outline of the proof of Theorems \ref{thirdAlbanese} and \ref{thirdAlbaneseIO} since the proof includes somewhat technical computations of $\GL(n,\Q)$-representations, which we leave in Appendix \ref{AppendixB}.

Let $R_i=\ker (\tau^*: H^i(U,\Q)\to H^i(\IA_n,\Q))$ for $i=2,3$, and
\begin{gather*}
       \cup : H^1(U,\Q)\otimes R_2\to H^3(U,\Q)
\end{gather*}
denote the restriction of the cup product map.
Then we have $\im \cup \subset R_3$ since $\tau^*$ coincides with the cup product map.
Therefore, by Proposition \ref{imagecup}, we obtain the following subrepresentation of $R_3$.

 \begin{proposition}\label{R3}
   For $n\ge 6$, $R_3$ contains a $\GL(n,\Q)$-subrepresentation which is isomorphic to the direct sum of $27$ irreducible representations:
    \begin{gather*}
    \begin{split}
        &V_{2, 3 2}
        \oplus V_{1^2, 3 2}
        \oplus V_{2, 3 1^2}
        \oplus V_{1^2, 3 1^2}
        \oplus V_{2, 2^2 1}
        \oplus V_{1^2, 2^2 1}
        \oplus V_{2, 2 1^3}
        \oplus V_{1^2, 2 1^3}\\
       &\quad\quad\quad 
        \oplus V_{1, 3 1}^{\oplus 2}
        \oplus V_{1, 2^2}^{\oplus 4}
        \oplus V_{1, 2 1^2}^{\oplus 4}
        \oplus V_{1, 1^4}^{\oplus 2}
        \oplus V_{0, 3}
        \oplus V_{0, 2 1}^{\oplus 3}
        \oplus V_{0, 1^3}^{\oplus 3}.
    \end{split}
    \end{gather*}
   \end{proposition}

By the same argument, in the case of $\IO_n$, we also have the following subrepresentation of $R^O_3=\ker (\tau_O^*: H^3(U^O,\Q)\to H^3(\IO_n,\Q))$ by Proposition \ref{kerIO}.

\begin{proposition}\label{R3O}
   For $n\ge 6$, $R^O_3$ contains a $\GL(n,\Q)$-subrepresentation which is isomorphic to the direct sum of $17$ irreducible representations:
    \begin{gather*}
    \begin{split}
      &V_{2, 3 2}
        \oplus V_{1^2, 3 2}
        \oplus V_{2, 3 1^2}
        \oplus V_{1^2, 3 1^2}
        \oplus V_{2, 2^2 1}
        \oplus V_{1^2, 2^2 1}
        \oplus V_{2, 2 1^3}
        \oplus V_{1^2, 2 1^3}\\
       &\quad\quad\quad 
        \oplus V_{1, 3 1}
        \oplus V_{1, 2^2}^{\oplus 2}
        \oplus V_{1, 2 1^2}^{\oplus 2}
        \oplus V_{1, 1^4}
        \oplus V_{0, 3}
        \oplus V_{0, 2 1}
        \oplus V_{0, 1^3}.
    \end{split}
    \end{gather*}
   \end{proposition}

  Now we complete the proof of Theorems \ref{thirdAlbanese} and \ref{thirdAlbaneseIO}.

  \begin{proof}[Proof of Theorem \ref{thirdAlbanese}]
  By Lemma \ref{H3U}, we have an irreducible decomposition of $H_3(U,\Q)$, which consists of $61$ irreducibles.
  By Proposition \ref{R3}, $27$ of them are contained in $R_3$.
  By Theorem \ref{Johnsonpart}, $H^A_3(\IA_n,\Q)$ contains a subrepresentation which is isomorphic to $W_3$, which consists of $34$ irreducible components.
  Therefore, we obtain 
  $$
   H^A_3(\IA_n,\Q)\cong W_3,
  $$
  which completes the proof of Theorem \ref{thirdAlbanese}.     
  \end{proof}
  
  \begin{remark}\label{quadratic3}
  By the above proof, we have
  $$
   H^3(U,\Q)/\im\cup\xrightarrow{\cong} H_A^3(\IA_n,\Q).
  $$
  Since the image of the cup product map coincides with the degree $3$ part of the ideal $\langle R_2\rangle\subset H^*(U,\Q)$,
  Conjecture \ref{quadratic} holds for $*=3$. 
  \end{remark}

\begin{proof}[Proof of Theorem \ref{thirdAlbaneseIO}]
   By Lemma \ref{H3U}, we have an irreducible decomposition of $H_3(U^O,\Q)$, which consists of $36$ irreducibles.
   By Proposition \ref{R3O}, $17$ of them are contained in $R^O_3$.
   Therefore, it suffices to check that $H^A_3(\IO_n,\Q)$ contains a subrepresentation which is isomorphic to $W^O_3$ since $W^O_3$ is decomposed into $19$ irreducibles. 
   By Proposition \ref{IOnWO}, we have an injective $\GL(n,\Q)$-homomorphism 
   $$W^O_3 \oplus (W^O_2\otimes H)\hookrightarrow H^A_3(\IO_n,\Q)\oplus (H^A_2(\IO_n,\Q)\otimes H).$$
   Since we have $H^A_2(\IO_n,\Q)\cong W^O_2$ by Pettet \cite{Pettet}, we have $W^O_3\hookrightarrow H^A_3(\IO_n,\Q)$, which completes the proof.
\end{proof}  

\section{Cohomology of $\Aut(F_n)$ with twisted coefficients}\label{KV}
For an algebraic $\GL(n,\Q)$-representation $V$, let $H_A^*(\IA_n,V)=H_A^*(\IA_n,\Q)\otimes V$.
In this section, we study the relation between $H_A^*(\IA_n,V)$ and $H^*(\Aut(F_n),V)$.

\subsection{Wheeled PROPs and wheeled operads}
Here, we recall the notions of PROPs and operads, and wheeled versions of PROPs and operads.
See \cite{CFP, Kawazumi--Vespa, LodayVallette, MMS, Markl} for precise definitions.

A \emph{PROP} is a symmetric monoidal category $P=(P,\otimes,0, S)$ with non-negative integers as objects and $m\otimes n=m+n$ for any $m,n\in \N$.
We consider PROPs enriched over the category of graded $\Q$-vector spaces.

A \emph{wheeled PROP} \cite{MMS} is a PROP $P$ equipped with contraction maps (or partial trace maps)
$$\xi^i_j: P(m,n)\to P(m-1,n-1),$$
which can be viewed as connecting the $i^{\text{th}}$ input and the $j^{\text{th}}$ output, satisfying compatibility and unitality axioms.
A \emph{non-unital wheeled PROP} (or \emph{TRAP} in the sense of \cite{CFP}) is a wheeled PROP without identity morphisms and unitality axioms.

An \emph{operad} in the category of graded $\Q$-vector spaces is a collection $\calP=\{\calP(n)\}_{n\ge 0}$ of graded right $\Q[\gpS_n]$-modules equipped with operadic compositions, which are graded $\Q$-linear maps
$$\gamma: \calP(n)\otimes \calP(k_1)\otimes \cdots\otimes \calP(k_n)\to\calP(k_1+\cdots+k_n),$$
and a unit map 
$$\eta: \Q\to \calP(1)$$
satisfying associativity, equivariance and unitality axioms.
A \emph{non-unital operad} is an operad without unit.

For an operad $\calP$, a \emph{right $\calP$-module} is a graded $\gpS$-module $M=\{M(n)\}_{n\ge 0}$, which is a collection of graded $\gpS_n$-modules $M(n)$, 
equipped with right $\calP$-actions, which are graded $\Q$-linear maps
$$\alpha: M(l)\otimes \calP(m_1)\otimes \cdots\otimes \calP(m_l)\to M(m_1+\cdots+m_l),$$
satisfying the operadic form of the standard axioms for a right module over an algebra.
For a non-unital operad, we do not assume the unitality axioms.

A \emph{wheeled operad} $\calP=\{\calP(n,m)\}_{(n,m)\in \N\times \{0,1\}}$ consists of 
\begin{enumerate}[(i)]
    \item the operadic part: an operad $\calP_0=\{\calP(n,1)\}_{n\ge 0}$,\label{operadicpart}
    \item the wheeled part: a right $\calP_0$-module $\calP_w=\{\calP(n,0)\}_{n\ge 0}$,\label{wheeledpart}
    \item contraction maps $\xi_i: \calP_0(n)\to \calP_w(n-1)$ for $1\le i\le n$ satisfying compatibility with the structures \eqref{operadicpart} and \eqref{wheeledpart}.
\end{enumerate}
A \emph{non-unital wheeled operad} $\calP$ consists of a non-unital operad $\calP_0$, the wheeled part and the contraction maps as above.

The forgetful functor from the category of wheeled operads (resp. non-unital wheeled operads) to the category of operads (resp. non-unital operads) has a left adjoint denoted by $(-)^{\circlearrowright}$, which is called the \emph{wheeled completion}.

For any wheeled operad (resp. non-unital wheeled operad) $\calP$, there is a wheeled PROP (resp. non-unital wheeled PROP) $\calC_{\calP}$ which is freely generated by $\calP$.
We have
\begin{gather}\label{freePROP}
\calC_{\calP}(m,n)=
\bigoplus_{J\subset [m]}
\left( \bigoplus_{f: J\twoheadrightarrow [n]} \bigotimes_{i=1}^{n}\calP_0(|f^{-1}(i)|) \right)
\otimes \left(\bigoplus_{k\in \N} \left(\bigoplus_{g: [m]\setminus J\twoheadrightarrow [k]} \bigotimes_{i=1}^{k}\calP_w(|g^{-1}(i)|)\right)_{\gpS_{k}} \right),
\end{gather}
where $\gpS_{k}$ acts on $\bigoplus_{g: [m]\setminus J\twoheadrightarrow [k]} \bigotimes_{i=1}^{k}\calP_w(|g^{-1}(i)|)$ by postcomposing elements of $\gpS_{k}$ with $g: [m]\setminus J\twoheadrightarrow [k]$.
(See \cite[Proposition 2.3]{Kawazumi--Vespa} for details.)

\subsection{Stable cohomology of $\Aut(F_n)$ with twisted coefficients}  

Here, we recall the conjectural structure of the stable cohomology of $\Aut(F_n)$ with twisted coefficients given by Kawazumi--Vespa \cite{Kawazumi--Vespa}.

For $p,q\ge 0$, let $H^{p,q}=H^{\otimes p}\otimes (H^*)^{\otimes q}$.
In \cite{Kawazumi--Vespa}, Kawazumi and Vespa have studied the structure of the stable cohomology $H^*(\Aut(F_n),H^{p,q})$ for $p,q\ge 0$.
They defined a wheeled PROP $\calH$ such that 
$$\calH(p,q)= \lim_{n\in \N} H^*(\Aut(F_n), H^{p,q}).$$

They also defined a wheeled PROP $\calE$ such that 
$$\calE(p,q)=\bigoplus_{j\in \N} \Ext^{*-j}_{\calF(\mathbf{gr};\Q)}(\mathfrak{a}^{\otimes q}\otimes \bigwedge^j \mathfrak{a} , \mathfrak{a}^{\otimes p})$$
is the direct sum of $\Ext$-groups in the category $\calF(\mathbf{gr};\Q)$ of functors from the category $\mathbf{gr}$ of finitely generated free groups to the category of $\Q$-vector spaces, where $\mathfrak{a}\in \calF(\mathbf{gr};\Q)$ is the abelianization functor.

Let $\calP_0=\bigoplus_{k\ge 1} \calP_0(k)$ denote the operadic suspension of the operad $\calC om$ of non-unital commutative algebras. That is, $\calP_0$ is an operad such that $ \calP_0(0)=0$ and $\calP_0(k)=\sgn_k[k-1]$ for $k\ge 1$, where $\sgn_k[k-1]$ is the sign representation of $\gpS_k$ placed in cohomological dimension $k-1$.
Let $\calP_0^{\circlearrowright}$ denote the wheeled completion of $\calP_0$ and $\calC_{\calP_0^{\circlearrowright}}$ the wheeled PROP freely generated by the wheeled operad $\calP_0^{\circlearrowright}$.
Then they constructed a wheeled PROP isomorphism $\calC_{\calP_0^{\circlearrowright}} \xto{\cong} \calE $.

They constructed a morphism of wheeled PROPs $\varphi:\calH\to \calE$ and proposed the following conjecture.

\begin{conjecture}[Kawazumi--Vespa {\cite[Conjecture 6]{Kawazumi--Vespa}}]\label{conjectureKVorigin}
The morphism of wheeled PROP
$$\varphi: \calH\to \calE$$
is an isomorphism.
\end{conjecture}

Conjecture \ref{conjectureKVorigin} is equivalent to the following conjecture.

\begin{conjecture}[Kawazumi--Vespa]\label{conjectureKV}
Let $i,p,q$ be non-negative integers.
Then, for sufficiently large $n$, we have an isomorphism of $\Q[\gpS_{p}\times \gpS_{q}]$-modules
$$
  H^{i}(\Aut(F_n), H^{p,q})=
  \begin{cases}
    \calC_{\calP_0^{\circlearrowright}}(p,q)& (i=p-q)\\
    0 & (i\neq p-q).
  \end{cases} 
$$
\end{conjecture}

\subsection{Reconstruction of the traceless part $W_*$}

We defined the traceless part $W_*$ of the graded-symmetric algebra $S^*(U_*)$ of $U_*$ in Section \ref{Generalpart}.
Here, we reconstruct $W_*$ by using the operad $\calC om$.

Let $\calO$ denote the non-unital suboperad of $\calP_0$ such that $\calO(0)=\calO(1)=0$ and $\calO(k)=\calP_0(k),\; k\ge 2$.
We have a non-unital wheeled sub-PROP $\calC_{\calO^{\circlearrowright}}$ of $\calC_{\calP_0^{\circlearrowright}}$ associated to the non-unital wheeled operad $\calO^{\circlearrowright}$.
Then we have
$\calO^{\circlearrowright}_{0}=\{\calO^{\circlearrowright}(n,1)\}_{n\ge 2}$ and  $\calO^{\circlearrowright}_{w}=\{\calO^{\circlearrowright}(n,0)\}_{n\ge 1}$.
We can reconstruct $W_*$ by using the non-unital wheeled PROP $\calC_{\calO^{\circlearrowright}}$.

\begin{proposition}\label{Wiprop}
Let $i$ be a non-negative integer. Then, for sufficiently large $n$, we have a $\GL(n,\Q)$-isomorphism
\begin{gather*}
W_i\cong \bigoplus_{a-b=i}T_{a,b}\otimes_{\Q[\gpS_a\times \gpS_b]} \calC_{\calO^{\circlearrowright}}(a,b).
\end{gather*}
Here the direct sum is finite.
\end{proposition}

\begin{proof}
Since $T_{b+i,b}$ is the traceless part of $H^{b+i,b}$ and we have $W_i=\wti{S}^* (U_*)_i$, it suffices to show that 
$$
S^* (U_*)_i\cong \bigoplus_{b\ge 0}H^{b+i,b} \otimes_{\Q[\gpS_{b+i}\times \gpS_b]} \calC_{\calO^{\circlearrowright}}(b+i,b).
$$
By \eqref{freePROP}, we have
$$
    \calC_{\calO^{\circlearrowright}}(b+i,b)
    =\bigoplus_{J\subset [b+i]}\left(\bigoplus_{f:J\twoheadrightarrow [b]} \bigotimes_{j=1}^{b}\calO^{\circlearrowright}_0(|f^{-1}(j)|) \right)\otimes \left(\bigoplus_{k\in \N}\left(\bigoplus_{g: [b+i]\setminus J\twoheadrightarrow [k]} \bigotimes_{j=1}^{k}\calO^{\circlearrowright}_w(|g^{-1}(j)|) \right)_{\gpS_{k}}\right).
$$
Since we have $\calO^{\circlearrowright}_{0}=\{\calO^{\circlearrowright}(n,1)\}_{n\ge 2}$, we have only to consider $f$ such that $|f^{-1}(j)|\ge 2$ for each $j\in [b]$. Therefore, we have
\begin{gather*}
\begin{split}
&\calC_{\calO^{\circlearrowright}}(b+i,b)\\
&=\bigoplus_{\substack{(\mu,\nu)\vdash i \\ l(\mu)=b}} 
\operatorname{Ind}_
{\prod_{j=1}^{r}(\gpS_{\mu_j+1}\wr \gpS_{k'_j})\times \prod_{j=1}^{s}(\gpS_{\nu_j}\wr \gpS_{k''_j})}
^{\gpS_{b+i}\times \gpS_{b}}
\left(\bigotimes_{j=1}^{r} \calO(\mu_j+1)^{\otimes k'_j} \right)
\otimes 
\left(\bigotimes_{j=1}^{s} \calO_w^{\circlearrowright}(\nu_j)^{\otimes k''_j} \right),
\end{split}
\end{gather*}
where we write $\mu=(\mu_1^{k'_1},\cdots,\mu_r^{k'_r})$ and $\nu=(\nu_1^{k''_1},\cdots,\nu_s^{k''_s})$, and where we consider $\prod_{j=1}^{r}(\gpS_{\mu_j+1}\wr \gpS_{k'_j})\times \prod_{j=1}^{s}(\gpS_{\nu_j}\wr \gpS_{k''_j})$ as a subgroup of $\gpS_{b+i}\times \{1\}\subset \gpS_{b+i}\times \gpS_{b}$.

Let $\gpS(\mu)=\prod_{j=1}^{r}(\gpS_{\mu_j+1}\wr \gpS_{k'_j})$, $\gpS(\nu)=\prod_{j=1}^{s}(\gpS_{\nu_j}\wr \gpS_{k''_j})$ and $\gpS(\mu,\nu)=\gpS(\mu)\times \gpS(\nu)$.
Since we have $\calO(\mu_j+1)=\sgn_{\mu_j+1}[\mu_j]$ and $\calO_w^{\circlearrowright}(\nu_j)=\sgn_{\nu_j}[\nu_j-1]$, we have
\begin{gather*}
\begin{split}
&\bigoplus_{b\ge 0}H^{b+i,b} \otimes_{\Q[\gpS_{b+i}\times \gpS_b]} \calC_{\calO^{\circlearrowright}}(b+i,b)\\
&\cong 
\bigoplus_{b\ge 0}\bigoplus_{\substack{(\mu,\nu)\vdash i \\ l(\mu)=b}} 
H^{b+i,b} \otimes_{\Q[\gpS_{b+i}\times \gpS_b]}
\operatorname{Ind}_{\gpS(\mu,\nu)}^{\gpS_{b+i}\times \gpS_{b}}
\left(\bigotimes_{j=1}^{r} \calO(\mu_j+1)^{\otimes k'_j} \right)
\otimes 
\left(\bigotimes_{j=1}^{s} \calO_w^{\circlearrowright}(\nu_j)^{\otimes k''_j} \right)\\
&\cong 
\bigoplus_{b\ge 0}\bigoplus_{\substack{(\mu,\nu)\vdash i \\ l(\mu)=b}} 
H^{b+i,b} \otimes_{\Q[\gpS(\mu,\nu)]}
\left(\bigotimes_{j=1}^{r} \calO(\mu_j+1)^{\otimes k'_j} \right)
\otimes 
\left(\bigotimes_{j=1}^{s} \calO_w^{\circlearrowright}(\nu_j)^{\otimes k''_j} \right)\\
&\cong 
\bigoplus_{b\ge 0}\bigoplus_{\substack{(\mu,\nu)\vdash i \\ l(\mu)=b}} 
\bigotimes_{j=1}^{r} \left(H^{k'_j(\mu_j+1), k'_j} \otimes_{\Q[\gpS_{\mu_j+1}\wr \gpS_{k'_j}]} \calO(\mu_j+1)^{\otimes k'_j}\right)\\
&\quad \quad \quad \quad \quad \otimes 
\bigotimes_{j=1}^{s} \left(H^{k''_j \nu_j, 0} \otimes_{\Q[\gpS_{\nu_j}\wr \gpS_{k''_j}]} \calO_w^{\circlearrowright}(\nu_j)^{\otimes k''_j}\right)
\\
&\cong 
\bigoplus_{b\ge 0}\bigoplus_{\substack{(\mu,\nu)\vdash i \\ l(\mu)=b}} 
 \left(\bigotimes_{j=1}^r S^{k_j'}\left(U_{\mu_j}^{\tree}\right)  \otimes\;
  \bigotimes_{j=1}^s S^{k_j''}\left(U_{\nu_j}^{\wheel}\right)\right)\\
&\cong 
S^* (U_*)_i.
\end{split}
\end{gather*}
This completes the proof.
\end{proof}

By Proposition \ref{Wiprop}, for sufficiently large $n$, we have a $\GL(n,\Q)$-isomorphism
\begin{gather}\label{W^*PROP}
W_i^*\cong \bigoplus_{a-b=i}T_{b,a}\otimes_{\Q[\gpS_a\times \gpS_b]} \calC_{\calO^{\circlearrowright}}(a,b).
\end{gather}
Then we have the following relation between $W_i^*$ and the wheeled PROP $\calC_{\calP_0^{\circlearrowright}}$.

\begin{lemma}\label{lemmawheeledPROP}
 Let $i,p,q$ be non-negative integers.
 Then, for sufficiently large $n$, we have an isomorphism of $\Q[\gpS_{p}\times \gpS_{q}]$-modules
 $$
 (W_i^* \otimes H^{p,q})^{\GL(n,\Z)}
 \cong 
 \begin{cases}
   \calC_{\calP_0^{\circlearrowright}}(p,q)& (i=p-q)\\
   0 & (i\neq p-q).
 \end{cases}
 $$
\end{lemma}

\begin{proof}
Let $r=\min(p,q)$.
For $0\le c\le r$, let $K_c$ denote the kernel of the direct sum of all of the ways of the compositions of $(c+1)$ contraction maps
\begin{gather*}
    H^{p,q}\to (H^{p-c-1,q-c-1})^{\oplus \binom{p}{c}\binom{q}{c}c!}.
\end{gather*}
Then we have
\begin{gather*}
    0\subset K_0=T_{p,q}\subset K_1\subset \dots \subset K_r=H^{p,q},
\end{gather*}
and short exact sequences of $\GL(n,\Q)$-representations
\begin{gather*}
    0\to K_{c-1}\to K_{c}\to (T_{p-c,q-c})^{\oplus \binom{p}{c}\binom{q}{c}c!}.
\end{gather*}
Hence, we have $H^{p,q}\cong \bigoplus_{c=0}^{r}T_{p-c,q-c}^{\oplus \binom{p}{c} \binom{q}{c} c!}$.
Therefore, by \eqref{W^*PROP}, we stably have isomorphisms of $\Q[\gpS_{p}\times \gpS_{q}]$-modules
\begin{gather*}
 \begin{split}
    &((W_*)^* \otimes H^{p,q})^{\GL(n,\Z)}\\
     &\cong \left(\left(\bigoplus_{a,b}T_{b,a}\otimes_{\Q[\gpS_a\times \gpS_b]}\calC_{\calO^{\circlearrowright}}(a,b)\right)\otimes \left(\bigoplus_{c=0}^{r}T_{p-c,q-c}^{\oplus \binom{p}{c} \binom{q}{c} c!}\right)\right)^{\GL(n,\Z)}\\
     &\cong \bigoplus_{a,b}\bigoplus_{c=0}^{r} \left(T_{p-c,q-c}^{\oplus \binom{p}{c} \binom{q}{c} c!}\otimes T_{b,a}\right)^{\GL(n,\Z)} \otimes_{\Q[\gpS_a\times \gpS_b]}\calC_{\calO^{\circlearrowright}}(a,b) \\
     &\cong \bigoplus_{c=0}^{r}\left(T_{p-c,q-c}^{\oplus \binom{p}{c} \binom{q}{c} c!}\otimes T_{q-c,p-c}\right)^{\GL(n,\Z)} \otimes_{\Q[\gpS_{p-c}\times \gpS_{q-c}]}\calC_{\calO^{\circlearrowright}}(p-c,q-c) \\
     &\cong \bigoplus_{c=0}^{r}\left(T_{p-c,q-c}\otimes T_{q-c,p-c}\right)^{\GL(n,\Z)} \otimes_{\Q[\gpS_{p-c}\times \gpS_{q-c}]}\calC_{\calO^{\circlearrowright}}(p-c,q-c) ^{\oplus \binom{p}{c} \binom{q}{c} c!}\\
     &\cong \bigoplus_{c=0}^{r} \Q[\gpS_{p-c}\times \gpS_{q-c}] \otimes_{\Q[\gpS_{p-c}\times \gpS_{q-c}]}\calC_{\calO^{\circlearrowright}}(p-c,q-c) ^{\oplus \binom{p}{c} \binom{q}{c} c!}\\
     &\cong \bigoplus_{c=0}^{r}\calC_{\calO^{\circlearrowright}}(p-c,q-c)^{\oplus \binom{p}{c} \binom{q}{c} c!}\\
     &\cong \calC_{\calP_0^{\circlearrowright}}(p,q),
 \end{split}
\end{gather*}
where the last isomorphism can be checked by using \eqref{freePROP} since $\calP_0$ consists of $\calO$ and $\calP_0(1)=\Q$, which is generated by the unit.
\end{proof}

\subsection{Relation between $H_A^*(\IA_n, V)$ and $H^*(\Aut(F_n), V)$}

We make the following conjecture about the Albanese cohomology of $\IA_n$ and the cohomology of $\Aut(F_n)$ with twisted coefficients.

\begin{conjecture}\label{cohomologyAut}
Let $i$ be a non-negative integer and $\ul\lambda$ a bipartition.
Then, for sufficiently large $n$, we have a linear isomorphism
$$
  H^i(\Aut(F_n),V_{\ul\lambda})\cong H_A^i(\IA_n,V_{\ul\lambda})^{\GL(n,\Z)}.
$$
\end{conjecture}

We can check that Conjecture \ref{cohomologyAut} holds for $i\le 3$ and $\ul\lambda=(\lambda^+,0), (0,\lambda^-)$.
Conjecture \ref{cohomologyAut} is equivalent to the following conjecture since $H^{p,q}$ is decomposed into $V_{\ul\lambda}$'s.

\begin{conjecture}\label{cohomologyAutH}
Let $i,p,q$ be non-negative integers.
Then, for sufficiently large $n$, we have an isomorphism of $\Q[\gpS_{p}\times \gpS_{q}]$-modules 
$$
  H^i(\Aut(F_n),H^{p,q})\cong H_A^i(\IA_n,H^{p,q})^{\GL(n,\Z)}.
$$
\end{conjecture}

We have the following relation between conjectures about the structure of $H_A^*(\IA_n,\Q)$ and $H^*(\Aut(F_n),V)$.

\begin{proposition}\label{AlbaneseandAut}
Let $i$ be a non-negative integer.
 If two of the followings hold, then so does the third:
 \begin{enumerate}
     \item Conjecture \ref{conjectureAlbaneserelaxed} holds for homological degree $i$.
     \item Conjecture \ref{conjectureKV} holds for cohomological degree $i$.
     \item Conjecture \ref{cohomologyAut} holds for cohomological degree $i$.
 \end{enumerate}
\end{proposition}

\begin{proof}
Since $W_i^*$ and $H^i_A(\IA_n,\Q)$ are algebraic $\GL(n,\Q)$-representations, (1) is equivalent to the following statement:
for any $p,q$, for sufficiently large $n$, we have 
$$(H^i_A(\IA_n,\Q)\otimes H^{p,q})^{\GL(n,\Z)}\cong (W_i^* \otimes H^{p,q})^{\GL(n,\Z)}.$$

We have the following loop of possible isomorphisms for sufficiently large $n$
\begin{gather*}
 \begin{split}
     H^i_A(\IA_n,H^{p,q})^{\GL(n,\Z)}
     &\cong  (H^i_A(\IA_n,\Q)\otimes H^{p,q})^{\GL(n,\Z)}\\
     &\underset{(1)}{\cong} (W_i^* \otimes H^{p,q})^{\GL(n,\Z)}\\
     &\underset{\text{Lemma \ref{lemmawheeledPROP}}}{\cong} 
     \begin{cases}
       \calC_{\calP_0^{\circlearrowright}}(p,q)& (i=p-q)\\
       0 & (i\neq p-q)
     \end{cases}\\
     &\underset{(2)}{\cong} H^i(\Aut(F_n),H^{p,q})\\
     &\underset{(3)}{\cong} H^i_A(\IA_n,H^{p,q})^{\GL(n,\Z)}.
 \end{split}
\end{gather*}
 Therefore, if two of (1), (2) and (3) hold, then so does the third.
\end{proof}

It is well known that $\Aut(F_{2g})$ includes the mapping class group $\calM_{g,1}$ of a surface of genus $g$ with one boundary component as a subgroup.
For the cohomology of $\calM_{g,1}$ with twisted coefficients, we stably have the following isomorphism
\begin{equation}\label{mcglie}
    H^*(\calM_{g,1},H^{\otimes p})\cong (H^*(\Gr\mathfrak{u}_{g,1},\Q)\otimes H^{\otimes p})^{\Sp(2g,\Q)}
\end{equation}
(see \cite[Theorem 16.3]{Hainsurvey}, and see Section \ref{secLiecohom} for the definition of $\Gr \mathfrak{u}_{g,1}$).
We also have the variants of \eqref{mcglie} for the mapping class groups $\calM_{g}$ of a closed surface of genus $g$ and $\calM_{g}^{1}$ of a surface of genus $g$ with one marked point, respectively, as Garoufalidis--Getzler \cite{GG} first claimed.

A natural analogy of the isomorphism \eqref{mcglie} to $\Aut(F_n)$ is the following.

\begin{conjecture}\label{cohomologyAutHlie}
Let $i,p,q$ be non-negative integers.
For sufficiently large $n$, we have an isomorphism of $\Q[\gpS_{p}\times \gpS_{q}]$-modules 
$$
  H^i(\Aut(F_n),H^{p,q})\cong (H^i(\Gr \IA_n \otimes \Q,\Q)\otimes H^{p,q})^{\GL(n,\Z)},
$$
where $\Gr \IA_n$ is the graded Lie algebra of $\IA_n$ associated to the Andreadakis filtration of $\Aut(F_n)$.
\end{conjecture}

Based on Conjectures \ref{cohomologyAutH} and \ref{cohomologyAutHlie}, we make the following conjecture.

\begin{conjecture}\label{IAalbandGRIA}
We stably have an isomorphism of graded $\GL(n,\Q)$-representations
$$H_A^*(\IA_n,\Q)\cong H^*(\Gr \IA_n \otimes \Q,\Q).$$
\end{conjecture}

If $\Gr \IA_n\otimes \Q$ is stably quadratically presented and stably Koszul, then we stably have a surjective morphism of graded $\GL(n,\Q)$-representations $H^*(\Gr \IA_n \otimes \Q,\Q)\twoheadrightarrow H_A^*(\IA_n,\Q)$. See Section \ref{conjTorellialb} for the cases of the Torelli groups. 

\section{Algebraic $\Sp(2g,\Q)$-representations}\label{Sprep}
In this section, we recall representation theory of $\Sp(2g,\Q)$ and introduce the notion of traceless tensor products of $\Sp(2g,\Q)$-representations by adapting definitions in Section \ref{Preliminary}.

\subsection{Irreducible $\Sp(2g,\Q)$-representations}

Here, we fix a symplectic basis $\{a_i,b_i\mid 1\le i\le g\}$ for $H_1(\Sigma_g,\Q)$, and let $H=H_1(\Sigma_g,\Q)$, where $\Sigma_g$ is a closed surface of genus $g$.
Let $Q: H\otimes H\to \Q$ denote the \emph{symplectic form} such that $Q(a_i,b_i)=-Q(b_i,a_i)=1$ for each $1\le i\le g$.

For distinct elements $k, l\in [p]$, the \emph{contraction map} 
$$
 c^{\Sp}_{k,l}: H^{\otimes p}\to H^{\otimes p-2}
$$
is defined in a way similar to that of $\GL(n,\Q)$-representations by using the symplectic form $Q$.
We can define the \emph{traceless part} $T_{p}$ of $H^{\otimes p}$ as $\Sp(2g,\Q)$-representations by
$$
 T_{p}=\bigcap_{1\le k<l\le p} \ker c^{\Sp}_{k,l}\subset H^{\otimes p}.
$$
For a partition $\lambda\vdash p$, let 
$$V^{\Sp}_{\lambda}=T_{p}\cap V_{\lambda}\subset H^{\otimes p}.$$
If $\lambda$ has at most $g$ parts, then $V^{\Sp}_{\lambda}$ is an irreducible $\Sp(2g,\Q)$-representation corresponding to $\lambda$, and otherwise, we have $V^{\Sp}_{\lambda}=0$.
The irreducible representation $V^{\Sp}_{\lambda}$ can also be constructed as the image of the Young symmetrizer $c_{\lambda}: T_{p}\to T_{p}$, and it follows that $V^{\Sp}_{\lambda}$ is generated by $(a_1\wedge \cdots \wedge a_{\mu_1})\otimes \cdots \otimes (a_1\wedge \cdots \wedge a_{\mu_l})$, where $\mu=(\mu_1,\cdots, \mu_l)$ is the conjugate of $\lambda$.
See \cite{FH, Lindell} for more details.

We have the following irreducible decomposition of the tensor product of two irreducible $\Sp(2g,\Q)$-representations $V^{\Sp}_{\lambda}$ and $V^{\Sp}_{\mu}$
 $$
   V^{\Sp}_{\lambda}\otimes V^{\Sp}_{\mu}=\bigoplus_{\nu} (V^{\Sp}_{\nu})^{\oplus (N^{\Sp})_{\lambda \mu}^{\nu}},\quad 
   (N^{\Sp})_{\lambda\mu}^{\nu}=\sum_{\zeta,\sigma,\tau}N_{\zeta\sigma}^{\lambda} \; N_{\zeta\tau}^{\mu}\; N_{\sigma\tau}^{\nu},
 $$ 
where the $N$'s denote the Littlewood--Richardson coefficients (see \cite[Section 25.3]{FH}).

We call an $\Sp(2g,\Q)$-representation $V$ \emph{algebraic} if after choosing a basis for $V$, the $(\dim V)^2$ coordinate functions of the action $\Sp(2g,\Q)\to \GL(V)$ are rational functions of the $(2g)^2$ variables.

\subsection{Traceless tensor products of algebraic $\Sp(2g,\Q)$-representations}

The \emph{traceless tensor products} of algebraic $\Sp(2g,\Q)$-representations can be defined in a way similar to those of $\GL(n,\Q)$-representations in Section \ref{Preliminary}.

Let $\lambda \vdash p$ and $\mu\vdash q$.
The \emph{traceless tensor product} 
$V^{\Sp}_{\lambda}\wti{\otimes} V^{\Sp}_{\mu}$ of $V^{\Sp}_{\lambda}$ and $V^{\Sp}_{\mu}$ is
$$V^{\Sp}_{\lambda}\wti{\otimes} V^{\Sp}_{\mu}=(V^{\Sp}_{\lambda}\otimes V^{\Sp}_{\mu})\cap T_{p+q}\subset  H^{\otimes (p+q)}.$$
In other words, $V^{\Sp}_{\lambda}\wti{\otimes} V^{\Sp}_{\mu}$ is a subrepresentation of $V^{\Sp}_{\lambda}\otimes V^{\Sp}_{\mu}$ which vanishes under any contraction maps.
For $g\ge l(\lambda)+l(\mu)$, we have
$$V^{\Sp}_{\lambda}\wti{\otimes} V^{\Sp}_{\mu}\cong  \bigoplus_{|\nu|= p+q} (V^{\Sp}_{\nu})^{\oplus (N^{\Sp})_{\lambda \mu}^{\nu}}\subset
\bigoplus_{\nu} (V^{\Sp}_{\nu})^{\oplus (N^{\Sp})_{\lambda \mu}^{\nu}}
\cong V^{\Sp}_{\lambda}\otimes V^{\Sp}_{\mu}.$$

Let $M$ be an algebraic $\Sp(2g,\Q)$-representation. For each partition $\lambda$, define a vector space $$M_{\lambda}=\Hom_{\Sp(2g,\Q)}(V^{\Sp}_{\lambda},M).$$
Since the category of algebraic $\Sp(2g,\Q)$-representations is semisimple, we have a natural isomorphism 
$$
  \iota_M: M\xrightarrow{\cong} \bigoplus_{\lambda}V^{\Sp}_{\lambda}\otimes M_{\lambda}.
$$
For two algebraic $\Sp(2g,\Q)$-representations $M$ and $N$, we have an isomorphism
$$
  \iota_M\otimes \iota_N: M\otimes N \xrightarrow{\cong} \left(\bigoplus_{\lambda}V^{\Sp}_{\lambda}\otimes M_{\lambda}\right) 
  \otimes 
  \left(\bigoplus_{\mu}V^{\Sp}_{\mu}\otimes N_{\mu}\right)
  \cong 
  \bigoplus_{\lambda,\mu}(V^{\Sp}_{\lambda}\otimes V^{\Sp}_{\mu})\otimes (M_{\lambda}\otimes N_{\mu}).
$$
The \emph{traceless tensor product} $M \wti{\otimes} N$ of $M$ and $N$ is defined by
$$
  M \wti{\otimes} N=(\iota_{M}\otimes \iota_{N})^{-1} \left(\bigoplus_{\lambda,\mu}(V^{\Sp}_{\lambda}\wti{\otimes} V^{\Sp}_{\mu})\otimes (M_{\lambda}\otimes N_{\mu})\right)\subset M\otimes N.
$$
The \emph{traceless part} $\wti{T}^*M$ (resp. $\wti{\bigwedge}^* M$, $\wti{\Sym}^* M$) of the tensor algebra $T^*M$ (resp. the exterior algebra $\bigwedge^* M$, the symmetric algebra $\Sym^* M$) is defined in the same way as in Section \ref{Preliminary}.

Let $M_*=\bigoplus_{i\ge 1} M_i$ be a graded algebraic $\Sp(2g,\Q)$-representation. 
We can also define the \emph{traceless part} $\wti{S}^*(M_*)$ of the graded-symmetric algebra $S^*(M_*)$ as the image of $\wti{T}^*(M_*)$ under the projection $T^*(M_*)\twoheadrightarrow S^*(M_*)$.

We have the following relation between the traceless tensor products of $\Sp(2g,\Q)$-representations and those of $\GL(2g,\Q)$-representations.

\begin{lemma}\label{GLSPtl}
Let $M$ and $N$ be algebraic $\GL(2g,\Q)$-representations, which can be considered as algebraic $\Sp(2g,\Q)$-representations by restriction.
Let $M\wti{\otimes}^{\GL} N$ (resp. $M\wti{\otimes}^{\Sp} N$) denote the traceless tensor product of $M$ and $N$ as $\GL(2g,\Q)$-representations (resp. as $\Sp(2g,\Q)$-representations).
Then we have
\begin{equation}\label{tracelessSpGL}
    M\wti{\otimes}^{\Sp} N \subset M\wti{\otimes}^{\GL} N.
\end{equation}
\end{lemma}

\begin{proof}
It suffices to prove \eqref{tracelessSpGL} for $M=V_{\ul\lambda}\subset H^{p,q}$ and $N=V_{\ul\mu}\subset H^{r,s}$, where $\ul\lambda$ and $\ul\mu$ are bipartitions.
By the definitions of the traceless tensor products, we have
$$
 V_{\ul\lambda}\wti{\otimes}^{\GL} V_{\ul\mu}=(V_{\ul\lambda}\otimes V_{\ul\mu})\cap T_{p+r,q+s}
$$
and 
$$
 V_{\ul\lambda}\wti{\otimes}^{\Sp} V_{\ul\mu}= (V_{\ul\lambda}\otimes V_{\ul\mu})\cap \bigcap_{(k,l)\in J} \ker c^{\Sp}_{k,l},
$$
where $J$ consists of elements $(k,l)\in[p+q+r+s]^2$ such that $(k,l)\notin \{1,\cdots,p\}^2\cup \{p+1,\cdots,p+q\}^2\cup \{p+q+1,\cdots,p+q+r\}^2\cup \{p+q+r+1,\cdots,p+q+r+s\}^2$.
Since we have $\bigcap_{(k,l)\in J} \ker c^{\Sp}_{k,l} \subset T_{p+r,q+s}$, it follows that
$V_{\ul\lambda}\wti{\otimes}^{\Sp} V_{\ul\mu} \subset V_{\ul\lambda}\wti{\otimes}^{\GL} V_{\ul\mu}$.
\end{proof}

\section{On the Albanese cohomology of the Torelli groups}\label{problems}

Let $\I_{g}$ (resp. $\I_{g,1}$, $\I_{g}^{1}$) denote the Torelli group of an oriented closed surface of genus $g$ (resp. with one boundary component, with one marked point).
Here we adapt our arguments on $H^A_*(\IA_n,\Q)$ to $H^A_*(\I_{g},\Q)$, $H^A_*(\I_{g,1},\Q)$ and $H^A_*(\I_{g}^{1},\Q)$, which are algebraic $\Sp(2g,\Q)$-representations.
Note that the Albanese homology of the Torelli group is isomorphic to the Albanese cohomology of the Torelli group since the algebraic $\Sp(2g,\Q)$-representations are self-dual.

The Albanese cohomology of $\I_{g}$ of degree $\le 3$ has been determined by Johnson \cite{Johnson}, Hain \cite{Hain}, Sakasai \cite{Sakasai} and Kupers--Randal-Williams \cite{KRW}.

\subsection{Lindell's result about the Albanese homology of $\I_{g,1}$}\label{secLindell}

Lindell \cite{Lindell} detected some subquotient $\Sp(2g,\Q)$-representations of the Albanese homology of $\I_{g,1}$.

\begin{theorem}[{Lindell \cite[Theorem 1.5]{Lindell}}]
For a pair of partitions $(\lambda,\mu)\vdash i$ such that $\lambda=(\lambda_1^{k_1}>\cdots >\lambda_m^{k_m}>0)$, $\mu=(\mu_1^{l_1}>\cdots>\mu_{m+2}^{l_{m+2}}>0)$, $\mu_j=\lambda_j+2$,
for $g\ge i+2l(\lambda)+l(\mu)$,
the subrepresentation $W^{\I_{g,1}}_i(\lambda,\mu) \subset \bigotimes_{j=1}^{m+2}\bigwedge^{k_j+l_j}V^{\Sp}_{1^{\mu_j}}$
spanned by all irreducible subrepresentations of weight $i+2l(\lambda)$ is a subquotient $\Sp(2g,\Q)$-representation of $H^A_i(\I_{g,1},\Q)$.
\end{theorem}

We expect that one can adapt our argument of Theorem \ref{Johnsonpart} to the above result of Lindell to show that the direct sum $\overline{W}^{\I_{g,1}}_i$ of $W^{\I_{g,1}}_i(\lambda,\mu)$ is a subquotient representation of $H^A_i(\I_{g,1},\Q)$.
In Section \ref{conjTorellialb}, we will propose a conjectural structure of $H^A_*(\I_{g,1},\Q)$.

\subsection{Structures of the cohomology of Lie algebras associated to the Torelli groups}\label{secLiecohom}

Let $\mathfrak{t}_{g}$ denote the Malcev Lie algebra of $\I_{g}$, and $\mathfrak{u}_{g}$ the Lie algebra of the prounipotent radical of the relative Malcev completion of the mapping class group $\calM_{g}$ with respect to the short exact sequence $1\to \I_{g}\to \calM_{g}\to \Sp(2g,\Z)\to 1$. (See \cite{Hain} and \cite{KRW21} for details.)
Let $\Gr \mathfrak{t}_{g}$ (resp. $\Gr \mathfrak{u}_{g}$) denote the graded Lie algebra of $\mathfrak{t}_{g}$ (resp. $\mathfrak{u}_{g}$) associated to the lower central series.
The associated Lie algebras $\Gr \mathfrak{t}_{g,1}$ and $\Gr \mathfrak{u}_{g,1}$ of $\I_{g,1}$, and $\Gr \mathfrak{t}_{g}^{1}$ and $\Gr \mathfrak{u}_{g}^{1}$ of $\I_{g}^{1}$ are defined in a similar way.
For $g\ge 3$, we have central extensions 
\begin{equation}\label{centralextension}
    \Q[2]\to \Gr \mathfrak{t}_{g} \to \Gr \mathfrak{u}_{g}, \;\Q[2]\to \Gr \mathfrak{t}_{g,1} \to \Gr \mathfrak{u}_{g,1},\;\Q[2]\to \Gr \mathfrak{t}_{g}^{1} \to \Gr \mathfrak{u}_{g}^{1},
\end{equation}
\begin{equation}\label{extension}
    \Q[2]\to \Gr \mathfrak{t}_{g,1} \to \Gr \mathfrak{t}_{g}^{1},
\end{equation}
where $\Q[2]$ is the graded Lie algebra concentrated in degree $2$, and an extension
\begin{equation}\label{puu}
     \Gr \mathfrak{p}_{g}^{1}\to \Gr\mathfrak{u}_{g}^{1}\to \Gr\mathfrak{u}_{g},
\end{equation} 
where $\mathfrak{p}_{g}^{1}$ is the Malcev Lie algebra of the fundamental group of a closed surface of genus $g$, given by \cite{Hain} (see also \cite[Lemma 3.2]{KRW21}).
Hain \cite{Hain} proved that $\Gr \mathfrak{t}_{g},\Gr \mathfrak{u}_{g},\Gr \mathfrak{t}_{g,1},\Gr \mathfrak{u}_{g,1} ,\Gr \mathfrak{t}_{g}^{1}$ and $\Gr \mathfrak{u}_{g}^{1}$ are quadratically presented for $g\ge 6$.
Kupers--Randal-Williams \cite{KRW21} proved that these graded Lie algebras are Koszul in a stable range.

In what follows, we study the structures of the cohomology of Lie algebras $\Gr\mathfrak{t}_{g}, \Gr\mathfrak{u}_{g}, \Gr\mathfrak{t}_{g,1}, \Gr\mathfrak{u}_{g,1}, \Gr\mathfrak{t}_{g}^{1}$ and $\Gr \mathfrak{u}_{g}^{1}$ by using computation of characters by Garoufalidis--Getzler \cite{GG}.

Consider a graded algebraic $\Sp(2g,\Q)$-representation $Y_*=\bigoplus_{i\ge 1}Y_i$, $Y_i=\bigwedge^{i+2}H$.
Also consider quotients $Y'_* = Y_* / (\Q[2])$, $X_*=Y_*/ (H[1])$ and $X'_*=X_*/ (\Q[2])$, where $\Q[2]\subset Y_{2}=\bigwedge^{4} H$ is the trivial $\Sp(2g,\Q)$-subrepresentation, and where $H[1]\subset Y_1= \bigwedge^{3}H$.
We also consider an extension $Z_*=Y_*\oplus (\Q[2])$, and let $Z'_*=Z_*/(\Q[2])= Y_*$.
Then we obtain the following structures of the cohomology of six Lie algebras.

\begin{proposition}\label{Liealgebracohomology}
 We stably have isomorphisms of graded $\Sp(2g,\Q)$-representations
\begin{gather*}
\begin{split}
    H^* (\Gr \mathfrak{u}_{g},\Q)\cong \wti{S}^*(X_*),& \quad H^* (\Gr \mathfrak{t}_{g},\Q)\cong \wti{S}^*(X'_*),\\
    H^* (\Gr \mathfrak{u}_{g,1},\Q)\cong \wti{S}^*(Y_*),& \quad H^* (\Gr \mathfrak{t}_{g,1},\Q)\cong \wti{S}^*(Y'_*),\\
    H^* (\Gr \mathfrak{u}_{g}^{1},\Q)\cong \wti{S}^*(Z_*),& \quad H^* (\Gr \mathfrak{t}_{g}^{1},\Q)\cong \wti{S}^*(Z'_*),
\end{split}
\end{gather*}
where stably means that for each cohomological degree $i$, we have isomorphisms for $g\ge 3i$.
\end{proposition}

We will review the computation of Garoufalidis--Getzler \cite[Theorem 1.1]{GG}.
Let $A=\bigwedge^*(V^{\Sp}_{1^3})/(V^{\Sp}_{2^2})$ be the quadratic algebra, that is, the quotient of the exterior algebra $\bigwedge^*(V^{\Sp}_{1^3})$ by the ideal generated by $V^{\Sp}_{2^2}\subset \bigwedge^2(V^{\Sp}_{1^3})$.
Let $A^1=\bigwedge^*(V^{\Sp}_{1^3}\oplus V^{\Sp}_{1})/(V^{\Sp}_{2^2}\oplus V^{\Sp}_{1^2})$ be the quadratic algebra defined similarly.
Then $A$ (resp. $A^1$) is the quadratic dual of the universal enveloping algebra of $\mathfrak{u}_{g}$ (resp. $\mathfrak{u}_{g}^{1}$).

\newcommand\ch{\mathrm{ch}_t}
\newcommand\Exp{\mathrm{Exp}}
Let $\Lambda$ denote the ring of symmetric functions.
Let $h_n$ denote the complete symmetric function, $e_n$ the elementary symmetric function, $p_n$ the power sum symmetric function, $s_{\lambda}$ the Schur function, $s_{\langle \lambda \rangle}$ the symplectic Schur function.
Let $\omega: \Lambda\to \Lambda,\; p_n \mapsto -p_n$ denote an algebra involution, which satisfies $\omega(h_n)=(-1)^n e_n$.
We have a linear automorphism of $\Lambda$
\begin{gather*}
 \begin{split}
D:\Lambda\to \Lambda,\quad s_{\lambda}\mapsto s_{\langle \lambda \rangle},
 \end{split}
\end{gather*}
and an operation
$$\Exp:\Lambda[[t]]\to \hat{\Lambda}[[t]],\quad f\mapsto \sum_{q=0}^{\infty}h_q\circ f,$$
where $h_q\circ f$ denotes the plethysm and $\hat{\Lambda}[[t]]$ is the completion of $\Lambda[[t]]$ with respect to the augmentation ideal.
For an algebraic $\Sp(2g,\Q)$-representation $V=\bigoplus_{\lambda} (V^{\Sp}_{\lambda})^{\oplus c(\lambda)}$, the \emph{character} $\mathrm{ch}(V)$ of $V$ is defined by $\mathrm{ch}(V)=\sum_{\lambda} c(\lambda)s_{\langle \lambda \rangle}\in \Lambda$.
The \emph{character} $\ch(M_*)$ of a graded algebraic $\Sp(2g,\Q)$-representation $M_*$ is defined by $\ch(M_*)=\sum_{q=0}^{\infty} (-t)^q \mathrm{ch}(M_q)\in \Lambda[[t]]$.

Let 
$$L(t)=\sum_{q=1}^{\infty} \left(\sum_{r=0}^{\lfloor q+2/2\rfloor}h_{q+2-2r}\right)t^q\in \Lambda[[t]].$$

\begin{theorem}[Garoufalidis--Getzler {\cite[Theorem 1.1]{GG}}]
 The characters $\ch(A)$ of $A$ and $\ch(A^1)$ of $A^1$ in $\Lambda[[t]]$ are
 \begin{gather*}
    \begin{split}
        \ch(A)&= D\omega \Exp \left(-h_1 t+ L(t)
        \right),\quad 
        \ch(A^1)= D\omega \Exp\left(h_0 t^2 +
        L(t)
        \right).
    \end{split}
 \end{gather*}
\end{theorem}

\begin{remark}
The above formula for $\ch(A^1)$ has one more copies of $h_0 t^2$ in the $\Exp(-)$ than the original one given in \cite[Theorem 1.1]{GG}.
By their proof, we have a stable isomorphism of algebras
$$
 \C\langle \text{T-graphs} \rangle/ (\mathrm{IH}^1)\to (A^1\otimes T(H))^{\Sp},
$$
where the left-hand side is the quotient of the algebra of trivalent graphs with ordered legs modulo the $\mathrm{IH}^1$-relation, and where $T(H)=\bigoplus_{q=0}^{\infty}H^{\otimes q}$ is the tensor algebra.
Here, the $\mathrm{IH}^1$-relation is as follows: if a graph $G$ has a part of shape $I$ at most two of whose four endpoints are connected in $G$, then $G$ is identified with another graph which is obtained from $G$ by replacing the part of shape $I$ with a graph of shape $H$.
Therefore, in degree $2$, the theta-shaped graph, which they wrote as $e_{2,0}$, and the dumbbell-shaped graph, which has two loops connected to one edge, are not equivalent.
Hence, we need one more $h_0 t^2$ in the $\Exp(-)$ of $\ch(A^1)$.
\end{remark}

\begin{remark}\label{chA^1andchA}
Alternatively, we can compute the character $\ch(A^1)$ by using the character $\ch(A)$ and the Hochschild--Serre spectral sequence for the extension \eqref{puu} for $g\ge 3$.
Since the action of $\Gr \mathfrak{u}_{g}$ on $H^*(\Gr\mathfrak{p}_{g}^{1},\Q)=\Q[0]\oplus H[1]\oplus \Q[2]$ is trivial,
we have an isomorphism of graded $\Sp(2g,\Q)$-representations
$$
  H^*(\Gr \mathfrak{u}_{g}^{1},\Q)\cong H^*(\Gr \mathfrak{u}_{g},\Q)\otimes H^*(\Gr\mathfrak{p}_{g}^{1},\Q).
$$
Therefore, we have
$$
\ch(A^1)= \ch(A) (1-h_1t +t^2).
$$
\end{remark}

\begin{proof}[Proof of Proposition \ref{Liealgebracohomology}]
We can check that for any graded $\Sp(2g,\Q)$-representation $M_*$, we have
 $$
   \ch(\wti{S}^*(M_*))= D \Exp(D^{-1} \ch(M_*)).
 $$
The characters of $X_*, Y_*$ and $Z_*$ are
\begin{gather*}
 \begin{split}
    \ch(Y_*)&=D\sum_{q=1}^{\infty} \left(\sum_{r=0}^{\lfloor q+2/2\rfloor}e_{q+2-2r}\right)(-t)^q=D \omega \left(L(t)\right),\\
    \ch(X_*)&=D \omega \left(-h_1 t + L(t)\right),\\
    \ch(Z_*)&=D \omega \left(h_0 t^2 + L(t)\right).
 \end{split}
\end{gather*}
It follows that
\begin{gather*}
\begin{split}
    \ch(\wti{S}^*(X_*))
    &=D \Exp(D^{-1} \ch(X_*))\\
    &=D \Exp\left(\omega \left(-h_1 t + L(t)\right)\right)\\
    &=D \omega \Exp\left(-h_1 t + L(t)\right)\\
    &=\ch(A).
\end{split}
\end{gather*}
We have an isomorphism of graded $\Sp(2g,\Q)$-representations
$$H^* (\Gr \mathfrak{u}_{g},\Q)\cong A$$
for $g\ge 3*$, since $\Gr \mathfrak{u}_{g}$ is quadratically presented for $g\ge 6$, and Koszul for $g\ge 3*$.
Therefore, we have $H^* (\Gr \mathfrak{u}_{g},\Q)\cong \wti{S}^*(X_*)$. The case of $\Gr \mathfrak{u}_{g}^{1}$ is similar.

For the other four graded Lie algebras, we use the Hochschild--Serre spectral sequences for extensions \eqref{centralextension} and \eqref{extension}.
In what follows, we only consider the case of $\Gr \mathfrak{t}_{g}$.
We can check the other cases in a similar way.
For the extension $\Gr \mathfrak{t}_{g}$ of $\Gr \mathfrak{u}_{g}$, since the action of $\Gr \mathfrak{u}_{g}$ on $H^*(\Q[2],\Q)=H^0(\Q[2],\Q)\oplus H^1(\Q[2],\Q)=\Q[0]\oplus \Q[2]$ is trivial, we have
$$
H^*(\Gr \mathfrak{t}_{g},\Q)\cong H^*(\Gr \mathfrak{u}_{g},\Q)/(H^*(\Gr \mathfrak{u}_{g},\Q)[2]),
$$
where $H^*(\Gr \mathfrak{u}_{g},\Q)[2]$ denotes the degree-shift by $2$.
Therefore, we have
\begin{gather*}
\begin{split}
     \ch(H^*(\Gr \mathfrak{t}_{g},\Q))
    &=\ch(H^*(\Gr \mathfrak{u}_{g},\Q)) (1-t^2)\\
    &=\left( D \omega \Exp\left(-h_1 t+
        L(t)
        \right)\right) \Exp(-t^2)\\
    &= D\omega  \Exp\left(-h_1 t -h_0 t^2 +
        L(t)
        \right)\\
    &=\ch(\wti{S}^*(X'_*)).
\end{split}
\end{gather*}
Therefore, we have $H^*(\Gr \mathfrak{t}_{g},\Q)\cong \wti{S}^*(X'_*)$, which completes the proof.
\end{proof}

\subsection{Conjectural structures of the Albanese cohomology of the Torelli groups}\label{conjTorellialb}

Here we study the structures of the Albanese cohomology of $\I_{g}$, $\I_{g,1}$ and $\I_{g}^{1}$.

\begin{proposition}\label{Torellisurj}
We stably have surjective morphisms of graded $\Sp(2g,\Q)$-representations
\begin{gather*}
    H^*(\Gr \mathfrak{t}_{g},\Q)\twoheadrightarrow H_A^*(\I_{g},\Q),\quad  H^*(\Gr \mathfrak{t}_{g,1},\Q)\twoheadrightarrow H_A^*(\I_{g,1},\Q),\\
    H^*(\Gr \mathfrak{t}_{g}^{1},\Q)\twoheadrightarrow H_A^*(\I_{g}^{1},\Q).
\end{gather*}
\end{proposition}

\begin{proof}
Let $A'$ denote the quadratic algebra $A'=\bigwedge^*(V^{\Sp}_{1^3})/(V^{\Sp}_{2^2}\oplus V^{\Sp}_{0})$.
Since $\Gr \mathfrak{t}_{g}$ is quadratically presented for $g\ge 6$ \cite{Hain} and stably Koszul \cite{KRW21}, in a way similar to the argument of \cite{GG}, we have an isomorphism of graded $\Sp(2g,\Q)$-representations
$$
  H^*(\Gr \mathfrak{t}_{g},\Q)\cong A' \cong H^*(\I_{g}/\Gamma_{3}\I_{g},\Q),
$$
where $\Gamma_{3}\I_{g}$ is the third term of the lower central series of $\I_{g}$.
The natural projection $\I_{g}\twoheadrightarrow \I_{g}/\Gamma_{3}\I_{g}$ induces a surjective morphism of graded $\Sp(2g,\Q)$-representations
$$
  H^*(\I_{g}/\Gamma_{3}\I_{g},\Q)\twoheadrightarrow H_A^*(\I_{g},\Q).
$$
Therefore, we have a surjective morphism of graded $\Sp(2g,\Q)$-representations
$$
H^*(\Gr \mathfrak{t}_{g},\Q)\twoheadrightarrow H_A^*(\I_{g},\Q).
$$
Similar arguments hold for $\I_{g,1}$ and $\I_{g}^{1}$.
\end{proof}

By Proposition \ref{Torellisurj}, the Albanese cohomology of $\I_{g}, \I_{g,1}$ and $\I_{g}^{1}$ are quotient $\Sp(2g,\Q)$-representations of the cohomology of $\Gr \mathfrak{t}_{g}$, $\Gr \mathfrak{t}_{g,1}$ and $\Gr \mathfrak{t}_{g}^{1}$, respectively.
In what follows, we propose conjectural structures of the Albanese cohomology of $\I_{g}, \I_{g,1}$ and $\I_{g}^{1}$.
Consider the quotients $X''_*= X_* / (\bigoplus_{i\ge 1}\Q[4i-2])$, $Y''_*= Y_* / (\bigoplus_{i\ge 1}\Q[4i-2])$ and $Z''_*=Z_* / (\bigoplus_{i\ge 1}\Q[4i-2])$, where $\Q[4i-2]\subset\bigwedge^{4i} H$ is the trivial $\Sp(2g,\Q)$-subrepresentation.

\begin{conjecture}\label{Torellialb}
We stably have graded $\Sp(2g,\Q)$-isomorphisms
$$
H_A^*(\I_{g},\Q)\cong  \wti{S}^*(X''_*), \quad 
H_A^*(\I_{g,1},\Q)\cong \wti{S}^*(Y''_*),\quad 
H_A^*(\I_{g}^{1},\Q)\cong \wti{S}^*(Z''_*).
$$
\end{conjecture}

We stably have $H_A^*(\I_{g},\Q)\cong  \wti{S}^*(X''_*)$ in degree $\le 3$ by \cite{Hain, Sakasai, KRW}.
By the definition of $\wti{S}^*(Y''_*)$, we can check that the direct sum $\overline{W}^{\I_{g,1}}_i$, which we observed in Section \ref{secLindell}, is included in $\wti{S}^*(Y''_*)_i.$

Here, we will recall the result of Kupers--Randal-Williams \cite{KRW} and study the relation with Conjecture \ref{Torellialb}.
In \cite[Theorem 4.1 and Sections 6.2 and 8]{KRW}, they constructed an algebra homomorphism $\Phi:W\to H^*(\I_{g,1},\Q)$, where $W$ is an algebra whose character $\ch(W)$ is 
$$
\ch(W)=D \omega \left(
\left(\prod_{4i>2} (1-t^{4i-2})\right)
\Exp \left(
\frac{\frac{1}{1-t^2}(\sum_{q=0}^{\infty}h_qt^q)- h_0(1+t^2)-h_1t-h_2t^2}{t^2}
\right)
\right).
$$
We can easily check that 
\begin{equation}\label{KRWandtraceless}
    \ch(W)=\ch(\wti{S}^*(Y''_*))
\end{equation}
since we have
$$
\ch(Y''_*)=D \omega \left(L(t)- \sum_{q=1}^{\infty}t^{4q-2}\right)
$$
and 
$$
\ch(W)=D\omega \Exp\left(L(t)- \sum_{q=1}^{\infty}t^{4q-2}\right).
$$
In \cite{KRW}, they proved that for $N\ge 0$, if $H^*(\I_{g,1},\Q)$ is stably finite-dimensional for $*<N$, then $\Phi$ is an isomorphism for $*\le N$ and is injective for $*=N+1$.
Therefore, $H^i(\I_{g,1},\Q)$ contains an $\Sp(2g,\Z)$-subrepresentation which is isomorphic to $\wti{S}^*(Y''_*)_i$ for $i\le 3$.

In \cite[Sections 7 and 8.1]{KRW}, they also considered the cases for $\I_{g}$ and $\I_{g}^{1}$.
We can also check the variants of \eqref{KRWandtraceless} for these two cases by using Remark \ref{chA^1andchA}.

\begin{remark}
Note that in \cite{KRW}, they used a different definition of $\ch(W)=\sum_{q\ge 0}t^q \mathrm{ch}(W_q)$ and a different involution $\omega: p_n\mapsto (-1)^n p_n$ from ours.
\end{remark}

\begin{remark}\label{remarkKRW}
In \cite[Remark 8.2]{KRW}, they claimed that $H_A^3(\I_{g},\Q)$ has one fewer copies of $V^{\Sp}_{2^3 1^3}$ than the \emph{algebraic part} ${H^3(\I_{g},\Q)}^{\text{alg}}$ of $H^3(\I_{g},\Q)$, which is the union of algebraic subrepresentations of $H^3(\I_{g},\Q)$.
However, their computation of the characters of the third cohomology of $\I_{g}^{1}$ and $\I_{g}$ seems incorrect, and indeed it is possible that we have $H^3(\I_g,\Q)^{\text{alg}}=H_A^3(\I_g,\Q)$ for sufficiently large $g$.
\end{remark}

Kawazumi--Morita \cite{Kawazumi-Morita} studied the $\Sp$-invariant stable continuous cohomology $H_{c}^*(\lim_{g\to \infty} \I_{g,1},\Q)^{\Sp}$.
Their conjecture \cite[Conjecture 13.8]{Kawazumi-Morita}, which first appeared in \cite[Conjecture 3.4]{Morita1999}, is the following.

\begin{conjecture}[Kawazumi--Morita]\label{conjKM}
We stably have isomorphisms of graded algebras
$$H^*(\I_{g},\Q)^{\Sp(2g,\Z)}\cong \Q[e_2,e_4,\cdots],$$
where $e_{2i}$ is the Mumford--Morita--Miller class of degree $4i$. 
\end{conjecture}

Furthermore, we make the following conjecture, which is closely related to, but different from, the above conjecture.

\begin{conjecture}\label{Torellialbsp}
We stably have isomorphisms of graded algebras
$$H_A^*(\I_{g},\Q)^{\Sp(2g,\Z)}\cong H_A^*(\I_{g,1},\Q)^{\Sp(2g,\Z)} \cong \Q[y_1,y_2,\cdots],$$
$$H_A^*(\I_{g}^{1},\Q)^{\Sp(2g,\Z)}\cong \Q[z,y_1,y_2,\cdots],$$
where $\deg y_i=4i$ and $\deg z=2$.
\end{conjecture}

\begin{proposition}
 If Conjecture \ref{Torellialb} holds, then Conjecture \ref{Torellialbsp} holds.
\end{proposition}

\begin{proof}
We have a coalgebra structure on $\wti{S}^*(X''_*)$ induced by the coalgebra structure of $S^*(X''_*)$ defined in Section \ref{Preliminary}.
Since algebraic $\Sp(2g,\Q)$-representations are self-dual, we have an algebra structure on $\wti{S}^*(X''_*)$.
Therefore, we obtain an algebra structure on $\wti{S}^*(X''_*)^{\Sp(2g,\Z)}$.

For a graded algebraic $\Sp(2g,\Q)$-representation $M_*$,
we have 
$$\wti{S}^*(M_*)^{\Sp(2g,\Z)}\cong S^*((M_*)^{\Sp(2g,\Z)}).$$
Since we stably have 
$$(X''_*)^{\Sp(2g,\Z)}=(Y''_*)^{\Sp(2g,\Z)}=\bigoplus_{j\ge 1}\Q[4j],\quad (Z''_*)^{\Sp(2g,\Z)}=\Q[2]\oplus(\bigoplus_{j\ge 1}\Q[4j]),$$
we stably have isomorphisms of graded algebras
$$
\wti{S}^*(X''_*)^{\Sp(2g,\Z)}\cong \wti{S}^*(Y''_*)^{\Sp(2g,\Z)}\cong \Q[y_1,y_2,\cdots],\quad \wti{S}^*(Z''_*)^{\Sp(2g,\Z)}\cong \Q[z,y_1,y_2,\cdots],
$$ 
where $\deg y_i=4i$, and $\deg z=2$.
Therefore, if Conjecture \ref{Torellialb} holds, then Conjecture \ref{Torellialbsp} holds.
\end{proof}

\subsection{Relation between $H^A_*(\I_{g},\Q)$, $H^A_*(\I_{g,1},\Q)$ and $H^A_*(\I_{g}^{1},\Q)$}

Let $g\ge 2$.
Let $\Sigma_g$ (resp. $\Sigma_{g,1}$, $\Sigma_{g}^{1}$) be a closed surface of genus $g$ (resp. with one boundary component, with one marked point).
Let $\calM_{g,1}$ (resp. $\calM_{g}^{1}$) denote the mapping class group of $\Sigma_{g,1}$ (resp. $\Sigma_{g}^{1}$).

We have exact sequences of groups
\begin{equation}\label{exactmarkedpoint}
1\to \pi_1(\Sigma_g)\to \I_{g}^{1}\to \I_{g}\to 1,\quad
1\to \pi_1(U\Sigma_g)\to \I_{g,1}\to \I_{g}\to 1,
\end{equation}
where $U\Sigma_g$ denotes the unit tangent bundle of $\Sigma_g$.
Since we have $H_*(\pi_1(\Sigma_g),\Q)\cong H_*(\Sigma_g,\Q)$, we have
$$H^A_*(\pi_1(\Sigma_g),\Q)\cong \Q[0]\oplus H[1]\oplus \Q[2].$$
By using the Hochschild--Serre spectral sequence for the exact sequence 
$$1\to \Z\to \pi_1(U\Sigma_g)\to \pi_1(\Sigma_g)\to 1,$$
we have
\begin{gather*}
 H_*(\pi_1(U\Sigma_g),\Q)\cong\Q[0]\oplus H[1]\oplus H[2]\oplus \Q[3].
\end{gather*}
Therefore, we have
$$
H^A_*(\pi_1(U\Sigma_g),\Q)\cong \Q[0]\oplus H[1].
$$
By using Propositions \ref{spectralsequence} and \ref{cospectralsequence3}, we obtain the following proposition.

\begin{proposition}\label{torelliinj}
(1) We have a graded $\Sp(2g,\Q)$-isomorphism
$$
  H^A_*(\I_{g}^{1},\Q) \xrightarrow{\cong}  H^A_*(\I_{g},\Q)\otimes H^A_*(\pi_1(\Sigma_g),\Q).
$$

(2) We have an injective graded $\Q$-linear map
$$
  H^A_*(\I_{g,1},\Q) \hookrightarrow H^A_*(\I_{g},\Q) \otimes H^A_*(\pi_1(U\Sigma_g),\Q).
$$
\end{proposition}

\begin{proof}
    The exact sequences \eqref{exactmarkedpoint} induce exact sequences
    \begin{gather*}
        0\to (\pi_1(\Sigma_g))^{\ab}\to (\I_{g}^{1})^{\ab}\to (\I_{g})^{\ab}\to 0,\\ 0\to (\pi_1(U\Sigma_g))^{\ab}\to (\I_{g,1})^{\ab}\to (\I_{g})^{\ab}\to 0,
    \end{gather*}
    which split as $\Z[\Sp(2g,\Z)]$-modules.
    Therefore, by Proposition \ref{spectralsequence}, we obtain (2).
    It is easy to see that $\I_{g}$ acts trivially on $H^*(\pi_1(\Sigma_g),\Q)$, $E_2^{p,q}=H^p(\I_g, H^q(\pi_1(\Sigma_g),\Q))\cong H^p(\I_g,\Q)\ot H^q(\pi_1(\Sigma_g),\Q)=0$ for $q\ge 3$, all $d_2^{0,1}, d_2^{0,2}, d_3^{0,2}$ vanish and that $H^*((\I_{g}^{1})^{\ab},\Q)$ is a semisimple $\Q[\Sp(2g,\Z)]$-module. Therefore, by Proposition \ref{cospectralsequence3}, we obtain (1).
\end{proof}

\begin{remark}
The composition 
$$s:\Sigma_{g,1}\xrightarrow{\text{section}} U\Sigma_{g,1}\hookrightarrow U\Sigma_g$$ induces an isomorphism of $\Sp(2g,\Q)$-representations
$$s_*=H^A_i(\pi_1(s),\Q): H^A_i(\pi_1(\Sigma_{g,1}),\Q)\to H^A_i(\pi_1(U\Sigma_g),\Q).$$
\end{remark}

\begin{remark}
The Albanese cohomology of the Torelli groups stably coincide with the algebraic parts of the cohomology of them in degree $\le 2$.
The case of $\I_{g}$ follows from \cite{Hain} and \cite{KRW}.
The case of $\I_{g}^{1}$ follows from \cite{KRW} and Proposition \ref{torelliinj}.
The case of $\I_{g,1}$ follows from \cite{KRW} and \cite{Lindell} since we have $\overline{W}^{\I_{g,1}}_2=\wti{S}^*(Y''_*)_2$.
As we mentioned in Remark \ref{remarkKRW}, we stably have $H_A^3(\I_{g},\Q)= H^3(\I_{g},\Q)^{\text{alg}}$ if $H^2(\I_{g},\Q)$ is stably finite-dimensional \cite{Sakasai, KRW}.
The same holds for $\I_{g}^{1}$ by Proposition \ref{torelliinj} and \cite{KRW}. 
\end{remark}

\subsection{Albanese homology of $\IA_{2g}$ and $\I_{g,1}$} 

We have an injective group homomorphism $\iota: \I_{g,1}\hookrightarrow \IA_{2g}$, which induces an injective homomorphism
$$\iota_*:H_i(H_1(\I_{g,1}),\Q)\to H_i(H_1(\IA_{2g}),\Q).$$
Let $\iota_*^A$ denote the restriction 
$$
  \iota_*^A: H^A_i(\I_{g,1},\Q)\to H^A_i(\IA_{2g},\Q)
$$
of $\iota_*$ to $H^A_i(\I_{g,1},\Q)$ and $H^A_i(\IA_{2g},\Q)$, which is also injective.

\begin{remark}
Recall Conjectures \ref{conjectureAlbanese} and \ref{Torellialb} for the structures of $H^A_*(\IA_{2g},\Q)$ and $H^A_*(\I_{g,1},\Q)$.
By Lemma \ref{GLSPtl}, we have 
\begin{equation}\label{UY}
    \wti{S}^*(Y''_*)\subset \wti{S}^*(U_*)=W_*
\end{equation}
since we have for $i\ge 1$
$$U_i=\Hom(H, \bigwedge^{i+1}H)\cong H^*\otimes \bigwedge^{i+1}H \cong H\otimes \bigwedge^{i+1}H\supset \bigwedge^{i+2}H\supset Y''_i$$ as $\Sp(2g,\Q)$-representations.
The inclusion \eqref{UY} conjecturally describes the injective $\Sp(2g,\Q)$-homomorphism 
$\iota_*^A: H^A_i(\I_{g,1},\Q)\to H^A_i(\IA_{2g},\Q)$, where we consider $ H^A_i(\IA_{2g},\Q)$ as an $\Sp(2g,\Q)$-representation.
\end{remark}

We make the following conjecture about the image of $\iota_*^A$.

\begin{conjecture}\label{torelligroupandIA}
For $g\ge 3i$, 
the image $\iota_* (H^A_i(\I_{g,1},\Q))$ generates $H^A_i(\IA_{2g},\Q)$ as $\GL(2g,\Q)$-representations.
\end{conjecture}

We can easily check that Conjecture \ref{torelligroupandIA} holds for $i=1$.
Conjecture \ref{torelligroupandIA} also holds for $i=2$, which can be verified by using abelian cycles $(\rho_1)^2\in H^A_2(\I_{4,1},\Q)$ and $\rho_2\in H^A_2(\I_{3,1},\Q)$ given in \cite{Lindell} and the irreducible decomposition of $H^A_2(\IA_{2g},\Q)$ as $\GL(2g,\Q)$-representations \cite{Pettet}.

\appendix
\section{Properties of Albanese homology and cohomology}\label{Albanesefunctor}
Here we give a brief summary of some properties about Albanese homology and cohomology.

\subsection{Albanese homology functor}
Let $\Gp$ denote the category of groups and group homomorphisms and $\grVect$ the category of graded $\Q$-vector spaces and graded linear maps.
Let 
$$H^A_*: \Gp\to \grVect$$
denote the functor which maps a group $G$ to $H^A_*(G,\Q)$, and 
$$H_*: \Gp\to \grVect$$
the functor which maps a group $G$ to $H_*(G,\Q)$.
Then we have a natural transformation 
$$(\pi_*)_{G}:=\pi^G_*: H_*(G,\Q)\twoheadrightarrow H^A_*(G,\Q),$$
where $\pi^G_*$ is the map induced by the abelianization $\pi^G: G\to G^{\ab}$.

Let $\CCoalg$ denote the category of graded-cocommutative coalgebras and graded coalgebra morphisms over $\Q$.
Then the functors $H^A_*$ and $H_*$ lift to functors
$$H^A_*: \Gp\to \CCoalg,\quad H_*: \Gp\to \CCoalg$$ 
and the natural transformation $\pi_*$ lifts to 
a natural transformation 
$$\pi_*: H_*\Rightarrow H^A_*.$$

\subsection{Filtered colimits}

\newcommand\colim{\operatorname{colim}}

A \emph{filtered category} $I$ is a category satisfying the following two conditions:
\begin{itemize}
    \item for any objects $i,j\in I$, there are an object $k$ and morphisms $i\to k$ and $j\to k$,
    \item for any parallel morphisms $f:i\to j$ and $g:i\to j$, there exists a morphism $w: j\to k$ such that $wf=wg$.
\end{itemize}
For a functor $F$ from a filtered category $I$ to another category $\calC$, the colimit $\colim_{i\in I} F_i$ of $F$ is called a \emph{filtered colimit}.

Group homology preserves filtered colimits. We observe that the same property holds for Albanese homology.

\begin{proposition}\label{preservedirectlimit}
 The functor $H^A_*$ preserves filtered colimits.
 That is, for a filtered category $I$, the natural map
 $$\colim_{i\in I} H^A_*(G_i,\Q)\to H^A_*(\colim_{i\in I} G_i,\Q)$$
 is an isomorphism.
 In particular, $H^A_*$ preserves direct limits.
\end{proposition}
 
\begin{proof}
Let $A, B: I\to \grVect$ be two functors. 
For a natural transformation $\alpha_i:A_i\rightarrow B_i$, we have
$$
  \colim_{i\in I}(\im (A_i\xrightarrow{\alpha_i} B_i))\xrightarrow{\cong} \im (\colim_{i\in I}A_i\xrightarrow{\colim_{i\in I}\alpha_i} \colim_{i\in I}B_i).
$$
Therefore, we have
\begin{gather*}
 \begin{split}
  \colim_{i\in I} H^A_*(G_i,\Q)
  &= \colim_{i\in I} (\im (H_*(G_i,\Q)\to H_*(G_i^{\ab},\Q)))\\
  &\cong 
  \im (\colim_{i\in I} (H_*(G_i,\Q))\to \colim_{i\in I}H_*(G_i^{\ab},\Q))\\
  &\cong\im (H_*(\colim_{i\in I} G_i,\Q)\to H_*(\colim_{i\in I}(G_i^{\ab}),\Q))\\
  &\cong\im (H_*(\colim_{i\in I} G_i,\Q)\to H_*((\colim_{i\in I}G_i)^{\ab},\Q))\\
  &=H^A_*(\colim_{i\in I} G_i,\Q).
 \end{split}
\end{gather*}
\end{proof}

\subsection{Duality}\label{duality}
 We have the universal coefficient theorem for group homology and group cohomology.
 Here we observe that the universal coefficient theorem also holds for Albanese homology and cohomology.
 
 \begin{lemma}\label{dualityisom}
  For a group $G$, we have a linear isomorphism
  $$
   H_A^i(G,\Q)\xrightarrow{\cong} (H^A_i(G,\Q))^*.
  $$
 \end{lemma}
 
 \begin{proof}
 This lemma follows from the following fact.
 Let $V,W$ be $\Q$-vector spaces and $f: V\to W$ a linear map.
 Define a linear map
 $$
 \Phi: \im (f^*)\to (\im f)^* 
 $$
 by sending $\psi\in \im (f^*)$ to the restriction $\phi|_{\im f}$, where $\phi\in W^*$ is a linear map satisfying $\psi=\phi f$.
 Then we can check that $\Phi$ is a linear isomorphism.
 \end{proof}

 We have $\GL(n,\Q)$-representation structures on $H^A_i(\IA_n,\Q)$ and $H_A^i(\IA_n,\Q)$.
 Then we can check that the map $\Phi$ in the proof of Lemma \ref{dualityisom} is a $\GL(n,\Q)$-isomorphism.
 We obtain the following duality as $\GL(n,\Q)$-representations.
 
 \begin{proposition}\label{dualityisomIA}
  We have a $\GL(n,\Q)$-isomorphism
  $$
   H_A^i(\IA_n,\Q)\xrightarrow{\cong} (H^A_i(\IA_n,\Q))^*.
  $$
 \end{proposition}

\subsection{Albanese cohomology functor}
Let $$H_A^*: \Gp^{\op}\to \grVect$$
denote the functor which maps a group $G$ to $H_A^*(G,\Q)$, and 
$$H^*: \Gp^{\op}\to \grVect$$
the functor which maps a group $G$ to $H^*(G,\Q)$.
Then we have a natural transformation 
$$\iota_{G}: H_A^*(G,\Q) \hookrightarrow H^*(G,\Q).$$

For any group $G$, we have a graded-commutative algebra structure on $H^*(G,\Q)$ with the cup product as the multiplication.
Let $\CAlg$ denote the category of graded-commutative $\Q$-algebras and graded algebra morphisms.
Then the functors $H_A^*$ and $H^*$ lift to functors
$$H_A^*: \Gp^{\op}\to \CAlg, \quad H^*: \Gp^{\op}\to \CAlg$$
and the natural transformation $\iota$ lifts to a natural transformation
$$\iota: H_A^*\Rightarrow H^*.$$

The following property holds for Albanese cohomology as in the case of group cohomology.

\begin{proposition}
 The functor $H_A^*$ preserves filtered limits. That is, for a filtered category $I$, the natural map
 $$
 H_A^*(\colim_{i\in I} G_i,\Q)\rightarrow \lim_{i\in I} H_A^*(G_i,\Q)
 $$
 is an isomorphism.
\end{proposition}

\begin{proof}
 By Proposition \ref{preservedirectlimit} and Lemma \ref{dualityisom}, it follows that
 \begin{gather*}
  \begin{split}
       H_A^*(\colim_{i\in I} G_i,\Q)&\cong (H^A_*(\colim_{i\in I} G_i,\Q))^*\cong (\colim_{i\in I} H^A_*(G_i,\Q))^*\\
     &\cong \lim_{i\in I} (H^A_*(G_i,\Q))^* \cong 
 \lim_{i\in I} H_A^*(G_i,\Q).
  \end{split}
 \end{gather*}
\end{proof}

\subsection{Hochschild--Serre spectral sequence}\label{spseq}
Here we study Albanese homology of groups by using the Hochschild--Serre spectral sequences for exact sequences of groups. See \cite[Chapter 5]{Weibel} for details of spectral sequences and their convergence.

For an exact sequence of groups
\begin{gather}\label{exactseq}
1\to N\xto{\iota} G\xto{\pr} Q\to 1,
\end{gather}
we have the following commutative diagram whose rows are exact 
 \begin{gather*}
   \xymatrix{
   1 \ar[r]  & N \ar[r]^{\iota} \ar[d]^{\pi^N} & G\ar[r]^{\pr} \ar[d]^{\pi^G} & Q \ar[r]\ar[d]^{\pi^Q} & 1  \\
          & N^{\ab} \ar[r]_{\iota_*}  & G^{\ab} \ar[r]_{\pr_*} & Q^{\ab} \ar[r] &1.
   }
 \end{gather*}
We have the Hochschild--Serre spectral sequence which converges to $H_{p+q}(G,\Q)$:
$$
 E^2_{p,q}=H_p(Q,H_q(N,\Q)) \Rightarrow H_{p+q}(G,\Q).
$$

Suppose that $\iota_*$ is injective.
Then we also have the Hochschild--Serre spectral sequence which converges to $H_{p+q}(G^{\ab},\Q)$:
$$
 \wti{E}^2_{p,q}=H_p(Q^{\ab},H_q(N^{\ab},\Q)) \Rightarrow H_{p+q}(G^{\ab},\Q).
$$
Since $Q^{\ab}$ acts trivially on $H_q(N^{\ab},\Q)$, we have
$$H_p(Q^{\ab},H_q(N^{\ab},\Q))\cong H_p(Q^{\ab},\Q)\otimes H_q(N^{\ab},\Q),$$
and thus we have $\wti{E}^2_{p,q}=\wti{E}^{\infty}_{p,q}$.

A morphism between short exact sequences of groups induces a morphism of spectral sequences.
Thus, the abelianization induces a morphism of spectral sequences
$$f^*_{p,q}: E^{*}_{p,q} \to \wti{E}^{*}_{p,q}.$$
By using the above two spectral sequences, we obtain the following proposition.

\begin{proposition}\label{spectralsequence}
 Suppose that $\iota_*:N^{\ab}\to G^{\ab}$ is injective.
 Then we have an injective graded $\Q$-linear map
 $$
  H^A_*(G,\Q)\hookrightarrow H^A_*(Q,\Q)\otimes H^A_*(N,\Q).
 $$
\end{proposition}

\begin{proof}
We can check that the image of
$$
  f^2_{p,q}: H_p(Q,H_q(N,\Q)) \to H_p(Q^{\ab}, H_q(N^{\ab},\Q))= H_p(Q^{\ab},\Q)\otimes H_q(N^{\ab},\Q)
$$
is $H^A_p(Q,\Q)\otimes H^A_q(N,\Q).$
Since $\wti{E}^{\infty}_{p,q}=\wti{E}^2_{p,q}$, we have 
$$
 \im f^{\infty}_{p,q}\subset \im f^2_{p,q}=H^A_p(Q,\Q)\otimes H^A_q(N,\Q).
$$
Since the map $f^{\infty}_{p,q}$ is compatible with
$$
  \pi_*: H_{p+q}(G,\Q)\twoheadrightarrow H^A_{p+q}(G,\Q)\hookrightarrow H_{p+q}(G^{\ab},\Q),
$$
we have a filtration of $\Q$-vector spaces
$$
0=F_{-1}\subset F_0\subset \cdots\subset F_{n-1}\subset F_{n}=H^A_{n}(G,\Q)
$$
satisfying 
$F_r/F_{r-1}=\im f^{\infty}_{r,n-r}$ for $0\le r\le n$.
Therefore, we have
$$
H^A_{n}(G,\Q)\cong \bigoplus_{p+q=n}\im f^{\infty}_{p,q}\subset \bigoplus_{p+q=n}H^A_p(Q,\Q)\otimes H^A_q(N,\Q).
$$
\end{proof}

Suppose that a group $K$ acts on $H_*(N,\Z), H_*(G,\Z), H_*(Q,\Z)$ in such a way that the induced morphisms $\iota_*: H_*(N,\Z)\to H_*(G,\Z)$ and $\pr_*: H_*(G,\Z)\to H_*(Q,\Z)$ are $\Z[K]$-module maps.
Suppose that the exact sequence \eqref{exactseq} induces a split exact sequence of $\Z[K]$-modules
\begin{gather}\label{exactseqsplit}
    1\to N^{\ab}\xto{\iota_*} G^{\ab}\xto{\pr_*} Q^{\ab}\to 1.
\end{gather}
If we consider the cohomological Hochschild--Serre spectral sequences, then we have 
$$
 E_2^{p,q}=H^p(Q,H^q(N,\Q)) \Rightarrow H^{p+q}(G,\Q)
$$
and
$$
\wti{E}_2^{p,q}=H^p(Q^{\ab},H^q(N^{\ab},\Q)) \Rightarrow H^{p+q}(G^{\ab},\Q).
$$

\begin{proposition}\label{cospectralsequence3}
In the above setting,
suppose 
\begin{itemize}
    \item  $Q$ acts trivially on $H^*(N,\Q)$,
    \item  $E_2^{p,q}=0$ for any $q\ge 3$,
    \item  the differentials $d_2^{0,1}$, $d_2^{0,2}$, $d_3^{0,2}$ are zero maps,
    \item $H^*(G^{\ab},\Q)$ is a semisimple $\Q[K]$-module.
\end{itemize}
Then we have an isomorphism of $\Q[K]$-modules
$$H_A^*(G,\Q)\cong H_A^*(Q,\Q)\otimes H_A^*(N,\Q).$$
\end{proposition}

\begin{proof}
Since $Q$ acts trivially on $H^*(N,\Q)$, we have $E_2^{p,q}=E_2^{p,0}\otimes E_2^{0,q}$ for any $p,q$.
Since $E_2^{p,q}=0$ for any $q\ge 3$ and the differentials $d_2^{0,1}$, $d_2^{0,2}$, $d_3^{0,2}$ are zero maps,
by the multiplicative structure of the cohomological Hochschild--Serre spectral sequence, we have $E_2^{p,q}=E_{\infty}^{p,q}$ and $\wti{E}_2^{p,q}=\wti{E}_{\infty}^{p,q}$ for any $p,q$.
Therefore, we have 
$\im f_{\infty}^{p,q}=\im f_2^{p,q}=H_A^p(Q,\Q)\otimes H_A^q(N,\Q)$ for any $p,q$.
Since the exact sequence \eqref{exactseqsplit} splits as $\Z[K]$-modules, we have $H^*(G^{\ab},\Q)\cong H^*(Q^{\ab},\Q)\otimes H^*(N^{\ab},\Q)$ as $\Q[K]$-modules.
By the semi-simplicity, it follows that $H_A^*(G,\Q)\cong H_A^*(Q,\Q)\otimes H_A^*(N,\Q)$ as $\Q[K]$-modules.
\end{proof}

\section{The upper bounds of the third Albanese homology of $\IA_n$ and $\IO_n$}\label{AppendixB}

Here we give the computation of upper bounds of the third Albanese homology of $\IA_n$ and $\IO_n$, which we postponed in Section \ref{ThirdhomologyIAandIO}.

By the following commutative diagram
   \begin{gather*}
   \xymatrix{
    H^2(U^O,\Q) \ar@{^{(}->}[d]_-{\pr^*}\ar[r]_-{\tau_O^*}
    &
    H^2(\IO_n,\Q)\ar[d]_-{\pi^*}
      \\
    H^2(U,\Q)\ar[r]^-{\tau^*}
    &
    H^2(\IA_n,\Q),
   }
   \end{gather*}
   we obtain 
   \begin{gather}\label{kerIOkertau}
       R^O_2=\ker \tau^*_O\subset \ker \tau^*=R_2.
   \end{gather}

Pettet \cite{Pettet} computed the $\GL(n,\Q)$-representation structure of $R_2$ and $R^O_2$.

    \begin{lemma}[{Pettet \cite[Theorem 1.2]{Pettet}}]\label{Pettetker2}
   For $n\ge 3$, we have 
   $$R_2\cong V_{1, 2 1}\oplus V_{0, 1^2},\quad
   R^O_2\cong V_{1, 2 1}.$$
   \end{lemma}

\subsection{Generators of $R_2$}

We will explicitly write down the generator of each irreducible subrepresentation of $R_2\subset H^2(U,\Q)$.
Let $\{e_{c}^{a,b}\mid 1\le a,b,c\le n,\; a< b\}$ denote the basis for $U^*=\Hom(U,\Q)$ dual to the basis $\{e_{a,b}^{c}\mid 1\le a,b,c\le n,\; a< b\}$ for $U$.
Then $e_{c}^{a,b}$ is identified with $e_c\otimes (e_a^* \wedge e_b^*)\in  H\otimes \bigwedge^2 (H^*)$ via the isomorphism $U^*\cong H\otimes \bigwedge^2 (H^*)$.
Note that $(U^O)^*\subset U^*$ is generated by 
$$\{e_{c}^{a,b}\mid 1\le a,b,c\le n,\; a< b,\; a\ne c\ne b\}\cup \{e_{b}^{a,b}- \frac{1}{n-1} \sum_{j=1}^n e_{j}^{a,j}\mid  1\le a,b \le n,\; a\ne b \}.$$
  
   \begin{lemma}\label{PettetIO}
    For $n\ge 3$, 
    \begin{enumerate}
    \item the subrepresentation $R^O_2\subset H^2(U^O,\Q)$ is generated by 
    $$\beta=\sum_{j=1}^n e_{j}^{1,2}\wedge e_{n}^{j,1}\in H^2(U^O,\Q)\subset \bigwedge^2 (U^*).$$
      
    \item the subrepresentation of $R_2$ that is isomorphic to $V_{0,1^2}$ is generated by 
    \begin{gather*}
    \begin{split}
    \gamma=
    \sum_{j,k =1}^n e_{j}^{1,2}\wedge e_{k}^{j,k}
    +\sum_{j,k =1}^n & e_{j}^{1,k}\wedge e_{k}^{j,2}\in H^2(U,\Q)\cong \bigwedge^2 (U^*).
    \end{split}
    \end{gather*}
    \end{enumerate}
   \end{lemma}
   
   \begin{proof}
   (1) 
   We have
   \begin{gather*}
   \begin{split}
       \beta&=e_{2}^{1,2}\wedge e_{n}^{2,1}+e_{n}^{1,2}\wedge e_{n}^{n,1} +\sum_{j=3}^{n-1} e_{j}^{1,2}\wedge e_{n}^{j,1}\\
       &=(e_{2}^{1,2}-\frac{1}{n-1}\sum_{j=1}^{n}e_{j}^{1,j})\wedge e_{n}^{2,1}+e_{n}^{1,2}\wedge (e_{n}^{n,1}-\frac{1}{n-1}\sum_{j=1}^{n}e_{j}^{j,1})+\sum_{j=3}^{n-1} e_{j}^{1,2}\wedge e_{n}^{j,1},
   \end{split}
   \end{gather*}
   which is an element of $\bigwedge^2 (U^O)^*$.
   By Lemma \ref{hwv}, we can check that $\beta$ generates an irreducible representation which is isomorphic to $V_{1,21}$.
   We have 
   $$H^2(U^O,\Q)\cong \bigwedge^2 (V_{1,1^2})\cong V_{1^2,2^2}\oplus V_{2,21^2}\oplus V_{1^2,1^4}\oplus V_{1,21}\oplus V_{1,1^3}\oplus V_{0,1^2}.$$
   Since the multiplicity of $V_{1,21}$ in $H^2(U^O,\Q)$ is $1$, it follows from Lemma \ref{Pettetker2} that $\beta$ generates $R^O_2\subset H^2(U^O,\Q)$.

   (2)
   For a group $G$, let $G^{(i)}$ denote the $i$-th term of the lower central series of $G$, and $\opegr^i (G)=G^{(i)}/G^{(i+1)}$ the degree $i$ part of the associated graded Lie algebra.
   By \cite[Lemma 2.1]{Pettet}, we have the following exact sequence of $\GL(n,\Q)$-representations
   \begin{gather}\label{pettetexact}
       0\to \Hom(\opegr^2(\IA_n),\Q)\xrightarrow{[\;,\;]^*}\Hom(\bigwedge^2 \IA_n^{\ab} ,\Q)\;(\cong H^2(U,\Q))\xrightarrow{\tau^*}H^2(\IA_n,\Q),
   \end{gather}
   where $[\;,\;]^*$ is induced by the surjective map
   \begin{gather*}
     [\;,\;]:  \bigwedge^2 \IA_n^{\ab} \twoheadrightarrow \opegr^2(\IA_n)
   \end{gather*}
   defined by 
   \begin{gather*}
       [\;,\;](\tau(f)\wedge \tau(g))=[f,g]\IA_n^{(3)}
   \end{gather*}
   for $f,g\in \IA_n$.
   By \eqref{pettetexact}, we have
   \begin{gather*}
       \im [\;,\;]^*=\ker \tau^*=R_2.
   \end{gather*}
   
   We will use the \emph{second Johnson homomorphism}, which is an injective morphism of $\GL(n,\Z)$-representations
   \begin{gather*}
       \tau^{(2)}: \opegr^2(\IA_n)\to \Hom(\opegr^1(F_n), \opegr^3(F_n))
   \end{gather*}
   defined by
   \begin{gather*}
       \tau^{(2)}(f \IA_n^{(3)})(x F_n^{(2)})= f(x)x^{-1} F_n^{(4)}
   \end{gather*}
   for $f\in \IA_n, x\in F_n$.
   It is well known that
   \begin{gather*}
       \opegr^1(F_n)\cong H_{\Z}, \quad \opegr^3(F_n)\cong (\bigwedge^2 H_{\Z}\otimes H_{\Z})/\bigwedge^3 H_{\Z},
   \end{gather*}
   where $\bigwedge^3 H_{\Z}\subset \bigwedge^2 H_{\Z}\otimes H_{\Z}$ is given by $a\wedge b \wedge c\mapsto (a\wedge b)\otimes c+(b\wedge c)\otimes a+(c\wedge a)\otimes b$.
    Then the composition map
  \begin{gather*}
      \bigwedge^2 \IA_n^{\ab} \xrightarrow{[\;,\;]} \opegr^2(\IA_n)\xrightarrow{\tau^{(2)}}  \Hom(H_{\Z}, (\bigwedge^2 H_{\Z}\otimes H_{\Z})/\bigwedge^3 H_{\Z})
  \end{gather*}
  sends
  \begin{gather*}
  \begin{split}
      e_{a,b}^{c}\wedge e_{p,q}^{r}&\mapsto \delta_{c,p} \; e_r^*\otimes (\overline{(e_a\wedge e_b)\otimes e_q})
      -\delta_{c,q} \; e_r^*\otimes (\overline{(e_a\wedge e_b)\otimes e_p})\\
      &\quad  -\delta_{r,a} \; e_c^*\otimes (\overline{(e_p\wedge e_q)\otimes e_b})
      +\delta_{r,b} \; e_c^*\otimes (\overline{(e_p\wedge e_q)\otimes e_a})
  \end{split}
  \end{gather*}
  for distinct $a,b$ and distinct $p,q$.
  
  Note that we have 
   \begin{gather*}
       \Hom(H^*\otimes ((\bigwedge^2 H\otimes H)/\bigwedge^3 H),\Q)\cong V_{1,21}\oplus V_{0,1^2}\oplus V_{0,2}.
   \end{gather*}
   Let $\{e_{a\wedge b,c}^{d}=e_d^*\otimes (e_a\wedge e_b) \otimes e_c\mid 1\le a,b,c,d\le n, a<b\}$ be a basis for $H^*\otimes \bigwedge^2 H\otimes H$ and $\{(e_{a\wedge b,c}^{d})^*\}$ the dual basis.
   Via the inclusion map
   \begin{gather*}
       \Hom(H^*\otimes ((\bigwedge^2 H\otimes H)/\bigwedge^3 H),\Q)\hookrightarrow \Hom(H^*\otimes\bigwedge^2 H\otimes H,\Q),
   \end{gather*}
   induced by the projection $\bigwedge^2 H\otimes H\twoheadrightarrow (\bigwedge^2 H\otimes H)/\bigwedge^3 H$, we will describe elements of $\Hom(H^*\otimes ((\bigwedge^2 H\otimes H)/\bigwedge^3 H),\Q)$.
   Then we can check that
   \begin{gather*}
       \gamma'=\frac{2}{3}\sum_{k=1}^{n}(e_{1\wedge 2, k}^{k})^*-\frac{1}{3}\sum_{k=1}^{n}(e_{2\wedge k, 1}^{k})^*-\frac{1}{3}\sum_{k=1}^{n}(e_{k\wedge 1, 2}^{k})^*
   \end{gather*}
   is a generator of the subrepresentation of $\Hom(H^*\otimes ((\bigwedge^2 H\otimes H)/\bigwedge^3 H),\Q)$ that is isomorphic to $V_{0,1^2}$.

   We have
   \begin{gather*}
       \begin{split}
           \gamma'\tau^{(2)}[\;,\;](e_{1,2}^{j}\wedge e_{j,k}^{k})
           &=\frac{2}{3}+\frac{2}{3}\delta_{k,1}(1+\delta_{j,2})+\frac{2}{3}\delta_{k,2}(1+\delta_{j,1})\quad (1\le j\neq k\le n),\\
           \gamma'\tau^{(2)}[\;,\;](e_{1,k}^{j}\wedge e_{j,2}^{k})
           &=\frac{2}{3}+\frac{2}{3}\delta_{k,1}+\frac{2}{3}\delta_{k,2}\quad (1\le j,k\le n,\; j\neq 2,\;k\neq 1)
       \end{split}
   \end{gather*}
   and $\gamma'\tau^{(2)}[\;,\;](e_{a,b}^{c}\wedge e_{p,q}^{r})$ vanishes for any $e_{a,b}^{c}\wedge e_{p,q}^{r}$ other than the above elements (up to sign).
   Therefore, we have 
   \begin{gather*}
        (\tau^{(2)}[\;,\;])^*(\gamma')=\frac{2}{3}\gamma,
   \end{gather*}
    which implies that $\gamma$ generates the subrepresentation of $R_2=\im[\;,\;]^*$ that is isomorphic to $V_{0,1^2}$.
    This completes the proof.
   \end{proof}

\subsection{The image of the cup product map for $\IO_n$}

   Let 
   \begin{gather*}
       \cup^O : H^1(U^O,\Q)\otimes R^O_2\to H^3(U^O,\Q)
   \end{gather*}
   denote the restriction of the cup product for $H^*(U^O,\Q)$.
   The tensor product $H^1(U^O,\Q)\otimes R^O_2$ is decomposed into irreducible $\GL(n,\Q)$-representations as follows.
   
   \begin{lemma}
   Let $n\ge 6$.
   We have 
   \begin{gather*}
    \begin{split}
       H^1(U^O,\Q)\otimes R^O_2&\cong V_{1, 1^2}\otimes V_{1, 2 1}\\
       & \cong V_{2, 3 2}
        \oplus V_{1^2, 3 2}
        \oplus V_{2, 3 1^2}
        \oplus V_{1^2, 3 1^2}
        \oplus V_{2, 2^2 1}
        \oplus V_{1^2, 2^2 1}
        \oplus V_{2, 2 1^3}
        \oplus V_{1^2, 2 1^3}\\
        &\quad\oplus V_{1, 3 1}^{\oplus 2}
        \oplus V_{1, 2^2}^{\oplus 2}
        \oplus V_{1, 2 1^2}^{\oplus 3}
        \oplus V_{1, 1^4}
        \oplus V_{0, 3}
        \oplus V_{0, 2 1}^{\oplus 2}
        \oplus V_{0, 1^3}.
    \end{split}
    \end{gather*}
   \end{lemma}

   In order to detect irreducible components of $\im\cup^O$, we will use contraction maps defined as follows.
   Let 
   \begin{gather*}
       \iota^*: H^3(U^O,\Q)\cong \bigwedge^3 (U^O)^*\hookrightarrow (H \otimes (H^*)^{\otimes 2})^{\otimes 3}
   \end{gather*}
   denote the composition of the inclusion $\bigwedge^3 (U^O)^*\hookrightarrow \bigwedge^3 U^*$ and the injection $\bigwedge^3 U^*\hookrightarrow (H \otimes (H^*)^{\otimes 2})^{\otimes 3}$ that is obtained from $\iota_3$ defined in Section \ref{subseccontraction} by replacing $H$ with $H^*$.
   Define contraction maps
   \begin{gather*}
   \begin{split}
    \phi: (H \otimes (H^*)^{\otimes 2})^{\otimes 3} &\to H^{\otimes 2} \otimes (H^*)^{\otimes 5},\\
    \bigotimes_{j=1}^3 (c_j\otimes a_j^*\otimes b_j^*)&\mapsto a_3^*(c_2) (c_1\otimes c_3)\otimes (a_1^*\otimes b_1^* \otimes a_2^*\otimes b_2^*\otimes b_3^*),\\
    \varphi_1: (H \otimes (H^*)^{\otimes 2})^{\otimes 3} &\to (H^*)^{\otimes 3},\\
    \bigotimes_{j=1}^3 (c_j\otimes a_j^*\otimes b_j^*)&\mapsto a_2^*(c_1) a_3^*(c_2) a_1^*(c_3) b_1^*\otimes b_2^* \otimes b_3^*,\\
    \varphi_2: (H \otimes (H^*)^{\otimes 2})^{\otimes 3} &\to (H^*)^{\otimes 3},\\
    \bigotimes_{j=1}^3 (c_j\otimes a_j^*\otimes b_j^*)&\mapsto b_3^*(c_1) a_3^*(c_2) a_1^*(c_3) b_1^*\otimes a_2^* \otimes b_2^*,\\
    \psi_1: (H \otimes (H^*)^{\otimes 2})^{\otimes 3} &\to H\otimes (H^*)^{\otimes 4},\\
    \quad\bigotimes_{j=1}^3 (c_j\otimes a_j^*\otimes b_j^*)&\mapsto a_2^*(c_1) a_3^*(c_2) c_3\otimes (a_1^*\otimes b_1^*\otimes b_2^* \otimes b_3^*),\\
    \psi_2: (H \otimes (H^*)^{\otimes 2})^{\otimes 3} &\to H\otimes (H^*)^{\otimes 4},\\
    \bigotimes_{j=1}^3 (c_j\otimes a_j^*\otimes b_j^*)&\mapsto a_3^*(c_2) b_3^*(c_1) c_3\otimes (a_1^*\otimes b_1^*\otimes a_2^* \otimes b_2^*).
   \end{split}
   \end{gather*}
   Then we obtain the following lower bound of $\im\cup^O$.
   
  \begin{proposition}\label{kerIO}
   For $n\ge 6$, the image of $\;\cup^O$ contains a $\GL(n,\Q)$-subrepresentation which is isomorphic to the direct sum of the following $17$ irreducible representations:
    \begin{gather*}
    \begin{split}
      & V_{2, 3 2}
        \oplus V_{1^2, 3 2}
        \oplus V_{2, 3 1^2}
        \oplus V_{1^2, 3 1^2}
        \oplus V_{2, 2^2 1}
        \oplus V_{1^2, 2^2 1}
        \oplus V_{2, 2 1^3}
        \oplus V_{1^2, 2 1^3}\\
        &\quad\oplus V_{1, 3 1}
        \oplus V_{1, 2^2}^{\oplus 2}
        \oplus V_{1, 2 1^2}^{\oplus 2}
        \oplus V_{1, 1^4}
        \oplus V_{0, 3}
        \oplus V_{0, 2 1}
        \oplus V_{0, 1^3}.
    \end{split}
    \end{gather*}
   \end{proposition}
   
   \begin{proof}
    Let $n\ge 6$.
    For distinct elements $a,b,c\in [n]$, let 
    \begin{gather*}
    \beta^O_{a,b,c}:=e_{a}^{b,c}\wedge\beta=
    e_{a}^{b,c}
    \wedge
    \sum_{j=1}^n(
    e_{j}^{1,2}\wedge e_{n}^{j,1}),\\
    \beta^O_{a,b}:= \left(e_{b}^{a,b}- \frac{1}{n-1}\sum_{k=1}^n e_{k}^{a,k}\right) \wedge \beta=
    \left(e_{b}^{a,b}- \frac{1}{n-1}\sum_{k=1}^n e_{k}^{a,k}\right) 
    \wedge
     \sum_{j=1}^n
    (e_{j}^{1,2}\wedge e_{n}^{j,1}),
    \end{gather*}
    which are elements of $\im \cup^O$ by Lemma \ref{PettetIO}.

    We will use $\beta^O_{5,3,4}$ to detect $8$ irreducible components of size $7$ in the statement.
    We have
    \begin{gather*}
    \begin{split}
        \phi\iota^*(\beta^O_{5,3,4})
        &=(n-1)(e_5\otimes e_n)\otimes (e_3^*\wedge e_4^*)\otimes (e_1^*\wedge e_2^*)\otimes e_1^*\\
        &\quad-(e_5\otimes e_n)\otimes (e_1^*\wedge e_2^*)\otimes (e_3^*\wedge e_4^*)\otimes e_1^*\\
        &\quad-(e_n\otimes e_5)\otimes (e_3^*\wedge e_1^*)\otimes (e_1^*\wedge e_2^*)\otimes e_4^*\\
        &\quad+(e_n\otimes e_5)\otimes (e_4^*\wedge e_1^*)\otimes (e_1^*\wedge e_2^*)\otimes e_3^*,
    \end{split}
    \end{gather*}
    where $e_a^*\wedge e_b^*=e_a^*\otimes e_b^*-e_b^*\otimes e_a^*\in (H^*)^{\otimes 2}$.
    It follows that $\im \cup^O$ includes a subrepresentation which is isomorphic to $V_{1^2,32}$ since 
    \begin{gather*}
        ([\id]-[P_{5,n}])([\id]-[E_{4,2}])([\id]-[E_{3,1}])\in \Q(\GL(n,\Q))
    \end{gather*}
    maps $\phi\iota^*(\beta^O_{5,3,4})$ to $(n-1)(e_5\otimes e_n- e_n\otimes e_5)\otimes (e_1^*\wedge e_2^*)\otimes (e_1^*\wedge e_2^*)\otimes e_1^*$, which is a generator of $V_{1^2,32}$.
    In a similar way, one can check that $\beta^O_{5,3,4}$ detects all of the irreducible components of size $7$.
  
   We will use $\beta^O_{1,n,3}$ to detect $3$ irreducible components of size $3$ in the statement. 
   We have
   \begin{gather*}
   \begin{split}
       \varphi_1 \iota^*(\beta^O_{1,n,3})
       &=(n-1) (e_1^*\otimes e_2^*\otimes e_3^*+e_2^*\otimes e_3^*\otimes e_1^*+e_3^*\otimes e_1^*\otimes e_2^*),\\
       \varphi_2 \iota^*(\beta^O_{1,n,3})
       &=(n-2) e_3^*\otimes (e_1^*\wedge e_2^*)+e_2^*\otimes (e_3^*\wedge e_1^*),
   \end{split}
   \end{gather*}
   where $e_a^*\wedge e_b^*=e_a^*\otimes e_b^*-e_b^*\otimes e_a^*\in (H^*)^{\otimes 2}$.
   Then $\beta^O_{1,n,3}$ detects all of the irreducible components of size $3$ since we have
   \begin{gather*}
       ([\id]-[E_{2,1}])([\id]-[E_{3,1}]) \varphi_1 \iota^*(\beta^O_{1,n,3})
       =3(n-1) e_1^*\otimes e_1^*\otimes e_1^*,
   \end{gather*}
   which is a generator of $V_{0,3}$,
   \begin{gather*}
       ([\id]-[P_{1,2}])\varphi_1 \iota^*(\beta^O_{1,n,3})
       =(n-1)\sum_{\sigma\in \gpS_3} \sgn(\sigma)e_{\sigma(1)}^*\otimes e_{\sigma(2)}^*\otimes e_{\sigma(3)}^*,
   \end{gather*}
   which is a generator of $V_{0,1^3}$,
   and 
   \begin{gather*}
       ([\id]-[E_{3,1}])\varphi_2 \iota^* (\beta^O_{1,n,3})
       =(n-2)e_1^*\otimes (e_1^*\wedge e_2^*),
   \end{gather*}
    which is a generator of $V_{0,21}$.
   
   We will use $\beta^O_{1,3,4}$ and $\beta^O_{3,1}$ to detect $4$ irreducible representations of size $5$ with multiplicities.
   We have
   \begin{gather*}
   \begin{split}
        \psi_1 \iota^*(\beta^O_{1,3,4})
        &=(n-1)e_n\otimes (e_3^*\wedge e_4^*)\otimes e_2^*\otimes e_1^*- e_n\otimes (e_1^*\wedge e_2^*)\otimes (e_3^*\wedge e_4^*),\\
        \psi_2 \iota^*(\beta^O_{1,3,4})
        &=(n-1)e_n\otimes \{(e_3^*\wedge e_4^*)\otimes (e_1^*\wedge e_2^*)+ (e_1^*\wedge e_2^*)\otimes (e_3^*\wedge e_4^*)\},\\
        \psi_1 \iota^*(\beta^O_{3,1})
        &=(n-2)e_n\otimes (e_3^*\wedge e_1^*)\otimes e_2^*\otimes e_1^*+ e_n\otimes (e_3^*\wedge e_2^*)\otimes e_1^*\otimes e_1^*\\
        &\quad-\frac{n-2}{n-1}e_n\otimes (e_1^*\wedge e_2^*)\otimes e_3^*\otimes e_1^*+\frac{2n-3}{n-1}e_n\otimes (e_1^*\wedge e_2^*)\otimes e_1^*\otimes e_3^*,\\
        \psi_2 \iota^*(\beta^O_{3,1})
        &=\frac{n^2-3n+3}{n-1}e_n\otimes \{(e_1^*\wedge e_2^*)\otimes(e_3^*\wedge e_1^*)+ (e_3^*\wedge e_1^*)\otimes (e_1^*\wedge e_2^*)\},\\
   \end{split}
   \end{gather*}
   where $e_a^*\wedge e_b^*=e_a^*\otimes e_b^*-e_b^*\otimes e_a^*\in (H^*)^{\otimes 2}$.
   Let
   $$a^O_{1,3,4}=([\id]-[P_{1,2}])([\id]-[E_{4,2}])([\id]-[E_{3,1}])(\beta^O_{1,3,4})$$
   and 
   $$a^O_{3,1}=([\id]-[P_{1,2}])([\id]-[E_{3,2}])(\beta^O_{3,1}).$$
   Then $a^O_{1,3,4}$ and $a^O_{3,1}$ detect $V_{1,2^2}^{\oplus 2}$ since we have
   \begin{gather*}
       \begin{split}
        \psi_1\iota^*(a^O_{1,3,4})&=-(n+1) e_n\otimes (e_1^*\wedge e_2^*)\otimes (e_1^* \wedge e_2^*),\\
        \psi_2\iota^*(a^O_{1,3,4})&=4(n-1) e_n\otimes (e_1^*\wedge e_2^*)\otimes (e_1^* \wedge e_2^*),\\
        \psi_1\iota^*(a^O_{3,1})&=-\frac{n^2-3}{n-1} e_n\otimes (e_1^*\wedge e_2^*)\otimes (e_1^* \wedge e_2^*),\\
        \psi_2\iota^*(a^O_{3,1})&=4\frac{n^2- 3n +3}{n-1}e_n\otimes (e_1^*\wedge e_2^*)\otimes (e_1^* \wedge e_2^*)
    \end{split}
   \end{gather*}
   and since $(-(n+1), 4(n-1))$ and $(-\frac{n^2-3}{n-1},4\frac{n^2- 3n +3}{n-1})$ are linearly independent if $n\ge 6$.
   One can also check that $\psi_1\iota^*(\beta^O_{1,3,4})$ and $\psi_1\iota^*(\beta^O_{3,1})$ detect $V_{1, 2 1^2}^{\oplus 2}$, and that $\psi_1 \iota^*(\beta^O_{1,3,4})$ detects $V_{1, 3 1}$ and $V_{1, 1^4}$.
   This completes the proof.
   \end{proof}

\subsection{The image of the cup product map for $\IA_n$}

   Let 
   \begin{gather*}
       \cup : H^1(U,\Q)\otimes R_2\to H^3(U,\Q)
   \end{gather*}
   denote the restriction of the cup product for $H^*(U,\Q)$.
   The tensor product $H^1(U,\Q)\otimes R_2$ is decomposed into irreducible $\GL(n,\Q)$-representations as follows.

    \begin{lemma}
   Let $n\ge 6$.
   We have 
   \begin{gather*}
    \begin{split}
       H^1(U,\Q)\otimes R_2
       & \cong (V_{1, 1^2} \oplus V_{0,1})\otimes (V_{1, 2 1}\oplus V_{0, 1^2})\\
       & \cong V_{2, 3 2}
        \oplus V_{1^2, 3 2}
        \oplus V_{2, 3 1^2}
        \oplus V_{1^2, 3 1^2}
        \oplus V_{2, 2^2 1}
        \oplus V_{1^2, 2^2 1}
        \oplus V_{2, 2 1^3}
        \oplus V_{1^2, 2 1^3}\\
       &\quad \oplus V_{1, 3 1}^{\oplus 3}
        \oplus V_{1, 2^2}^{\oplus 4}
        \oplus V_{1, 2 1^2}^{\oplus 5}
        \oplus V_{1, 1^4}^{\oplus 2}
        \oplus V_{0, 3}
        \oplus V_{0, 2 1}^{\oplus 5}
        \oplus V_{0, 1^3}^{\oplus 3}.
    \end{split}
    \end{gather*}
   \end{lemma}
   
   Note that it follows from \eqref{kerIOkertau} that
   \begin{gather}\label{imcupcupO}
       \im\cup^O\subset \im \cup.
   \end{gather}
   In order to detect irreducible components of $\im\cup$, we will use the following maps.
   Let 
   \begin{gather*}
       \iota^*: H^3(U,\Q)\cong \bigwedge^3 (U^*)\hookrightarrow (H \otimes (H^*)^{\otimes 2})^{\otimes 3}
   \end{gather*}
   denote the injection that is obtained from $\iota_3$ defined in Section \ref{subseccontraction} by replacing $H$ with $H^*$.
   Define contraction maps
   \begin{gather*}
   \begin{split}
     \varphi_3: (H \otimes (H^*)^{\otimes 2})^{\otimes 3} &\to (H^*)^{\otimes 3},\\
     \bigotimes_{j=1}^3 (c_j\otimes a_j^*\otimes b_j^*)&\mapsto a_2^*(c_1) a_3^*(c_2) b_3^*(c_3) a_1^*\otimes b_1^* \otimes b_2^*,\\
     \varphi_4: (H \otimes (H^*)^{\otimes 2})^{\otimes 3} &\to (H^*)^{\otimes 3},\\
     \bigotimes_{j=1}^3 (c_j\otimes a_j^*\otimes b_j^*)&\mapsto b_1^*(c_1) a_3^*(c_2) b_3^*(c_3) a_1^*\otimes a_2^* \otimes b_2^*,\\
     \psi_3: (H \otimes (H^*)^{\otimes 2})^{\otimes 3} &\to H\otimes (H^*)^{\otimes 4}, \\
     \bigotimes_{j=1}^3 (c_j\otimes a_j^*\otimes b_j^*)&\mapsto a_3^*(c_2) b_3^*(c_3) c_1\otimes (a_1^*\otimes b_1^*\otimes a_2^* \otimes b_2^*),\\
     \psi_4: (H \otimes (H^*)^{\otimes 2})^{\otimes 3} &\to H\otimes (H^*)^{\otimes 4}, \\
     \bigotimes_{j=1}^3 (c_j\otimes a_j^*\otimes b_j^*)&\mapsto b_1^*(c_1) a_3^*(c_2) c_3\otimes (a_1^*\otimes a_2^*\otimes b_2^* \otimes b_3^*).
   \end{split}
   \end{gather*}
   Then we obtain the following lower bound of $\im\cup$.

   \begin{proposition}\label{imagecup}
   For $n\ge 6$, the image of $\;\cup$ contains a $\GL(n,\Q)$-subrepresentation which is isomorphic to the direct sum of the following $27$ irreducible representations:
    \begin{gather*}
    \begin{split}
      & V_{2, 3 2}
        \oplus V_{1^2, 3 2}
        \oplus V_{2, 3 1^2}
        \oplus V_{1^2, 3 1^2}
        \oplus V_{2, 2^2 1}
        \oplus V_{1^2, 2^2 1}
        \oplus V_{2, 2 1^3}
        \oplus V_{1^2, 2 1^3}\\
       &\quad
        \oplus V_{1, 3 1}^{\oplus 2}
        \oplus V_{1, 2^2}^{\oplus 4}
        \oplus V_{1, 2 1^2}^{\oplus 4}
        \oplus V_{1, 1^4}^{\oplus 2}
        \oplus V_{0, 3}
        \oplus V_{0, 2 1}^{\oplus 3}
        \oplus V_{0, 1^3}^{\oplus 3}.
    \end{split}
    \end{gather*}
   \end{proposition}
   
   \begin{proof}
   By \eqref{imcupcupO} and Proposition \ref{kerIO}, the $8$ irreducible components of size $7$ and $V_{0,3}$ are included in $\im\cup$. 
   
   Let
   \begin{gather*}
   \begin{split}
        \beta_{a,b,c}&= e_{a}^{b,c} \wedge \beta
        =e_{a}^{b,c} \wedge \sum_{j=1}^n(e_{j}^{1,2}\wedge e_{n}^{j,1}),\\
        \gamma_{a,b,c}&=e_{a}^{b,c} \wedge \gamma
        =e_{a}^{b,c} \wedge(\sum_{j,k =1}^n e_{j}^{1,2}\wedge e_{k}^{j,k} +\sum_{j,k =1}^n  e_{j}^{1,k}\wedge e_{k}^{j,2}),
   \end{split}
   \end{gather*}
   which are elements of $\im\cup$ by Lemma \ref{PettetIO}.
   
   We will use $\beta_{1,n,3}$, $\gamma_{1,3,4}$ and $\gamma_{3,4,3}$ to detect $V_{0, 2 1}^{\oplus 3}\oplus V_{0, 1^3}^{\oplus 3}$ in $\im\cup$.
   We have
   \begin{gather*}
       \begin{split}
        \varphi_2 \iota^* (\beta_{1,n,3})&=(n-2) e_3^*\otimes (e_1^*\wedge e_2^*)+e_2^*\otimes (e_3^*\wedge e_1^*),\\
        \varphi_3 \iota^*(\beta_{1,n,3})&=0,\\
        \varphi_4 \iota^*(\beta_{1,n,3})&=0,\\
        \varphi_2 \iota^* (\gamma_{1,3,4})&=(n^2-3) e_2^*\otimes (e_3^*\wedge e_4^*)-(n-1) e_3^*\otimes (e_4^*\wedge e_2^*)-(n-1) e_4^*\otimes (e_2^*\wedge e_3^*),\\
        \varphi_3 \iota^* (\gamma_{1,3,4})&=(n^2+n)(e_3^*\wedge e_4^*)\otimes e_2^*,\\
        \varphi_4 \iota^* (\gamma_{1,3,4})&=(n+1)e_2^*\otimes (e_3^*\wedge e_4^*),\\
        \varphi_2 \iota^* (\gamma_{3,4,3})&=-e_1^*\otimes (e_2^*\wedge e_4^*)-e_2^*\otimes (e_4^*\wedge e_1^*)-(n-3)e_4^*\otimes (e_1^*\wedge e_2^*),\\
        \varphi_3 \iota^* (\gamma_{3,4,3})&=2n(e_1^*\wedge e_2^*)\otimes e_4^*-n(e_2^*\wedge e_4^*)\otimes e_1^*-n(e_4^*\wedge e_1^*)\otimes e_2^*,\\
        \varphi_4 \iota^* (\gamma_{3,4,3})&=(n^2-1)e_4^*\otimes (e_1^*\wedge e_2^*),\\
       \end{split}
   \end{gather*}
   where $e_a^*\wedge e_b^*=e_a^*\otimes e_b^*-e_b^*\otimes e_a^*\in (H^*)^{\otimes 2}$.
   Therefore, in a way similar to the argument in the proof of Proposition \ref{kerIO}, we can see that $\im\cup$ contains each of $V_{0,1^3}$ and $V_{0,2 1}$ with multiplicity $3$. 
  
  One can also check that $12$ irreducible components of size $5$ are contained in $\im\cup$ by using contraction maps $\psi_1,\psi_2,\psi_3,\psi_4$ and the elements $\beta_{a,b,c}$ and $\gamma_{a,b,c}$.
  For example, one can detect $V_{1,31}^{\oplus 2}$ by using $\psi_i\iota^*(\beta_{1,3,4})$ and $\psi_i\iota^*(\beta_{3,4,3})$ for $i=1,4$,
  and $V_{1, 1^4}^{\oplus 2}$ by using $\psi_i\iota^* (\beta_{1,3,4})$ and $\psi_i\iota^*(\gamma_{5,3,4})$ for $i=1,3$,
  and $V_{1, 2^2}^{\oplus 4}$ by using $\psi_i\iota^*(\beta_{1,3,4})$, $\psi_i\iota^*(\beta_{2,3,n})$, $\psi_i\iota^*(\beta_{3,4,3})$ and $\psi_i\iota^*(\gamma_{3,4,3})$ for $1\le i \le 4$,
  and $V_{1, 21^2}^{\oplus 4}$ by using $\psi_i\iota^*(\beta_{1,3,4})$, $\psi_i\iota^*(\beta_{3,4,3})$ and $\psi_i\iota^*(\gamma_{3,4,3})$ for $i=1,3,4$.
  This completes the proof.
   \end{proof}


\begin{thebibliography}{10}

\bibitem{Bachmuth}
S. Bachmuth.
\newblock Induced automorphisms of free groups and free metabelian groups.
\newblock {\em Trans. Amer. Math. Soc.}, 122:1--17, 1966.

\bibitem{BCHLLS}
Georgia Benkart, Manish Chakrabarti, Thomas Halverson, Robert Leduc, Chanyoung
  Lee, and Jeffrey Stroomer.
\newblock Tensor product representations of general linear groups and their
  connections with {B}rauer algebras.
\newblock {\em J. Algebra}, 166(3):529--567, 1994.

\bibitem{BBM}
Mladen Bestvina, Kai-Uwe Bux, and Dan Margalit.
\newblock Dimension of the {T}orelli group for {${\rm Out}(F_n)$}.
\newblock {\em Invent. Math.}, 170(1):1--32, 2007.

\bibitem{Brown}
Kenneth~S. Brown.
\newblock {\em Cohomology of groups}, volume~87 of {\em Graduate Texts in
  Mathematics}.
\newblock Springer-Verlag, New York-Berlin, 1982.

\bibitem{CEF}
Thomas Church, Jordan~S. Ellenberg, and Benson Farb.
\newblock F{I}-modules and stability for representations of symmetric groups.
\newblock {\em Duke Math. J.}, 164(9):1833--1910, 2015.

\bibitem{CFAbelJacobi}
Thomas Church and Benson Farb.
\newblock Parameterized {A}bel-{J}acobi maps and abelian cycles in the
  {T}orelli group.
\newblock {\em J. Topol.}, 5(1):15--38, 2012.

\bibitem{CFrep}
Thomas Church and Benson Farb.
\newblock Representation theory and homological stability.
\newblock {\em Adv. Math.}, 245:250--314, 2013.

\bibitem{CFP}
Pierre~J Clavier, Lo{\"\i}c Foissy, and Sylvie Paycha.
\newblock From non-unitary wheeled {PROP}s to smooth amplitudes and generalised
  convolutions.
\newblock {\em arXiv preprint arXiv:2103.00855}, 2021.

\bibitem{CP}
F.~Cohen and J~Pakianathan.
\newblock On automorphism groups of free groups, and their nilpotent quotients.
\newblock in preparation.

\bibitem{CHP}
F.~R. Cohen, Aaron Heap, and Alexandra Pettet.
\newblock On the {A}ndreadakis-{J}ohnson filtration of the automorphism group
  of a free group.
\newblock {\em J. Algebra}, 329:72--91, 2011.

\bibitem{DP}
Matthew Day and Andrew Putman.
\newblock On the second homology group of the {T}orelli subgroup of {${\rm
  Aut}(F_n)$}.
\newblock {\em Geom. Topol.}, 21(5):2851--2896, 2017.

\bibitem{Djament19}
Aur\'{e}lien Djament.
\newblock D\'{e}composition de {H}odge pour l'homologie stable des groupes
  d'automorphismes des groupes libres.
\newblock {\em Compos. Math.}, 155(9):1794--1844, 2019.

\bibitem{Djament-Vespa}
Aur\'{e}lien Djament and Christine Vespa.
\newblock Sur l'homologie des groupes d'automorphismes des groupes libres \`a
  coefficients polynomiaux.
\newblock {\em Comment. Math. Helv.}, 90(1):33--58, 2015.

\bibitem{Farb}
Benson Farb.
\newblock Automorphisms of {$F_n$} which act trivially on homology.
\newblock in preparation.

\bibitem{Formanek}
Edward Formanek.
\newblock Characterizing a free group in its automorphism group.
\newblock {\em J. Algebra}, 133(2):424--432, 1990.

\bibitem{FH}
William Fulton and Joe Harris.
\newblock {\em Representation theory}, volume 129 of {\em Graduate Texts in
  Mathematics}.
\newblock Springer-Verlag, New York, 1991.
\newblock A first course, Readings in Mathematics.

\bibitem{GG}
Stavros Garoufalidis and Ezra Getzler.
\newblock Graph complexes and the symplectic character of the {T}orelli group.
\newblock {\em arXiv preprint arXiv:1712.03606}, 2017.


\bibitem{Hain}
Richard Hain.
\newblock Infinitesimal presentations of the {T}orelli groups.
\newblock {\em J. Amer. Math. Soc.}, 10(3):597--651, 1997.

\bibitem{Hainsurvey}
Richard Hain.
\newblock Johnson homomorphisms.
\newblock {\em EMS Surv. Math. Sci.}, 7(1):33--116, 2020.

\bibitem{Johnsonsurvey}
Dennis Johnson.
\newblock A survey of the {T}orelli group.
\newblock In {\em Low-dimensional topology ({S}an {F}rancisco, {C}alif.,
  1981)}, volume~20 of {\em Contemp. Math.}, pages 165--179. Amer. Math. Soc.,
  Providence, RI, 1983.

\bibitem{Johnson}
Dennis Johnson.
\newblock The structure of the {T}orelli group. {III}. {T}he abelianization of
  {$\mathcal T$}.
\newblock {\em Topology}, 24(2):127--144, 1985.

\bibitem{Kawazumi}
Nariya Kawazumi.
\newblock Cohomological aspects of {M}agnus expansions.
\newblock {\em arXiv preprint math/0505497}, 2005.

\bibitem{Kawazumibraid}
Nariya Kawazumi.
\newblock Twisted {M}orita-{M}umford classes on braid groups.
\newblock {\em Groups, homotopy and configuration spaces}, 13:293--306, 2008.


\bibitem{Kawazumi-Morita}
Nariya Kawazumi and Shigeyuki Morita.
\newblock The primary approximation to the cohomology of the moduli space of
  curves and cocycles for the {M}umford-{M}orita-{M}iller classes.
\newblock preprint.


\bibitem{Kawazumi--Vespa}
Nariya Kawazumi and Christine Vespa.
\newblock On the wheeled {PROP} of stable cohomology of {${\rm Aut}(F_n)$} with
  bivariant coefficients.
\newblock {\em Algebraic Geom. Topol.}, to appear.

\bibitem{Koike}
Kazuhiko Koike.
\newblock On the decomposition of tensor products of the representations of the
  classical groups: by means of the universal characters.
\newblock {\em Adv. Math.}, 74(1):57--86, 1989.

\bibitem{KM}
Sava Krsti\'{c} and James McCool.
\newblock The non-finite presentability of {${\rm IA}(F_3)$} and {${\rm
  GL}_2({\bf Z}[t,t^{-1}])$}.
\newblock {\em Invent. Math.}, 129(3):595--606, 1997.

\bibitem{KRW}
Alexander Kupers and Oscar Randal-Williams.
\newblock On the cohomology of {T}orelli groups.
\newblock {\em Forum Math. Pi}, 8:e7, 83, 2020.

\bibitem{KRW21}
Alexander Kupers and Oscar Randal-Williams.
\newblock On the {T}orelli {L}ie algebra.
\newblock {\em arXiv preprint arXiv:2106.16010}, 2021.

\bibitem{Lindell}
Erik Lindell.
\newblock Abelian cycles in the homology of the {T}orelli group.
\newblock {\em J. Inst. Math. Jussieu}, published online, 2021.

\bibitem{LodayVallette}
Jean-Louis Loday and Bruno Vallette.
\newblock {\em Algebraic operads}, volume 346 of {\em Grundlehren der
  mathematischen Wissenschaften [Fundamental Principles of Mathematical
  Sciences]}.
\newblock Springer, Heidelberg, 2012.

\bibitem{Magnus}
Wilhelm Magnus.
\newblock \"{U}ber {$n$}-dimensionale {G}ittertransformationen.
\newblock {\em Acta Math.}, 64(1):353--367, 1935.

\bibitem{MMS}
M.~Markl, S.~Merkulov, and S.~Shadrin.
\newblock Wheeled {PROP}s, graph complexes and the master equation.
\newblock {\em J. Pure Appl. Algebra}, 213(4):496--535, 2009.

\bibitem{Markl}
Martin Markl.
\newblock Operads and {PROP}s.
\newblock In {\em Handbook of algebra. {V}ol. 5}, volume~5 of {\em Handb.
  Algebr.}, pages 87--140. Elsevier/North-Holland, Amsterdam, 2008.

\bibitem{Morita1999}
Shigeyuki Morita.
\newblock Structure of the mapping class groups of surfaces: a survey and a
  prospect.
\newblock In {\em Proceedings of the {K}irbyfest ({B}erkeley, {CA}, 1998)},
  volume~2 of {\em Geom. Topol. Monogr.}, pages 349--406. Geom. Topol. Publ.,
  Coventry, 1999.

\bibitem{Patzt}
Peter Patzt.
\newblock Representation stability for filtrations of {T}orelli groups.
\newblock {\em Math. Ann.}, 372(1-2):257--298, 2018.

\bibitem{Pettet}
Alexandra Pettet.
\newblock The {J}ohnson homomorphism and the second cohomology of {${\rm IA}_n$}.
\newblock {\em Algebr. Geom. Topol.}, 5:725--740, 2005.

\bibitem{Randal-Williams}
Oscar Randal-Williams.
\newblock Cohomology of automorphism groups of free groups with twisted
  coefficients.
\newblock {\em Selecta Math. (N.S.)}, 24(2):1453--1478, 2018.

\bibitem{Sakasai}
Takuya Sakasai.
\newblock The {J}ohnson homomorphism and the third rational cohomology group of
  the {T}orelli group.
\newblock {\em Topology Appl.}, 148(1-3):83--111, 2005.

\bibitem{SamSnowden}
Steven~V. Sam and Andrew Snowden.
\newblock Stability patterns in representation theory.
\newblock {\em Forum Math. Sigma}, 3:Paper No. e11, 108, 2015.

\bibitem{Satoh2006}
Takao Satoh.
\newblock Twisted first homology groups of the automorphism group of a free
  group.
\newblock {\em J. Pure Appl. Algebra}, 204(2):334--348, 2006.

\bibitem{Satoh2007}
Takao Satoh.
\newblock Twisted second homology groups of the automorphism group of a free
  group.
\newblock {\em J. Pure Appl. Algebra}, 211(2):547--565, 2007.

\bibitem{Satoh2021}
Takao Satoh.
\newblock On the low-dimensional cohomology groups of the {IA}-automorphism
  group of the free group of rank three.
\newblock {\em Proc. Edinb. Math. Soc. (2)}, 64(2):338--363, 2021.

\bibitem{Vespa}
Christine Vespa.
\newblock Extensions between functors from free groups.
\newblock {\em Bull. Lond. Math. Soc.}, 50(3):401--419, 2018.


\bibitem{Weibel}
Charles~A. Weibel.
\newblock {\em An introduction to homological algebra}, volume~38 of {\em
  Cambridge Studies in Advanced Mathematics}.
\newblock Cambridge University Press, Cambridge, 1994.

\end{thebibliography}
\end{document}